\documentclass{article}
\usepackage{graphicx}
\usepackage[english]{babel}
\usepackage{caption}
\usepackage{subcaption}
\usepackage{amsmath, amsfonts, amssymb, amsthm, bm}  
\usepackage{mathrsfs}
\usepackage{url}
\usepackage{hyperref}

\theoremstyle{plain}
\newtheorem{theorem}{Theorem}[section]

\newtheorem{lemma}[theorem]{Lemma}

\newtheorem{remark}[theorem]{Remark}
\newtheorem{definition}[theorem]{Definition}

\theoremstyle{definition}
\newtheorem{example}{Example}

\newcommand{\bydef}{\stackrel{\textnormal{\tiny def}}{=}}

\begin{document}
\title{
Numerical computation of transverse homoclinic orbits \\
for periodic solutions of delay differential equations
}
\author{
Olivier H\'{e}not
\thanks
{McGill University, Department of Mathematics and Statistics, 805 Sherbrooke Street West, Montreal, QC, H3A 0B9, Canada. {\tt olivier.henot@mail.mcgill.ca}.}
\and Jean-Philippe Lessard
\thanks
{McGill University, Department of Mathematics and Statistics, 805 Sherbrooke Street West, Montreal, QC, H3A 0B9, Canada. {\tt jp.lessard@mcgill.ca}.}
\and Jason D. Mireles James
\thanks
{Florida Atlantic University, Department of Mathematical Sciences, Science Building, Room 234, 777 Glades Road, Boca Raton, Florida, 33431, USA. {\tt jmirelesjames@fau.edu}.}
}

\date{}

\maketitle

% \tableofcontents

\begin{abstract}
We present a computational method for studying transverse homoclinic orbits for periodic solutions of delay differential equations, a phenomenon that we refer to as the \emph{Poincar\'{e} scenario}.
The strategy is geometric in nature, and consists of viewing the connection as the zero of a nonlinear map, such that the invertibility of its Fr\'{e}chet derivative implies the transversality of the intersection. The map is defined by a projected boundary value problem (BVP), with boundary conditions in the (finite dimensional) unstable and (infinite dimensional) stable manifolds of the periodic orbit.
The parameterization method is used to compute the unstable manifold and the BVP is solved using a discrete time dynamical system approach (defined via the \emph{method of steps}) and Chebyshev series expansions.
We illustrate this technique by computing transverse homoclinic orbits in the cubic Ikeda and Mackey-Glass systems.
\end{abstract}

\begin{center}
\end{center}

%%%%%%%%%%%%%%%%%%
%% INTRODUCTION %%
%%%%%%%%%%%%%%%%%%

\section{Introduction}\label{sec:introduction}
%!TEX root = poincare_scenario_dde.tex

A delay differential equation (DDE) relates the rate of change of a function with its state at present and past times. 
They are used, for example, to model networks with communication lags between subsystems, and particle systems where disturbances propagate with finite speed.
The delay gives a DDE a kind of \emph{memory}, and leads to the notion of an infinite dimensional dynamical system.
Thanks to this high dimensionality, even a scalar DDE can exhibit diverse and complex dynamics.
We refer the interested reader to the books \cite{MR1345150,Hale1977,MR1243878} 
on the subject of DDEs.

A notable example is the delayed-feedback model for the concentration of blood cells introduced in 1977 by Mackey and Glass \cite{Mackey1977}
\begin{equation}\label{eq:mg}
\frac{d}{dt} w(t) = - a w(t) + b \frac{w(t - \tau)}{1 + w(t - \tau)^\rho}, \qquad t \ge 0.
\end{equation}
Here $\tau > 0$ is the constant delay and $a, b, \rho \in \mathbb{R}$ are physiological parameters.
The authors introduce the notion of \emph{dynamical disease}, where pathological behaviors are produced by control systems after variation of the physiological parameters.
Since qualitative changes in the dynamics characterize the onset of symptoms, this notion ties dynamical bifurcation theory to disease pathology.
More sophisticated models of hematopoiesis, extending this concept, are found in the works of \cite{Mackey1987,PujoMenjouet2016,Souza2019}.

\begin{figure}
\centering
\begin{subfigure}[b]{0.41\textwidth}
\includegraphics[width=\textwidth]{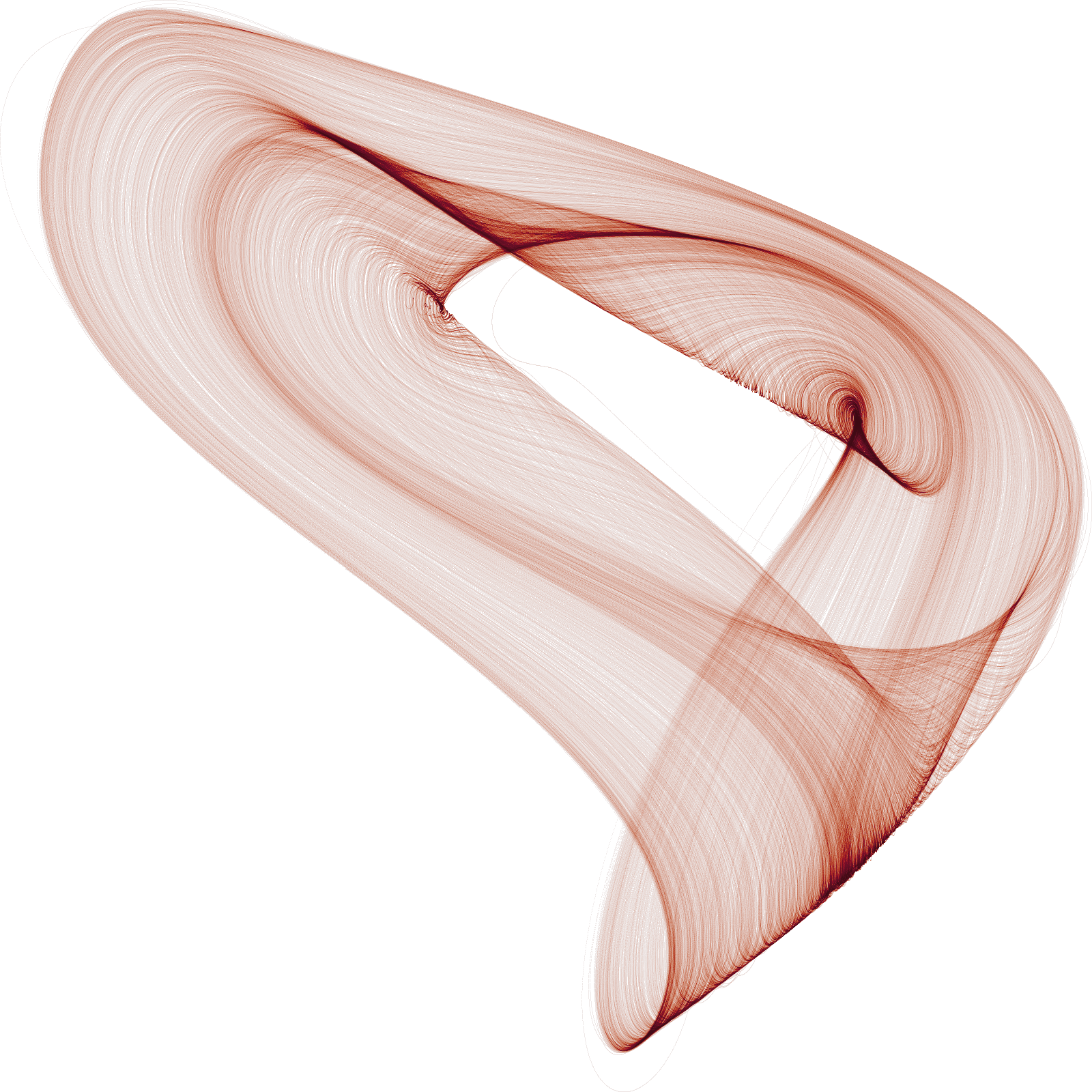}
\caption{Mackey-Glass equation.}
\label{fig:mg_attractor}
\end{subfigure}
\hspace{0.06\textwidth}
\begin{subfigure}[b]{0.41\textwidth}
\includegraphics[width=\textwidth]{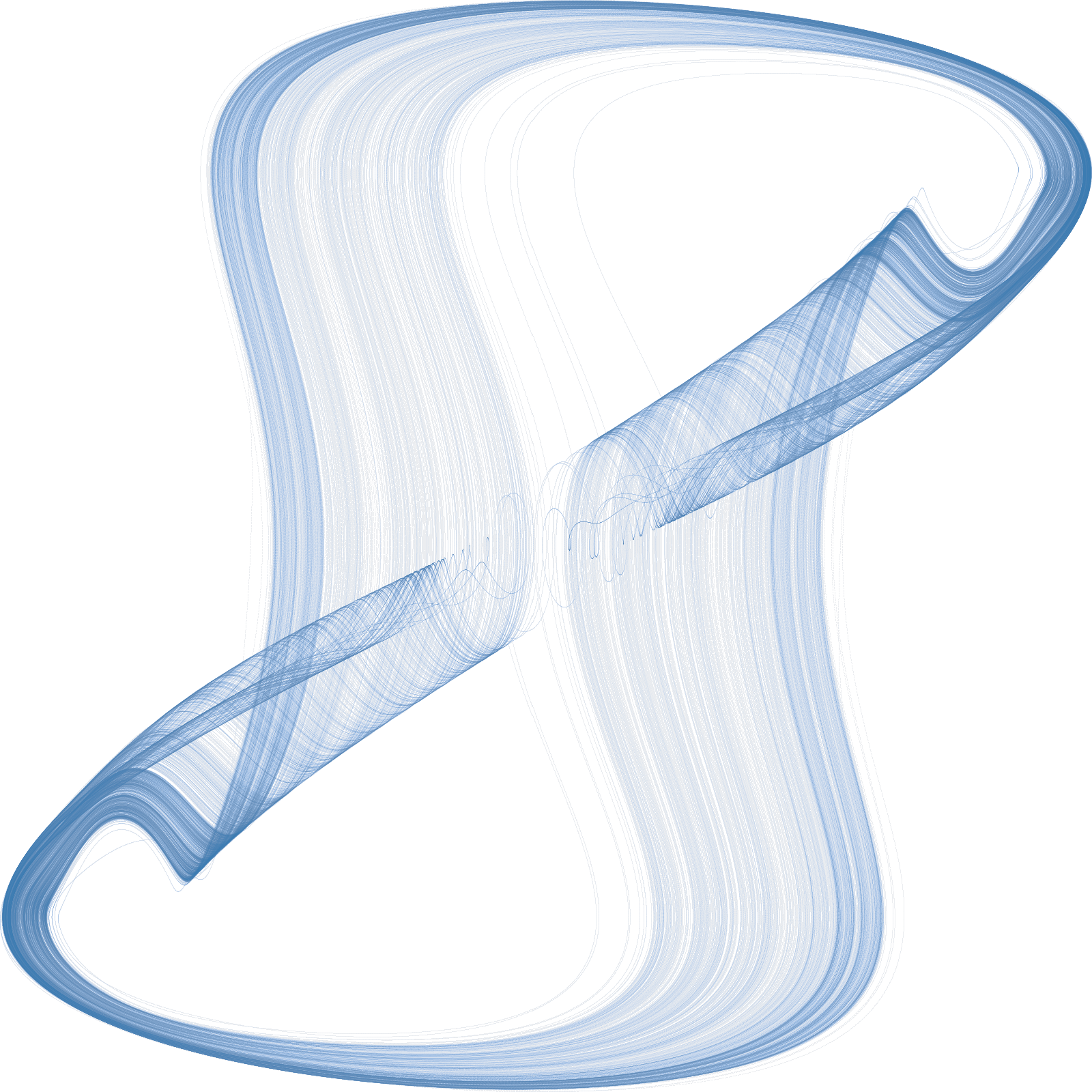}
\caption{Cubic Ikeda equation.}
\label{fig:cubic_ikeda_attractor}
\end{subfigure}
\caption{Numerical simulations providing evidence for the existence of chaotic attractors in some simple DDEs.}
\label{fig:attractors}
\end{figure}

In addition to its impact on pathology, Equation \eqref{eq:mg}, nowadays known as the \emph{Mackey-Glass equation}, is famous for its rich dynamics.
Indeed, the Mackey-Glass equation has become a flagship example of chaos in infinite dimensions.
In the original article \cite{Mackey1977}, Mackey and Glass numerically followed a sequence of period doubling bifurcations by increasing the delay $\tau$, which led to the onset of chaotic behavior.
Figure \ref{fig:mg_attractor} shows numerical simulation results for Equation \eqref{eq:mg} which suggest the existence of a chaotic attractor.
Further investigations were conducted by Farmer, Mensour and Longtin to compute the Lyapunov exponents, Lyapunov dimension and power spectra \cite{Farmer1982, Mensour1998}.
The interested reader is referred to \cite{Hale1988, Ikeda1987, Junges2012, Lepri1994, PujoMenjouet2016} and references therein for more information. 

In the present work we consider a geometric mechanism which 
gives rise to chaotic behavior.
The picture goes back to Poincar\'{e}'s groundbreaking work on the three-body problem, where he showed that homoclinic orbits associated with periodic solutions lead to extremely rich dynamics \cite{Poincare}.
More precisely, when the unstable manifold \emph{bends back} and returns to a neighborhood of the periodic orbit, it can intersect the stable manifold transversely; a phenomenon we refer to as the \emph{Poincar\'{e} scenario}.
Poincar\'{e} famously complained that the resulting picture
was \emph{difficult to draw}.
In modern language, the transverse intersection implies the existence of chaotic motions (symbolic dynamics) via Smale's Tangle Theorem \cite{MR228014}.
While a number of authors have shown existence of chaotic dynamics for DDEs (e.g. see \cite{Hale1986,Hale1988,MR1615995,MR1827805,MR1329849,MR1397673,MR623379} and the references therein), a rigorous proof of chaos in the Mackey-Glass equation remains an important conjecture in the field (see \cite{hans-otto-MG} for a more thorough discussion of this conjecture).

The aim of this article is to present a numerical method for studying the Poincar\'{e} scenario (transverse homoclinic orbits) in DDEs of the form
\begin{equation}\label{eq:dde_elem}
\frac{d}{dt} w(t) = g ( w(t), w_t(-\tau) ), \qquad t \ge 0,
\end{equation}
where $\tau > 0$ is the delay, $w_t (s) \bydef w(t + s)$, for all $s \in [-\tau, 0]$, and $g : \mathbb{R} \times \mathbb{R} \to \mathbb{R}$ is assumed to only be comprised of elementary nonlinearities (i.e. exponential, logarithmic, algebraic functions and compositions thereof).
Our strategy consists of rephrasing the intersection of the stable and unstable manifolds as an isolated zero of a nonlinear map.
This choice is motivated by current techniques for computer-assisted proofs, where zeros of infinite dimensional maps are proven to exist via contraction mapping arguments applied to appropriate fixed-point operators.
The reader will recognize here a Newton-Kantorovich type argument, and may refer to some related works \cite{Llave2016,Henot2021,Lessard2020,Lessard2021,Berg2022}.

An important feature of the proposed framework is that we do not exploit any numerical integration schemes for advecting the flow generated by the DDE.
Instead, we are careful to express the problem in a form so that, after truncation, we are left to solve large systems of polynomial equations; in particular, we formulate the \emph{method of steps} (e.g. see \cite{Diekmann1995}) as a $C^1$ Chebyshev integrator (see \cite{Lessard2021}) which amounts to solving polynomial equations.
While we do not give any computer-assisted proofs in the present work, the article has the ulterior motive of paving the way for future mathematically rigorous studies of chaos for DDEs.
In particular, we are convinced that the present research provides a sufficient framework for proving the existence of symbolic dynamics in the Mackey-Glass equation.

\subsection{Framework}

We begin by noting that for any DDE \eqref{eq:dde_elem} there exists an auxiliary 
polynomial DDE of the form
\begin{equation}\label{eq:dde}
\frac{d}{dt} u(t) = f(u(t), u_t(-\tau)), \qquad t \ge 0,
\end{equation}
where $\tau > 0$ is the delay, $u_t (s) \bydef u(t + s)$, for all $s \in [-\tau, 0]$, and $f : \mathbb{R}^n \times \mathbb{R}^n \to \mathbb{R}^n$ is a polynomial.
Indeed, if the right-hand-side of the DDE \eqref{eq:dde_elem} is polynomial, then the DDEs \eqref{eq:dde_elem} and \eqref{eq:dde} are identical with $u = w$, $f = g$ and $n = 1$.
Otherwise, we introduce new coordinates $v_t = ((v_1)_t, \dots, (v_d)_t)$ in place of an appropriate set of elementary nonlinear functions $\varphi( w_t )$ (cf. Theorem 2.2 in \cite{Henot2021}).
Then, $u = (w, v)$ satisfies a DDE of the form \eqref{eq:dde} with $n = 1 + d$ and
\begin{equation}\label{eq:f_poly}
f(u(t), u_t(\tau))
=
\begin{pmatrix}
g^{(1)} (u(t), u_t(-\tau)) \\
g^{(2)} (u(t), u_t(-\tau))
\end{pmatrix},
\end{equation}
where $g^{(1)}, g^{(2)}$ are polynomials, with range in $\mathbb{R}, \mathbb{R}^d$ respectively, such that
\[
g^{(1)} (u(t), u_t(-\tau)) = g ( w(t), w_t(-\tau) ) \qquad \text{and} \qquad
g^{(2)} (u(t), u_t(-\tau)) = [D\varphi(w_t) \frac{d}{dt} w_t](0),
\]
whenever $v_t = \varphi( w_t )$. In fact, for initial conditions satisfying $v_0 = \varphi( w_0 )$, it follows that $w(t)$ is a solution of the DDE \eqref{eq:dde_elem}.

The idea of enlarging the dimension of the original system to a larger polynomial system is variously referred to as \emph{automatic differentiation}, \emph{polynomial embedding}, or \emph{quadratic recast} (e.g. see \cite{Bucker2006,MR2531684,MR3973675,Henot2021,Jorba2005,Knuth1981,Lessard2016}).

\begin{remark}
The reader may be surprised by our insistence on working with a polynomial DDE \eqref{eq:dde}. 
Certainly, for numerical computations, which is the scope of this article, this may seem like an awkward annoyance as it is entirely possible to directly use the DDE \eqref{eq:dde_elem} with elementary nonlinearities; the results of this article will follow straightforwardly, albeit handling Chebyshev and Taylor expansions of elementary nonlinearities.
On the other hand, generating the polynomial DDE \eqref{eq:dde} is easily done (e.g. see \cite{Henot2021}) and is not an innocent decision.
From a numerical perspective, multiplication is a natural operation for Chebyshev series expansions: their interpretation as cosine series 
facilitates multiplication via discrete convolutions.
From a theoretical perspective, the Banach algebra structure enjoyed by Chebyshev series is most easily exploited in computer-assisted proofs when the nonlinearities are polynomial.
The reader will see in our strategy a flexible numerical technique for which computer-assisted proofs techniques can be applied as easily as possible.
\end{remark}

\begin{example}[Cubic Ikeda equation]
The Ikeda equation
\[
\frac{d}{dt} w(t) = \sin (w(t - \tau))
\]
was introduced in \cite{Ikeda1987}.
This simple DDE also displays a chaotic attractor and is often found as a \emph{sister} equation to the Mackey-Glass equation \eqref{eq:mg} in the literature.
In the present article, we will not consider the full sine nonlinearity.
Indeed, it has been thoroughly explored and we rather emphasize the Mackey-Glass equation.
In \cite{Sprott2007}, Sprott discusses how a low-order rescaled Taylor expansion of the sine nonlinearity can be considered while still retaining complex dynamics; the resulting DDE is the so-called \emph{cubic Ikeda equation} given by
\begin{equation}\label{eq:cubic_ikeda}
\frac{d}{dt} w(t) = w_t (-\tau) - w_t (-\tau)^3.
\end{equation}
We will use this cubic scalar DDE as an illustrative and intuitive template to guide the reader through the forthcoming complex notions of the article; its (numerically) chaotic attractor is shown on Figure \ref{fig:cubic_ikeda_attractor}.
Then, $u = w$ satisfies the DDE of the form \eqref{eq:dde} with $n = 1$ and
\begin{equation}\label{eq:f_cubic_ikeda}
f(u(t), u_t(-\tau)) = u_t (-\tau) - u_t (-\tau)^3.
\end{equation}
\end{example}

\begin{example}[Mackey-Glass equation]
Consider the Mackey-Glass equation \eqref{eq:mg}.
We define $\varphi_1 (w_t) \bydef w_t(1 + w_t^\rho)^{-1}$, $\varphi_2(w_t) \bydef w_t^{\rho-2}$ and $\varphi_3(w_t) \bydef w_t^{-1}$. Then, $u = (w, v_1, v_2, v_3)$ satisfies the DDE of the form \eqref{eq:dde} with $n = 1 + d = 4$ and
\begin{equation}\label{eq:f_mg}
f(u(t), u_t(\tau))
=
\begin{pmatrix}
\hspace{3.85cm} - \hspace{0.075cm} a w(t) + b \cdot (v_1)_t(-\tau)  \\
v_1(t)(v_3(t) - \rho v_1(t) v_2(t)) \big(- a w(t) + b \cdot (v_1)_t(-\tau)\big) \\
\hspace{1.3cm} (\rho-2) v_2(t) v_3(t) \big(- a w(t) + b \cdot (v_1)_t(-\tau)\big) \\
\hspace{2.65cm} - v_3(t)^2 \big(- a w(t) + b \cdot (v_1)_t(-\tau)\big)
\end{pmatrix}.
\end{equation}
A similar polynomial system for the Mackey-Glass equation was first presented in \cite{Berg2022} where the authors prove the existence of periodic orbits in DDEs. Their method is based on Fourier series expansions for which the resulting Banach algebra structure is, again, most easily exploited with polynomial nonlinearities.
Incidentally, note that, in the process, the equilibrium $0$ of the Mackey-Glass equation \eqref{eq:mg} has become a singular point. This should bear no impact in the present context since the numerically observed chaotic dynamics remain bounded away from $0$.
\end{example}

Now, since $f$ is locally Lipschitz, there exists $t_* > 0$ such that, for all $t \in [0,  t_*)$, the solution operator of the DDE \eqref{eq:dde} is a strongly continuous semi-flow $S_t : C([-\tau,0],\mathbb{R}^n) \to C([-\tau,0],\mathbb{R}^n)$ defined by
\[
[S_t (\phi)] (s) \bydef
\begin{cases}
\displaystyle \phi(0) + \int_{0}^{t + s} f([S_{t'}(\phi)](0), [S_{t'}(\phi)](-\tau)) \, dt', & t + s > 0, \\
\phi(t + s), & t + s \le 0,
\end{cases}
\]
for all $s \in [-\tau,0]$.
Following the strategy presented in \cite{Lessard2021}, the solution operator induces a discrete dynamical system (DDS) by considering the \emph{time-$\tau$ map} representing the forward integration of fixed step-size $\tau$, namely
\begin{equation}\label{eq:dds}
\begin{cases}
\phi \mapsto \mathcal{F} (\phi), \\
\phi \in \mathcal{C}^n,
\end{cases}
\end{equation}
where $\mathcal{C}^n \bydef \big\{ \phi \in C([-\tau,0],\mathbb{R}^n) \, : \, S_t(\phi) \textnormal{ exists for all } t \in [0,  \tau] \big\}$ and
\begin{equation}\label{eq:def_F}
[\mathcal{F} (\phi)] (s)
\bydef [S_{ \tau} (\phi)](s)
= \phi(0) + \int_{-\tau}^s f( [\mathcal{F} ( \phi )] (s'), \phi(s') ) \, ds'.
\end{equation}
Hence, $\mathcal{F} ( \phi )$, for a given $\phi \in \mathcal{C}^n$, is implicitly defined as the unique solution of
\[
u (s) = [\mathcal{T}( u, \phi )](s), \qquad \text{for all } s \in [-\tau,0],
\]
where
\begin{equation}\label{eq:time_stepping_map}
[\mathcal{T}( u, \phi )](s) \bydef \phi(0) + \int_{-\tau}^s f( u(s'), \phi(s') ) \, ds', \qquad \text{for all } s \in [-\tau,0].
\end{equation}

The DDS \eqref{eq:dds} yields a discretization of the DDE \eqref{eq:dde} and corresponds to the formalism behind the numerical scheme to solve DDEs known as the \emph{method of steps} (e.g. see \cite{Diekmann1995}).
The following lemma summarizes the correspondence between the solutions of the DDE \eqref{eq:dde} and the solutions of the DDS \eqref{eq:dds}.

\begin{lemma}\label{lem:equivalence_dde_dds}
Let $\tau > 0$, $m \in \mathbb{N}$, and $\phi \in C([-\tau, 0], \mathbb{R}^n)$. The following statements are equivalent:
\begin{enumerate}
\item $t \in [0, m\tau] \mapsto S_t (\phi)$ is a solution of the DDE \eqref{eq:dde}.
\item $j \in \{0, \dots, m-1\} \mapsto \mathcal{F}^j (\phi)$ is a solution of the DDS \eqref{eq:dds}.
\end{enumerate}
\end{lemma}

\begin{proof}
By construction, the existence of a solution $t \in [0, m\tau] \mapsto S_t ( \phi )$ of the DDE \eqref{eq:dde} is equivalent to $\phi = S_0 ( \phi ), S_\tau( \phi ), \dots, S_{(m-1)\tau} ( \phi ) \in \mathcal{C}^n$ such that
\[
S_{j\tau} ( \phi ) = S_\tau ( S_{(j-1)\tau} ( \phi ) ) = \mathcal{F} ( S_{(j-1)\tau} ( \phi ) ) = \ldots = \mathcal{F}^j (\phi), \qquad j = 0, \dots, m-1.
\]
In other words, $j \in \{0, \dots, m-1\} \mapsto \mathcal{F}^j (\phi)$ is a solution of the DDS \eqref{eq:dds}.
\end{proof}

The essence of the method presented in this article is to pursue the transverse homoclinic orbit of a $m\tau$-periodic solution of the DDE \eqref{eq:dde_elem} with the DDS \eqref{eq:dds}, whose dynamics are given by a compact operator: the time-$\tau$ map $\mathcal{F}$.

\subsection{Structure of the article}

In Section \ref{sec:periodic_orbit}, we present a zero-finding problem to compute periodic orbits of the DDE \eqref{eq:dde_elem}. In Section \ref{sec:eigendecomposition}, we investigate the spectrum and eigenvectors. Then, in Section \ref{sec:unstable_manifold}, we present the computation of the unstable manifold. In Section \ref{sec:homoclinic_orbit}, we combine all the ingredients to formulate a BVP as a zero-finding problem yielding a transverse connecting orbit; the scheme guarantees the transversality through the invertibility of the derivative of the map.
We also illustrate the strategy in each section with the cubic Ikeda equation.
Lastly, in Section \ref{sec:mackey_glass}, we apply our method to compute a transverse homoclinic orbit for the Mackey-Glass equation.

\medskip

The code implementing the method presented in this article can be found at \cite{Henot2023}.
The code relies on \emph{RadiiPolynomial} \cite{Henot2021-2}, a library -- written in Julia \cite{Bezanson2017} -- for computer-assisted proofs in dynamical systems.
We make no attempts to perform rigorous numerics in this article, yet the library provides useful resources to easily implement the method presented in this article.
Lastly, to visualize the data we use \emph{Makie} \cite{Danisch2021}.

%%%%%%%%%%%%%%%%%%%%
%% PERIODIC ORBIT %%
%%%%%%%%%%%%%%%%%%%%

\section{Computation of the periodic orbit}\label{sec:periodic_orbit}
%!TEX root = poincare_scenario_dde.tex

Let $m \in \mathbb{N}$. A $m\tau$-periodic orbit of \eqref{eq:dde} corresponds to a $m$-periodic orbit of \eqref{eq:dds}; that is, a fixed-point of the mapping $\phi \mapsto \mathcal{F}^m (\phi)$ ($\mathcal{F}$ composed with itself $m$ times).
Since working directly with $m$ compositions of $\mathcal{F}$ is laborious, we prefer \emph{unrolling} $\mathcal{F}^m$ at the cost of working with a DDS comprised of more equations.

Thus, we consider the following multiple shooting scheme for the DDS
\begin{equation}\label{eq:unrolled_dds}
\begin{cases}
\displaystyle
\phi \mapsto \mathring{\mathcal{F}}( \phi )
\bydef
\begin{pmatrix}
\mathcal{F}( \phi_m )\\
\mathcal{F}( \phi_1 )\\
\vdots\\
\mathcal{F}( \phi_{m-1} )
\end{pmatrix},\\
\phi \in (\mathcal{C}^n)^m,
\end{cases}
\end{equation}
such that $\mathring{\mathcal{F}}( \phi )$, for a given $\phi \in (\mathcal{C}^n)^m$, is the unique solution of
\begin{equation}\label{eq:Fcirc_identity}
[\mathring{\mathcal{F}}( \phi )] (s) = [\mathring{\mathcal{T}}( \mathring{\mathcal{F}}( \phi ), \phi )](s), \qquad \text{for all } s \in [-\tau,0],
\end{equation}
where
\begin{equation}
\mathring{\mathcal{T}}( u, \phi )
\bydef
\begin{pmatrix}
\mathcal{T}( u_1, \phi_m ) \\
\mathcal{T}( u_2, \phi_1 ) \\
\vdots \\
\mathcal{T}( u_m, \phi_{m-1} )
\end{pmatrix}.
\end{equation}

The following lemma summarizes the correspondence between periodic orbits of the DDE \eqref{eq:dde}, periodic orbits of the DDS \eqref{eq:dds} and fixed-points of the DDS \eqref{eq:unrolled_dds}.

\begin{lemma}\label{lem:equivalence_po_fp}
Let $\tau > 0$, $m \in \mathbb{N}$ and $\phi \in C([-\tau, 0], \mathbb{R}^n)$. The following statements are equivalent:
\begin{enumerate}
\item $t \in \mathbb{R}/m\tau\mathbb{Z} \mapsto S_t (\phi)$ is an $m\tau$-periodic orbit of the DDE \eqref{eq:dde}.
\item $j \in \mathbb{Z}/m\mathbb{Z} \mapsto \mathcal{F}^j(\phi)$ is an $m$-periodic orbit of the DDS \eqref{eq:dds}.
\item $(\phi, \mathcal{F}(\phi), \dots, \mathcal{F}^{m-1}(\phi))$ is a fixed-point of the DDS \eqref{eq:unrolled_dds}.
\end{enumerate}
\end{lemma}

\begin{proof}
The fact that Point 1 and Point 2 are equivalent follows immediately from Lemma \ref{lem:equivalence_dde_dds}. Moreover, Point 2 means that $\phi, \mathcal{F} (\phi), \dots, \mathcal{F}^{m-1} (\phi) \in \mathcal{C}^n$ such that $\mathcal{F} (\mathcal{F}^{m-1}(\phi)) = \phi$. By construction, this is equivalent to
\[
\mathring{\mathcal{F}} (\phi, \mathcal{F} (\phi), \dots, \mathcal{F}^{m-1} (\phi))
= \begin{pmatrix}
\mathcal{F} (\mathcal{F}^{m-1}(\phi)) \\
\mathcal{F} (\phi) \\
\vdots \\
\mathcal{F} (\mathcal{F}^{m-2}(\phi))
\end{pmatrix}
= \begin{pmatrix}
\phi \\
\mathcal{F} (\phi) \\
\vdots \\
\mathcal{F}^{m-1} (\phi)
\end{pmatrix}. \qedhere
\]
\end{proof}

Hence, Lemma \ref{lem:equivalence_po_fp} states that the computation of an $m\tau$-periodic orbit of the DDE \eqref{eq:dde} amounts to the computation of a fixed-point $c = (c_1, \dots, c_m) \in (C([-\tau, 0], \mathbb{R}^n))^m$ of the DDS \eqref{eq:unrolled_dds}; namely,
\begin{equation}\label{eq:fp_fun}
\begin{cases}
\displaystyle c_1(s) = c_m (0) + \int_{-\tau}^s f( c_1(s'), c_m(s') ) \, ds', \\
\displaystyle c_j(s) = c_{j-1} (0) + \int_{-\tau}^s f( c_j(s'), c_{j-1}(s') ) \, ds' , & j = 2, \dots, m,
\end{cases}\qquad \text{for all } s \in [-\tau, 0].
\end{equation}
Recall that the right-hand-side of the DDE \eqref{eq:dde} is polynomial, thereby guaranteeing that its periodic solutions are analytic (e.g. see \cite{Nussbaum1973}).
A practical basis for analytic functions on $[-\tau, 0]$ are the Chebyshev polynomials of the first kind given by
\begin{equation}\label{eq:def_cheb}
T_\alpha (t) = \cos(\alpha \arccos(t)), \qquad \alpha = 0, 1, 2,\dots \quad \text{and} \quad t \in [-1, 1].
\end{equation}
Thus, we expand $c_1, \dots, c_m$ as the Chebyshev series
\[
c_j (s(t)) = \{\mathbf{c}_j\}_0 + 2 \sum_{\alpha \ge 1} \{\mathbf{c}_j\}_\alpha T_\alpha (t), \quad \text{for all } t \in [-1, 1], \qquad j = 1, \dots, m,
\]
where $s(t) \bydef \frac{\tau}{2}(t-1)$ scales $[-1, 1]$ to $[-\tau, 0]$.
The analyticity of $c_1, \dots, c_m$ implies that there exists $\nu > 1$ such that their sequence of Chebyshev coefficients belongs to
\[
\ell^1_\nu \bydef \left\{ \mathbf{a} \in \mathbb{C}^{\mathbb{N} \cup \{0\}} \, : \, | \mathbf{a} |_{\ell^1_\nu} \bydef | \{\mathbf{a}\}_0 | + 2\sum_{\alpha \ge 1} | \{\mathbf{a}\}_\alpha | \nu^\alpha < \infty \right\}.
\]

\begin{remark}
The reader may wonder why $\ell^1_\nu$ is a sequence space over the complex field $\mathbb{C}$ and not the real field $\mathbb{R}$ since we only care for real solutions of the DDE \eqref{eq:dde_elem}.
For now, it suffices to say that this slight generalization will allow us to handle the case of complex unstable eigenvalues.
\end{remark}

The sequence space $\ell^1_\nu$ is a Banach algebra with the discrete convolution product
\begin{equation}\label{eq:def_Cheb_conv}
\mathbf{a} * \mathbf{b} \bydef \left\{\sum_{\beta \in \mathbb{Z}} \{\mathbf{a}\}_{|\alpha - \beta|} \{\mathbf{b}\}_{|\beta|}\right\}_{\alpha \ge 0}, \qquad \text{for all } \mathbf{a}, \mathbf{b} \in \ell^1_\nu,
\end{equation}
which corresponds to the natural convolution in Fourier space through \eqref{eq:def_cheb}.
It follows that there is a natural mapping, denoted with the same symbol, $f : (\ell^1_\nu)^n \times (\ell^1_\nu)^n \to (\ell^1_\nu)^n$ defined by replacing products of Chebyshev series with the aforementioned convolution product $*$.
There should be no confusion from this abuse of notation since $f$ denotes in both cases the same polynomial where the algebra depends directly on the nature of its arguments.

Then, the system of equations \eqref{eq:fp_fun} is equivalent to
\[
\begin{cases}
\mathbf{c}_1 = \mathbf{E} (\mathbf{c}_m) + \mathbf{S} (\tfrac{\tau}{2}f( \mathbf{c}_1, \mathbf{c}_m )), \\
\mathbf{c}_j = \mathbf{E} (\mathbf{c}_{j-1}) + \mathbf{S} (\tfrac{\tau}{2}f( \mathbf{c}_j, \mathbf{c}_{j-1} )), & j = 2, \dots, m,
\end{cases}
\]
where $\mathbf{E}, \mathbf{S} : (\ell^1_\nu)^n \to (\ell^1_\nu)^n$ represent the evaluation at $1$ and the integral from $-1$ to $s$ respectively. Namely, for all $\mathbf{a} = (\mathbf{a}_1, \dots, \mathbf{a}_n) \in (\ell^1_\nu)^n$ and $i = 1, \dots, n$,
\begin{subequations}
\begin{align}
\begin{split}\label{eq:def_Cheb_eval}
\{(\mathbf{E}(\mathbf{a}))_i\}_\alpha
&\bydef
\begin{cases}
\displaystyle \{\mathbf{a}_i\}_0 + 2\sum_{\beta \ge 1} \{\mathbf{a}_i\}_\beta, & \alpha = 0, \\
0, & \alpha \ge 1,
\end{cases}
\end{split} \\
\begin{split}\label{eq:def_Cheb_integral}
\{(\mathbf{S}(\mathbf{a}))_i\}_\alpha
&\bydef
\begin{cases}
\displaystyle \{\mathbf{a}_i\}_0 - \frac{\{\mathbf{a}_i\}_1}{2} - 2\sum_{\beta \ge 2} \frac{(-1)^\beta \{\mathbf{a}_i\}_\beta}{\beta^2 - 1}, & \alpha = 0, \\
\displaystyle \frac{\{\mathbf{a}_i\}_{\alpha-1} - \{\mathbf{a}_i\}_{\alpha+1}}{2\alpha}, & \alpha \ge 1.
\end{cases}
\end{split}
\end{align}
\end{subequations}

We now formalize our search for a periodic orbit of the DDE \eqref{eq:dde_elem} as a zero-finding problem.
There are two cases to address: either the DDE \eqref{eq:dde} coincides exactly with \eqref{eq:dde_elem}, or \eqref{eq:dde} is an auxiliary polynomial DDE of the original DDE \eqref{eq:dde_elem}.

To start with, suppose that the DDEs \eqref{eq:dde_elem} and \eqref{eq:dde} are identical, so $n = 1$.
Consider the mapping $\mathbf{F}_\circ : \mathbb{R} \times (\ell^1_\nu \cap \mathbb{R}^{\mathbb{N}\cup\{0\}})^m \to \mathbb{R} \times (\ell^1_\nu \cap \mathbb{R}^{\mathbb{N}\cup\{0\}})^m$ defined by
\begin{equation}\label{eq:zero_finding_problem_fp}
\mathbf{F}_\circ (\tau, \mathbf{c})
\bydef
\begin{pmatrix}
\{\mathbf{E}(\mathbf{c}_m)\}_0 - \delta \\
\mathbf{E} (\mathbf{c}_m) + \mathbf{S} (\frac{\tau}{2} f( \mathbf{c}_1, \mathbf{c}_m )) - \mathbf{c}_1 \\
\mathbf{E} (\mathbf{c}_1) + \mathbf{S} (\frac{\tau}{2} f( \mathbf{c}_2, \mathbf{c}_1 )) - \mathbf{c}_2 \\
\vdots \\
\mathbf{E} (\mathbf{c}_{m-1}) + \mathbf{S} (\frac{\tau}{2} f( \mathbf{c}_m, \mathbf{c}_{m-1} )) - \mathbf{c}_m
\end{pmatrix},
\end{equation}
where $\delta \in \mathbb{R}$ is fixed.
If $\mathbf{F}_\circ (\tau, \mathbf{c}) = 0$ with $\tau > 0$, then $\mathbf{c}_1, \dots, \mathbf{c}_m$ are the sequences of Chebyshev coefficients of an $m\tau$-periodic orbit of the DDE \eqref{eq:dde_elem}.
Observe that we impose $\{\mathbf{E}(\mathbf{c}_m)\}_0 - \delta = 0$ to quotient out the temporal translation invariance of the periodic orbit which transpires in the DDS \eqref{eq:unrolled_dds} as a 1-parameter family of fixed-points.
The specific choice of $\delta$ is determined from the numerical observations.
Additionally, this condition is compensated by solving for the value of the delay $\tau$ for which the period is a multiple of the delay.

On the other hand, when the polynomial DDE \eqref{eq:dde} is in fact an auxiliary polynomial DDE (recall the construction of $f$ given in \eqref{eq:f_poly}) of the DDE \eqref{eq:dde_elem} with elementary nonlinearities, a zero of the mapping $\mathbf{F}_\circ$ does not necessarily yield a periodic solution of \eqref{eq:dde_elem}.
Indeed, one must append the extra conditions $v_0 = \varphi(w_0)$ for solutions of \eqref{eq:dde} to coincide with solutions of \eqref{eq:dde_elem}.
According to Lemma 3.2 in \cite{Henot2021}, this requirement can be compensated by introducing \emph{unfolding parameters}.
We note that there is a slight limitation in the current statement of this lemma as the proxy variables $\eta \in \mathbb{R}^d$ cannot compensate the equality $v_0 = \varphi(w_0)$ set on the function space $C([-\tau, 0], \mathbb{R}^d)$.
Nevertheless, the proof of the lemma actually proves the stronger and more useful result that it suffices to impose the equality $v_0 (0) = [\varphi(w_0)](0)$ on $\mathbb{R}^d$.
We report this small modification of the result in the following lemma.

\begin{lemma}\label{lem:Lemma_3_2_revisited}
Consider a DDE \eqref{eq:dde_elem} with elementary nonlinearities and its auxiliary polynomial DDE \eqref{eq:dde}, where $f$ has the form given in \eqref{eq:f_poly}. Let $\eta \in \mathbb{R}^d$ and $t \in \mathbb{R} \to u_t = (w_t, v_t) \in C([-\tau, 0], \mathbb{R}^{1 + d})$ be a periodic solution of
\[
\frac{d}{d t}
u(t)
=
\begin{pmatrix}
g^{(1)}(u(t), u_t(-\tau)) \\
g^{(2)}(u(t), u_t(-\tau)) + \eta
\end{pmatrix}.
\]
If $v (0) = [\varphi(w_0)](0)$, then $t \in \mathbb{R} \to w_t$ is a periodic solution of \eqref{eq:dde_elem}; in other words, $\eta = 0$.
\end{lemma}

\begin{proof}
See the proof of Lemma 3.2 in \cite{Henot2021}.
\end{proof}

For all $(\tau, \eta, \mathbf{c}) \in \mathbb{R} \times \mathbb{R}^d \times ((\ell^1_\nu)^{1+d})^m$ and $\mathbf{c}_j = (\mathbf{c}_j^{(1)}, \mathbf{c}_j^{(2)}) \in (\ell^1_\nu)^{1+d}$ such that $\mathbf{c}_j^{(1)} \in \ell^1_\nu$, $\mathbf{c}_j^{(2)} \in (\ell^1_\nu)^d$ for $j = 1, \dots, m$, consider the mapping $\mathbf{F}_{\circ, \textnormal{elem}} : \mathbb{R} \times \mathbb{R}^d \times ((\ell^1_\nu \cap \mathbb{R}^{\mathbb{N}\cup\{0\}})^{1+d})^m \to \mathbb{R} \times \mathbb{R}^d \times ((\ell^1_\nu \cap \mathbb{R}^{\mathbb{N}\cup\{0\}})^{1+d})^m$ defined by
\begin{equation}\label{eq:zero_finding_problem_fp_poly}
\mathbf{F}_{\circ, \textnormal{elem}} (\tau, \eta, \mathbf{c})
\bydef
\begin{pmatrix}
\{\mathbf{E}(\mathbf{c}_m)\}_0 - \begin{pmatrix}\delta\\ \varphi(\delta) \end{pmatrix}\\
\mathbf{E} (\mathbf{c}_m) +
\mathbf{S} \left(\frac{\tau}{2} f(\mathbf{c}_1, \mathbf{c}_m) +
\begin{pmatrix}0 \\ \boldsymbol{\iota}(\eta))\end{pmatrix} \right)
- \mathbf{c}_1 \\
\mathbf{E} (\mathbf{c}_1) +
\mathbf{S} \left(\frac{\tau}{2} f(\mathbf{c}_2, \mathbf{c}_1) +
\begin{pmatrix}0 \\ \boldsymbol{\iota}(\eta))\end{pmatrix}\right)
- \mathbf{c}_2 \\
\vdots \\
\mathbf{E} (\mathbf{c}_{m-1}) +
\mathbf{S} \left(\frac{\tau}{2} f(\mathbf{c}_m, \mathbf{c}_{m-1}) +
\begin{pmatrix}0 \\ \boldsymbol{\iota}(\eta))\end{pmatrix}\right)
- \mathbf{c}_m
\end{pmatrix},
\end{equation}
where $\varphi(\delta)$ is understood as $[\varphi(\kappa)](0)$ where $\kappa(s) = \delta$ for all $s \in [-\tau,0]$ and $\boldsymbol{\iota} : \mathbb{C}^d \to (\ell^1_\nu)^d$ is the injection
\begin{equation}\label{eq:def_iota}
\{(\boldsymbol{\iota}(\eta))_i\}_\alpha \bydef
\begin{cases}
\eta_i, & \alpha = 0, \\
0, & \alpha \ge 1, 
\end{cases} \quad i = 1, \dots, d, \qquad \text{for all } \eta \in \mathbb{C}^d.
\end{equation}
According to Lemma \ref{lem:Lemma_3_2_revisited}, if $\mathbf{F}_{\circ, \textnormal{elem}} (\tau, \eta, \mathbf{c}) = 0$ with $\tau > 0$, then $\eta = 0$ and $\mathbf{c}_1^{(1)}, \dots, \mathbf{c}_m^{(1)}$ are the sequences of Chebyshev coefficients of an $m\tau$-periodic orbit of the original DDE \eqref{eq:dde_elem}.

\subsection{Numerical considerations}
\label{sec:periodic_orbit_num}

The role of the zero-finding problems $\mathbf{F}_\circ$ \eqref{eq:zero_finding_problem_fp} and $\mathbf{F}_{\circ, \textnormal{elem}}$ \eqref{eq:zero_finding_problem_fp_poly} are to obtain the central object of our Poincar\'{e} scenario: the periodic orbit whose unstable manifold intersect transversely its stable manifold.
From a practical point of view, only finitely many Chebyshev coefficients can be handled by a computer.
So, given an order $N \in \mathbb{N} \cup \{0\}$, we define the truncation operator $\boldsymbol{\pi}^N : \ell^1_\nu \to \ell^1_\nu$ by
\[
\{ \boldsymbol{\pi}^N \mathbf{a} \}_\alpha \bydef
\begin{cases}
\{\mathbf{a}\}_\alpha, & \alpha \le N, \\
0, & \alpha > N,
\end{cases} \qquad \text{for all } \mathbf{a} \in \ell^1_\nu.
\]
This operator extends in a natural fashion to $\mathbb{C}$ by acting as the identity and to cartesian products of $\mathbb{C}$ and $\ell^1_\nu$ by acting component-wise.

The general gist to implement the zero-finding problems $\mathbf{F}_\circ$ and $\mathbf{F}_{\circ, \textnormal{elem}}$ is to store the Chebyshev coefficients as numerical vectors for which one defines the convolution product $*$ \eqref{eq:def_Cheb_conv}, the evaluation operator $\mathbf{E}$ \eqref{eq:def_Cheb_eval} and the integration operator $\mathbf{S}$ \eqref{eq:def_Cheb_integral}.
We rely on the \emph{RadiiPolynomial} library \cite{Henot2021-2} to handle this.

Then, an approximate zero of $\mathbf{F}_\circ$ (resp. $\mathbf{F}_{\circ,\textnormal{elem}}$) is obtained by applying Newton’s method to the finite dimensional problem $\boldsymbol{\pi}^N \mathbf{F}_\circ \boldsymbol{\pi}^N$ (resp. $\boldsymbol{\pi}^N \mathbf{F}_{\circ, \textnormal{elem}} \boldsymbol{\pi}^N$). To be explicit, having an initial guess $\bar{\tau} > 0, \bar{\mathbf{c}} \in \boldsymbol{\pi}^N (\ell^1_\nu \cap \mathbb{R}^{\mathbb{N}\cup\{0\}})^m$ (typically generated from the time series of a numerical integration of the DDE \eqref{eq:dde}), one recursively applies the iterates
\begin{equation}\label{eq:Newton_method_fp} 
\begin{pmatrix}\bar{\tau} \\ \bar{\mathbf{c}}\end{pmatrix}
\mapsto
\begin{pmatrix}\bar{\tau} \\ \bar{\mathbf{c}}\end{pmatrix} - (\boldsymbol{\pi}^N D\mathbf{F}_\circ (\bar{\tau}, \bar{\mathbf{c}}) \boldsymbol{\pi}^N)^{-1} \boldsymbol{\pi}^N \mathbf{F}_\circ (\bar{\tau}, \bar{\mathbf{c}}),
\end{equation}
and similarly for $\mathbf{F}_{\circ, \textnormal{elem}}$.
Concerning the choice of the truncation order $N$, one typically adjusts it depending on the available memory and the tolerance below which one deems the remaining terms negligible; the latter is bound to the precision used for the computations, e.g. machine precision is of order $\sim 10^{-16}$ in double precision.

Let us expand slightly on the memory consumption of the scheme.
We work with truncated sequence spaces $\boldsymbol{\pi}^N (\ell^1_\nu \cap \mathbb{R}^{\mathbb{N}\cup\{0\}})$ consisting of sequences with $N+1$ non-trivial Chebyshev coefficients in $\mathbb{R}$.
Hence, formally, the mapping $\boldsymbol{\pi}^N \mathbf{F}_\circ \boldsymbol{\pi}^N$ sends $\mathbb{R}^{1 + m(N+1)}$ into itself; similarly, $\boldsymbol{\pi}^N \mathbf{F}_{\circ,\textnormal{elem}} \boldsymbol{\pi}^N$ sends $\mathbb{R}^{1 + d + m(1+d)(N+1)}$ into itself.
While we do not detail this, depending on the profile of the solution (how large $m$ and $N$ are) one may want to exploit the structure of the Fr\'{e}chet derivative $D\mathbf{F}_\circ$ and $D \mathbf{F}_{\circ, \textnormal{elem}}$ as these block-wise operators have sparse matrix representations.

We conclude this section by detailing the computation of a $m\tau$-periodic orbit in the numerically observed chaotic attractor of the cubic Ikeda equation \eqref{eq:cubic_ikeda} (see also Figure \ref{fig:cubic_ikeda_attractor}).

\subsection{Example: periodic orbit for the cubic Ikeda equation}
\label{sec:cubic_ikeda_po}

Consider the cubic Ikeda equation \eqref{eq:cubic_ikeda}. We first fix a value for the delay $\tau$ within the numerically observed chaotic window $1.538 \lesssim \tau \lesssim 1.723$ (e.g. see \cite{Sprott2007}).
By a standard method of steps (using, for instance, the Tsitouras 5/4 Runge-Kutta method), we sweep the chaotic attractor, looking for a periodic orbit.
Once an approximate periodic time series is identified, we perform a simple parameter continuation with respect to the delay $\tau$ so as to have approximately a period $m\tau$ for some $m \in \mathbb{N}$.

Next, we split into $m$ pieces the time series of the periodic orbit and retrieve their Chebyshev series.
At this point, we have obtained an initial guess for Newton's iterations \eqref{eq:Newton_method_fp}.
Precisely, for the cubic Ikeda equation, $f$, given in \eqref{eq:f_cubic_ikeda}, is polynomial and acts on the sequences of Chebyshev coefficients as $f (\mathbf{a}, \mathbf{b}) = \mathbf{b} - \mathbf{b}^{*3}$ where $\mathbf{b}^{*k} \bydef \underbrace{\mathbf{b} * \dots * \mathbf{b}}_{k \text{ times}}$.
It follows that
\[
\mathbf{F}_\circ (\tau, \mathbf{c})
=
\begin{pmatrix}
\{\mathbf{E}(\mathbf{c}_m)\}_0 - \delta \\
\mathbf{E} (\mathbf{c}_m) + \mathbf{S} ( \frac{\tau}{2} (\mathbf{c}_m - \mathbf{c}_m^{*3}) ) - \mathbf{c}_1 \\
\mathbf{E} (\mathbf{c}_1) + \mathbf{S} ( \frac{\tau}{2} (\mathbf{c}_1 - \mathbf{c}_1^{*3}) ) - \mathbf{c}_2 \\
\vdots \\
\mathbf{E} (\mathbf{c}_{m-1}) + \mathbf{S} ( \frac{\tau}{2} (\mathbf{c}_{m-1} - \mathbf{c}_{m-1}^{*3}) ) - \mathbf{c}_m
\end{pmatrix},
\]
where the phase $\delta$ is prescribed by the numerical data, and
\begin{align*}
&D\mathbf{F}_\circ (\tau, \mathbf{c}) \\
&=
\begin{pmatrix}
D_\tau \mathbf{F}_\circ (\tau,\mathbf{c}) &
D_{\mathbf{c}_1} \mathbf{F}_\circ (\tau,\mathbf{c}) &
D_{\mathbf{c}_2} \mathbf{F}_\circ (\tau,\mathbf{c}) &
\dots &
D_{\mathbf{c}_m} \mathbf{F}_\circ (\tau,\mathbf{c})
\end{pmatrix} \\
&=
{\small
\begin{pmatrix}
0 &
0 & 0 & \cdots & 0 & \mathbf{E} \\
\mathbf{S}(\frac{1}{2}(\mathbf{c}_m - \mathbf{c}_m^{*3})) &
- \mathbf{I} & 0 & \cdots & 0 & \mathbf{E} + \mathbf{S}[\frac{\tau}{2}(\mathbf{I} - 3\mathbf{M}_{\mathbf{c}_m^{*2}})] \\
\mathbf{S}(\frac{1}{2}(\mathbf{c}_1 - \mathbf{c}_1^{*3})) &
\mathbf{E} + \mathbf{S}[\frac{\tau}{2}(\mathbf{I} - 3\mathbf{M}_{\mathbf{c}_1^{*2}})] & -\mathbf{I} & & & 0 \\
\vdots &
& & \ddots & & \\
\mathbf{S}(\frac{1}{2}(\mathbf{c}_{m-1} - \mathbf{c}_{m-1}^{*3})) &
0 & & & \mathbf{E} + \mathbf{S}[\frac{\tau}{2}(\mathbf{I} - 3\mathbf{M}_{\mathbf{c}_{m-1}^{*2}})] & -\mathbf{I}
\end{pmatrix}
},
\end{align*}
where $\mathbf{I}$ is the identity on $\ell^1_\nu$ and $\mathbf{M}_\mathbf{a} : \ell^1_\nu \to \ell^1_\nu$ is the multiplication operator of a given $\mathbf{a} \in \ell^1_\nu$, specifically $\mathbf{M}_\mathbf{a} (\mathbf{b}) \bydef \mathbf{a} * \mathbf{b}$ for all $\mathbf{b} \in \ell^1_\nu$.
Observe that $f(\mathbf{a}, \mathbf{b})$ is independent of $\mathbf{a}$, which has simplified a little the expression for $D\mathbf{F}_\circ (\tau, \mathbf{c})$.

\begin{figure}
\centering
\begin{subfigure}[b]{0.41\textwidth}
\includegraphics[width=\textwidth]{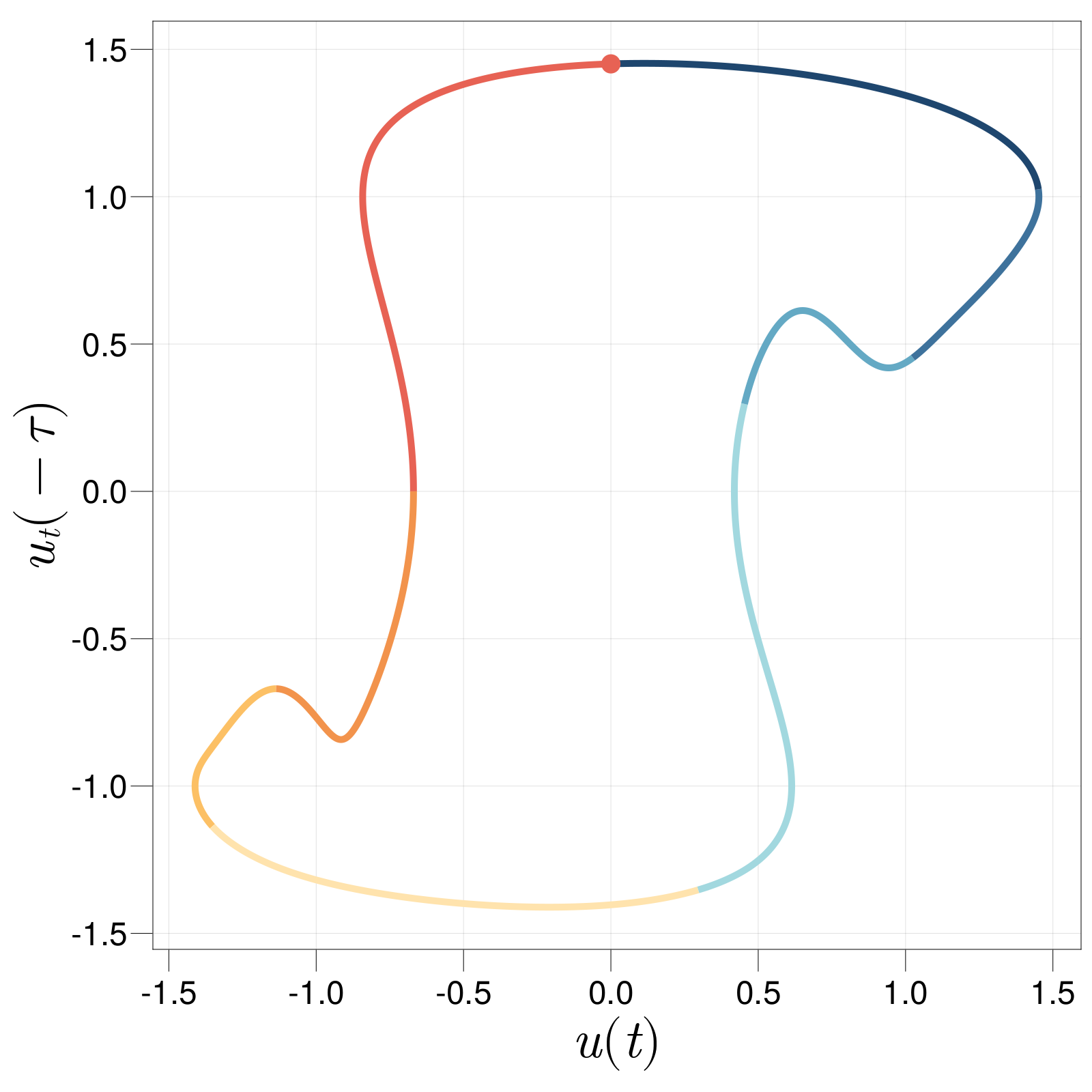}
\caption{}
\end{subfigure}
\hspace{0.06\textwidth}
\begin{subfigure}[b]{0.41\textwidth}
\includegraphics[width=\textwidth]{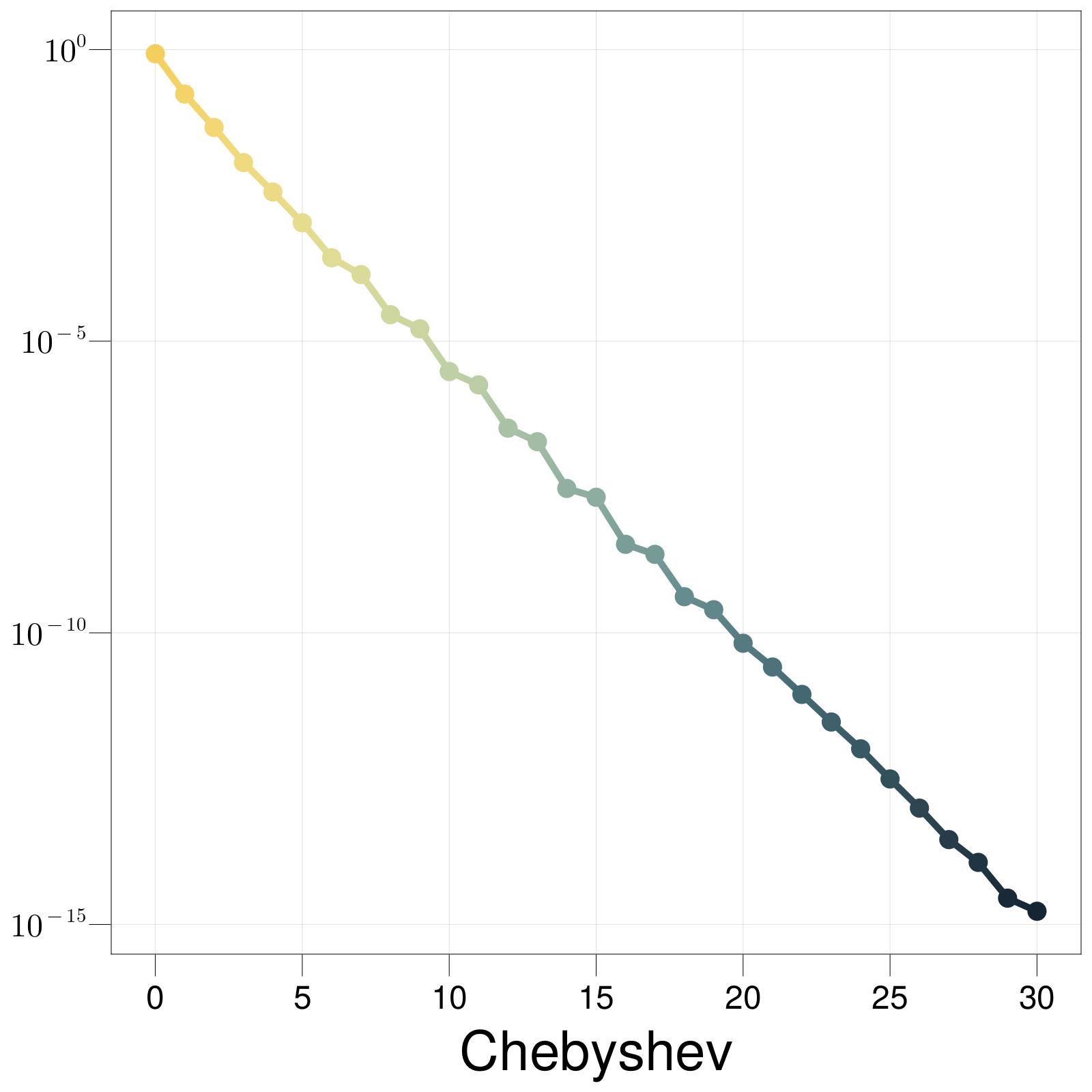}
\caption{}
\end{subfigure}
\caption{(a) $m\tau$-periodic orbit, with $m=8$, for the cubic Ikeda equation. The dot corresponds to the phase $\delta = 0$ of the periodic orbit.
(b) Average $\{ m^{-1}\sum_{j=1}^m |\{ (\bar{\mathbf{c}}_\textnormal{init})_j \}_\alpha| \}_{\alpha \ge 0}$ of the sequences of Chebyshev coefficients of the $m\tau$-periodic orbit shown in (a).}
\label{fig:cubic_ikeda_po}
\end{figure}

In our case, we identified a time series of a $m\tau$-periodic orbit with $m = 8$ and a phase $\delta = 0$.
We choose the truncation order $N = 30$ for the Chebyshev series.
Therefore, the Newton iterations are set on $\mathbb{R} \times \boldsymbol{\pi}^N (\ell^1_\nu \cap \mathbb{R}^{\mathbb{N}\cup\{0\}})^m \simeq \mathbb{R} \times \mathbb{R}^{m(N+1)} \simeq \mathbb{R}^{1 + 8\times31} = \mathbb{R}^{249}$.
Performing Newton's iterations yields $\bar{\tau}_\textnormal{init} \approx 1.5649592985680902$ and the sequences of Chebyshev coefficients $\bar{\mathbf{c}}_\textnormal{init} = ((\bar{\mathbf{c}}_\textnormal{init})_1, \dots, (\bar{\mathbf{c}}_\textnormal{init})_m) \in \boldsymbol{\pi}^N (\ell^1_\nu \cap \mathbb{R}^{\mathbb{N}\cup\{0\}})^m$.
Figure \ref{fig:cubic_ikeda_po} shows the approximate $m\tau$-periodic orbit and the average of the sequences Chebyshev coefficients.

%%%%%%%%%%%%%%%%%%%%%%%%
%% EIGENDECOMPOSITION %%
%%%%%%%%%%%%%%%%%%%%%%%%

\section{Computation of the eigendecomposition}\label{sec:eigendecomposition}
%!TEX root = poincare_scenario_dde.tex

To describe the intersection of the invariant manifolds of a periodic orbit of the DDE \eqref{eq:dde} in the framework of the discrete dynamics given by the time-$\tau$ map $\mathcal{F}$, we begin by investigating the spectrum and eigenspaces of the linearized problem.

The assumption that $f$ is a polynomial guarantees that $\mathcal{F}$ is compact (e.g. see \cite{Hale1977}). Its Fr\'{e}chet derivative inherits this property such that the spectrum of $D\mathcal{F}$ is comprised of eigenvalues accumulating at $0$.
Similarly, $\mathring{\mathcal{F}}$ and $D\mathring{\mathcal{F}}$ are also compact.
The following lemma relates the Floquet multipliers of the DDE \eqref{eq:dde} with the ones of the DDSs \eqref{eq:dds} and \eqref{eq:unrolled_dds}.
Note that it is important here that the discrete dynamics are generated by the time-$\tau$ map of the flow of the DDE, and that these results would have to be modified for more general Poincare sections.

\begin{lemma}\label{lem:equivalence_eig}
Let $\tau > 0$, $m \in \mathbb{N}$, $\phi \in C([-\tau, 0], \mathbb{R}^n)$, $\lambda \in \mathbb{C}$ and $v \in C([-\tau, 0], \mathbb{C})^n$. The following statements are equivalent:
\begin{enumerate}
\item $\lambda^m$ is a Floquet multiplier associated to the $m\tau$-periodic orbit $t \in \mathbb{R}/m\tau\mathbb{Z} \mapsto S_t (\phi)$ of the DDE \eqref{eq:dde}. Namely, for $j = 1, \dots, m$, $\lambda^m$ is an eigenvalue of $D S_{m\tau}( S_{(j-1)\tau}(\phi) )$ with eigenvector $v_j = [D\mathcal{F}^{j-1}(\phi)]v$.
\item $\lambda^m$ is a Floquet multiplier associated to the $m$-periodic orbit $j \in \mathbb{Z}/m\mathbb{Z} \mapsto \mathcal{F}^j(\phi)$ of the DDS \eqref{eq:dds}. Namely, for $j = 1, \dots, m$, $\lambda^m$ is an eigenvalue of $D \mathcal{F}^m ( \mathcal{F}^{j-1}(\phi) )$ with eigenvector $v_j = [D\mathcal{F}^{j-1}(\phi)]v$.
\item $\lambda, \lambda e^{i\frac{2\pi}{m}}, \dots, \lambda e^{i\frac{2\pi(m-1)}{m}}$ are eigenvalues associated to the fixed-point $c \bydef (\phi, \mathcal{F}(\phi), \dots, \mathcal{F}^{m-1}(\phi))$ of the DDS \eqref{eq:unrolled_dds}. Namely, for any $k \in \{ 0, \dots, m-1 \}$, the complex number $\lambda_k \bydef \lambda e^{i\frac{2\pi k}{m}}$ is an eigenvalue of $D \mathring{\mathcal{F}}(c)$ with corresponding eigenvector $(v_1, \lambda_k^{-1} v_2, \dots,  \lambda_k^{-(m-1)} v_m)$, where $v_j = [D\mathcal{F}^{j-1}(\phi)]v$ for $j = 1, \dots, m$.
\end{enumerate}
\end{lemma}

\begin{proof}
On the one hand, for $j = 1, \dots, m$, Lemma \ref{lem:equivalence_po_fp} guarantees the equality $S_{(j-1)\tau}(\phi) = \mathcal{F}^{j-1} (\phi)$, thus $[D S_{m\tau} ( S_{(j-1)\tau}(\phi) )] v_j = [D \mathcal{F}^m ( \mathcal{F}^{j-1}(\phi) )] v_j$. It follows that Point 1 is equivalent to Point 2.

Denoting $c = (c_1, \dots, c_m)$, we have that
\[
D \mathring{\mathcal{F}} (c) =
\begin{pmatrix}
0 & \cdots & 0 & D \mathcal{F} (c_m) \\
D \mathcal{F} (c_1) & & 0 & 0 \\
& \ddots & & \vdots \\
0 & & D \mathcal{F} (c_{m-1}) & 0
\end{pmatrix}.
\]
Therefore, fixing $k \in \{0, \dots, m-1\}$, we have that the equality
\[
D \mathring{\mathcal{F}} (c)
\begin{pmatrix}
v_1 \\
\lambda_k^{-1} v_2 \\
\vdots \\
\lambda_k^{-(m-1)} v_m
\end{pmatrix}
=
\begin{pmatrix}
\lambda_k^{-(m-1)} [D \mathcal{F}^m (\phi)] v \\
[D \mathcal{F}(\phi)] v \\
\vdots \\
\lambda_k^{-(m-2)} [D \mathcal{F}^{m-1}(\phi)] v
\end{pmatrix}
=
\lambda_k
\begin{pmatrix}
v_1 \\
\lambda_k^{-1} v_2 \\
\vdots \\
\lambda_k^{-(m-1)} v_m
\end{pmatrix}
\]
is equivalent to $D \mathcal{F}^m (c_j) v_j = \lambda_k^m v_j = \lambda^m v_j$ for $j = 1, \dots, m$. It follows that Point 2 is equivalent to Point 3.
\end{proof}

Now, given a periodic orbit, the phase-space $C([-\tau,0], \mathbb{R}^n)$ of the DDE \eqref{eq:dde} can be decomposed into $C([-\tau,0], \mathbb{R}^n) = E_\textnormal{u} \oplus E_\textnormal{c} \oplus E_\textnormal{s}$ where $E_\textnormal{u}, E_\textnormal{c},E_\textnormal{s}$ are the unstable, center and stable eigenspaces respectively.
By compactness of the solution operator, $E_\textnormal{u}$ and $E_\textnormal{c}$ are necessarily finite dimensional and correspond to the span of the unstable and center (generalized) eigenvectors respectively.

Even though the time-$\tau$ map $\mathcal{F}$ is a priori implicitly defined by \eqref{eq:def_F}, it turns out that its Fr\'{e}chet derivative can be explicitly computed by noting that
\[
D\mathcal{F} (\phi) = D_\phi \mathcal{T} (\mathcal{F}(\phi), \phi ) = [D_1 \mathcal{T} (\mathcal{F}(\phi), \phi )] D\mathcal{F}(\phi) + D_2 \mathcal{T} (\mathcal{F}(\phi), \phi ),
\]
where
\[
\begin{aligned}
\relax [ [ D_1 \mathcal{T} (\mathcal{F}(\phi), \phi )] u](s) &= \int_{-\tau}^s [D_1 f([\mathcal{F}(\phi)](s'), \phi(s'))] u(s') \, ds', \\
\relax [ [ D_2 \mathcal{T} (\mathcal{F}(\phi), \phi )] u](s) &= u(0) + \int_{-\tau}^s [D_2 f([\mathcal{F}(\phi)](s'), \phi(s'))] u(s') \, ds',
\end{aligned}
\qquad \text{for all } u \in C([-\tau, 0], \mathbb{R}^n),
\]
Since $D_1 \mathcal{T} (\mathcal{F}(\phi), \phi )$ is compact, $I - D_1 \mathcal{T} (\mathcal{F}(\phi), \phi )$ is invertible if and only if it is injective.
We have that $(I - D_1 \mathcal{T} (\mathcal{F}(\phi), \phi )) u = 0$ is the integral form of a linear ODE. By uniqueness of the solution, it follows that $u$ is identically zero. Therefore, the injectivity holds and so does the equality
\[
D\mathcal{F} (\phi) = (I - D_1 \mathcal{T} (\mathcal{F}(\phi), \phi ) )^{-1} D_2 \mathcal{T} (\mathcal{F}(\phi), \phi ).
\]

According to Lemma \ref{lem:equivalence_eig}, the spectrum and eigenspaces of a $m$-periodic orbit $c_j$, for $j = 1, \dots, m$, of the DDS \eqref{eq:dds} can be equivalently studied with $[D_\phi \mathcal{F}^m ( \phi )]_{\phi = c_j}$, for some $j \in \{1, \dots, m\}$, or $D\mathring{\mathcal{F}} (c)$, with $c \bydef (c_1, \dots, c_m)$ being a fixed-point of the DDS \eqref{eq:unrolled_dds}.
Despite the fact that we compute the periodic orbit as a fixed-point of $\mathring{\mathcal{F}}$ (cf. Section \ref{sec:periodic_orbit}), we favour looking at the eigendecomposition of $[D_\phi \mathcal{F}^m ( \phi )]_{\phi = c_j}$.
The reason is that in practice, numerical errors tend to tarnish the eigenvalues accumulating to $0$. The spectrum of $D\mathring{\mathcal{F}} (c)$ is then laborious to parse since it contains $m$ copies, coming from the $m$-th root, of each eigenvalue of $[D_\phi \mathcal{F}^m ( \phi )]_{\phi = c_j}$.

By using the chain rule, we can then compute the eigendecomposition of $D\mathcal{F}^m(c_j)$.
More precisely, let $\mathbf{c} = (\mathbf{c}_1, \dots, \mathbf{c}_m) \in ((\ell^1_\nu)^n)^m$ where $\mathbf{c}_j$ is the sequence of Chebyshev coefficients representing $c_j$. Define
\begin{align*}
\hat{\mathbf{H}}_j(\tau, \mathbf{c}) &\bydef (\mathbf{I} - \mathbf{S} [\tfrac{\tau}{2} D_1 f(\mathbf{c}_{j+1}, \mathbf{c}_j)] )^{-1} ( \mathbf{E} + \mathbf{S} [\tfrac{\tau}{2} D_2 f(\mathbf{c}_{j+1}, \mathbf{c}_j)] ), \qquad j = 1, \dots, m-1, \\
\hat{\mathbf{H}}_m(\tau, \mathbf{c}) &\bydef (\mathbf{I} - \mathbf{S} [\tfrac{\tau}{2} D_1 f(\mathbf{c}_1, \mathbf{c}_m)] )^{-1} ( \mathbf{E} + \mathbf{S} [\tfrac{\tau}{2} D_2 f(\mathbf{c}_1, \mathbf{c}_m)] ).
\end{align*}
Consequently, for $j = 1, \dots, m$, we have that $[D_{\phi} \mathcal{F}^m ( \phi )]_{\phi = c_j}$ is represented by
\begin{equation}\label{eq:def_H}
\mathbf{H}_j(\tau, \mathbf{c}) \bydef \hat{\mathbf{H}}_{j-1}(\tau, \mathbf{c}) \cdot \hat{\mathbf{H}}_{j-2}(\tau, \mathbf{c}) \cdot \ldots \cdot \hat{\mathbf{H}}_1(\tau, \mathbf{c}) \cdot \hat{\mathbf{H}}_m(\tau, \mathbf{c}) \cdot \hat{\mathbf{H}}_{m-1}(\tau, \mathbf{c}) \cdot \ldots \cdot \hat{\mathbf{H}}_j(\tau, \mathbf{c}).
\end{equation}

Suppose we have a Floquet multiplier $\lambda^m$ of $[D_\phi \mathcal{F}^m ( \phi )]_{\phi = c_j}$, for some $j \in \{1, \dots, m\}$.
In the forthcoming Section \ref{sec:unstable_manifold}, to parameterize the local unstable manifold, we will need to retrieve the eigenvector of $D\mathring{\mathcal{F}} (c)$ associated with $\lambda$.
According to Point 3 of Lemma \ref{lem:equivalence_eig}, this eigenvector is represented by $\mathbf{v} = (\mathbf{v}_1, \dots, \mathbf{v}_m) \in ((\ell^1_\nu)^n)^m$ where $\mathbf{v}_j$ is the eigenvector of $\mathbf{H}_j$ associated with $\lambda^m$ such that
\begin{equation}\label{eq:eigenvector}
\begin{aligned}
\relax [\mathbf{H}_1 (\tau, \mathbf{c})] \mathbf{v}_1 &= \lambda^m \mathbf{v}_1, \\
\mathbf{v}_2
&= \lambda^{-1} [\hat{\mathbf{H}}_1 (\tau, \mathbf{c})] \mathbf{v}_1, \\
\mathbf{v}_j
&= \lambda^{-1} [\hat{\mathbf{H}}_{j-1} (\tau, \mathbf{c})] \mathbf{v}_{j-1}
= \lambda^{-(j-1)} [\hat{\mathbf{H}}_{j-1} (\tau, \mathbf{c})] \cdot \ldots \cdot [\hat{\mathbf{H}}_1 (\tau, \mathbf{c})] \mathbf{v}_1,
\qquad j = 3, \dots, m.
\end{aligned}
\end{equation}

So far, we have characterized the Floquet multipliers and eigenspaces for the polynomial DDE \eqref{eq:dde}.
However, Lemma \ref{lem:equivalence_eig} does not hint on what the spectrum and eigenspaces are for the DDE \eqref{eq:dde_elem} when the DDE \eqref{eq:dde} is an auxiliary polynomial DDE of \eqref{eq:dde_elem}.
According to Theorem 3.5 in \cite{Henot2021}, if the periodic orbit of the auxiliary polynomial DDE \eqref{eq:dde} represents a periodic orbit of the DDE \eqref{eq:dde_elem}, then the stable and unstable Floquet multipliers are identical and the associated eigenvectors coincide.
On the other hand, the $d$ additional coordinates in the construction of $f$, given in \ref{eq:f_poly}, introduce $d$ center Floquet multipliers whose eigenvectors do not pertain to the original DDE \eqref{eq:dde_elem}; in particular, if the periodic orbit is hyperbolic, then $\dim E_\textnormal{c} = 1 + d$.
This last claim can be made rigorous from the results presented in \cite{Henot2021}; although, for the needs of this article, we shall be satisfied by observing this numerically.

\subsection{Numerical considerations}
\label{sec:eigendecomposition_num}

Consider an approximate zero $\bar{\tau}>0$, $\bar{\mathbf{c}} = (\bar{\mathbf{c}}_1, \dots, \bar{\mathbf{c}}_m) \in \boldsymbol{\pi}^N ((\ell^1_\nu)^n)^m$ of $\mathbf{F}_\circ$, or $\mathbf{F}_{\circ,\textnormal{elem}}$ (cf. Section \ref{sec:periodic_orbit_num}).
Then, for $j = 1, \dots, m$, we have that the operator $\mathbf{H}_j$ given in \eqref{eq:def_H} is approximated by
\begin{equation}\label{eq:def_H_num}
\mathbf{H}_j^N(\bar{\tau}, \bar{\mathbf{c}}) \bydef \hat{\mathbf{H}}_{j-1}^N(\bar{\tau}, \bar{\mathbf{c}}) \cdot \hat{\mathbf{H}}_{j-2}^N(\bar{\tau}, \bar{\mathbf{c}}) \cdot \ldots \cdot \hat{\mathbf{H}}_1^N(\bar{\tau}, \bar{\mathbf{c}}) \cdot \hat{\mathbf{H}}_m^N(\bar{\tau}, \bar{\mathbf{c}}) \cdot \hat{\mathbf{H}}_{m-1}^N(\bar{\tau}, \bar{\mathbf{c}}) \cdot \ldots \cdot \hat{\mathbf{H}}_j^N(\bar{\tau}, \bar{\mathbf{c}}),
\end{equation}
where
\begin{align*}
\hat{\mathbf{H}}_j^N(\bar{\tau}, \bar{\mathbf{c}}) &\bydef (\boldsymbol{\pi}^N - \boldsymbol{\pi}^N \mathbf{S} [\tfrac{\bar{\tau}}{2} D_1 f(\bar{\mathbf{c}}_{j+1}, \bar{\mathbf{c}}_j)] \boldsymbol{\pi}^N )^{-1} ( \mathbf{E} \boldsymbol{\pi}^N + \boldsymbol{\pi}^N \mathbf{S} [\tfrac{\bar{\tau}}{2} D_2 f(\bar{\mathbf{c}}_{j+1}, \bar{\mathbf{c}}_j)] \boldsymbol{\pi}^N ), \quad j = 1, \dots, m-1, \\
\hat{\mathbf{H}}_m^N(\bar{\tau}, \bar{\mathbf{c}}) &\bydef (\boldsymbol{\pi}^N - \boldsymbol{\pi}^N \mathbf{S} [\tfrac{\bar{\tau}}{2} D_1 f(\bar{\mathbf{c}}_1, \bar{\mathbf{c}}_m)] \boldsymbol{\pi}^N )^{-1} ( \mathbf{E} \boldsymbol{\pi}^N + \boldsymbol{\pi}^N \mathbf{S} [\tfrac{\bar{\tau}}{2} D_2 f(\bar{\mathbf{c}}_1, \bar{\mathbf{c}}_m)] \boldsymbol{\pi}^N ).
\end{align*}

Clearly, $\mathbf{H}_j^N (\bar{\tau}, \bar{\mathbf{c}})$ can only provide a finite portion of the spectrum of $[D_\phi \mathcal{F}^m ( \phi )]_{\phi = c_j}$, namely $n(N+1)$ many eigenvalues.
Nevertheless, the missing eigenvalues get closer to $0$ as the truncation order $N$ increases.
To give some insights, we sketch the argument.
The operators $\mathbf{H}_j^N$ define a sequence, with respect to $N$, of finite rank operators converging in norm (inherited from $\ell^1_\nu$) to $\mathbf{H}_j$ whose spectrum is identical to the one of $[D_\phi \mathcal{F}^m ( \phi )]_{\phi = c_j}$. It turns out that for any $\zeta \in [0, 1]$, the eigenvalues of the homotopy $\zeta \mapsto (1-\zeta)\mathbf{H}_j^N(\tau, \mathbf{c}) + \zeta \mathbf{H}_j (\tau, \mathbf{c})$, which are not eigenvalues of $\mathbf{H}_j^N(\tau, \mathbf{c})$, are contained in a disk in $\mathbb{C}$ which shrinks as $N$ grows.
The interested reader may refer to \cite{Lessard2020}, and references therein, for more details as well as an application of this property to rigorously validate the spectrum of equilibria for DDEs.

Furthermore, the eigenvectors $\mathbf{v}_1, \dots, \mathbf{v}_m$ given in \eqref{eq:eigenvector} are approximated by $\bar{\mathbf{v}}_1, \dots, \bar{\mathbf{v}}_m \in \boldsymbol{\pi}^N (\ell^1_\nu)^n$ satisfying
\begin{equation}\label{eq:eigenvector_num}
[\mathbf{H}_1^N (\bar{\tau}, \bar{\mathbf{c}})] \bar{\mathbf{v}}_1 = \bar{\lambda}^m \bar{\mathbf{v}}_1, \qquad
\bar{\mathbf{v}}_j
= \bar{\lambda}^{-1} [\hat{\mathbf{H}}^N_{j-1} (\bar{\tau}, \bar{\mathbf{c}})] \bar{\mathbf{v}}_{j-1}, \quad j = 2, \dots, m.
\end{equation}

We conclude this section by detailing the computation of the eigendecomposition associated with the $m\tau$-periodic orbit obtained for the cubic Ikeda equation \eqref{eq:cubic_ikeda} (cf. Section \ref{sec:cubic_ikeda_po}).

\subsection{Example: eigendecomposition for the cubic Ikeda equation}
\label{sec:cubic_ikeda_eigendecomposition}

For the cubic Ikeda equation \eqref{eq:cubic_ikeda}, $f(\mathbf{a}, \mathbf{b})$ is independent of $\mathbf{a}$, such that the operators $\hat{\mathbf{H}}_j^N$ reduce to
\begin{align*}
\hat{\mathbf{H}}_j^N (\tau, \mathbf{c}) &= \mathbf{E} \boldsymbol{\pi}^N + \boldsymbol{\pi}^N \mathbf{S} ( \tfrac{\tau}{2} [\mathbf{I} - 3\mathbf{M}_{\mathbf{c}_j^{*2}}] ) \boldsymbol{\pi}^N, \qquad j = 1, \dots, m-1, \\
\hat{\mathbf{H}}_m^N (\tau, \mathbf{c}) &= \mathbf{E} \boldsymbol{\pi}^N + \boldsymbol{\pi}^N \mathbf{S} ( \tfrac{\tau}{2} [\mathbf{I} - 3\mathbf{M}_{\mathbf{c}_m^{*2}}] ) \boldsymbol{\pi}^N.
\end{align*}

\begin{figure}
\centering
\includegraphics[height=4cm]{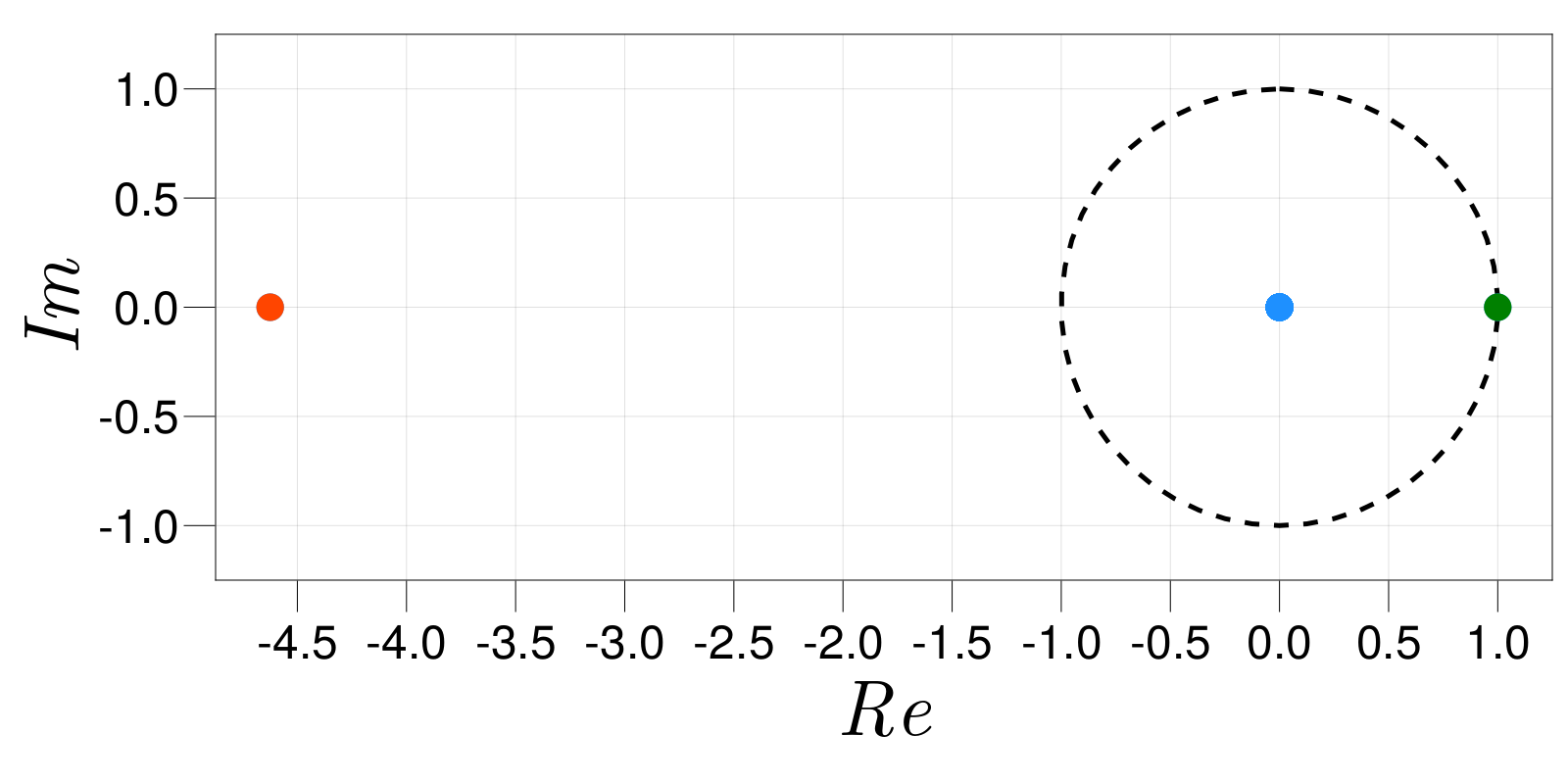}
\vspace{-.3cm}
\caption{Collection of $N+1 = 31$ Floquet multipliers associated with the $m\tau$-periodic orbit shown on Figure \ref{fig:cubic_ikeda_po} for the cubic Ikeda equation.
The black dashed circle is the unit circle.
There is $1$ unstable eigenvalue (red dot), $1$ centre eigenvalue (green dot) and $29$ stable eigenvalues (blue dots).
Note that due to the proximity of the stable eigenvalues, only a single blue dot appears on the figure.}
\label{fig:cubic_ikeda_spectrum}
\end{figure}

We numerically retrieve the spectrum of $\mathbf{H}^N_1 (\bar{\tau}_\textnormal{init}, \bar{\mathbf{c}}_\textnormal{init})$, given in \eqref{eq:def_H_num}, with $\bar{\tau}_\textnormal{init} > 0$ and $\bar{\mathbf{c}}_\textnormal{init} \in \boldsymbol{\pi}^N (\ell^1_\nu)^m$ computed in Section \ref{sec:cubic_ikeda_po}; in particular, here $m=8$ and $N = 30$.
The numerical spectrum consists of $N+1 = 31$ eigenvalues; Figure \ref{fig:cubic_ikeda_spectrum} suggests that the periodic orbit has a single unstable Floquet multiplier $\bar{\mu} \approx -4.624622928960324$.
We choose arbitrarily one of the $m$-th root $\bar{\lambda} = |\bar{\mu}|^{\frac{1}{m}} e^{i\frac{\pi}{m}}$.
For $j=1, \dots, m$, we consider the approximate unstable eigenvector $\bar{\mathbf{v}}_j \in \boldsymbol{\pi}^N \ell^1_\nu$ of $\mathbf{H}^N_j (\bar{\tau}_\textnormal{init}, \bar{\mathbf{c}}_\textnormal{init})$ associated with $\bar{\lambda}^m$ as given by \eqref{eq:eigenvector_num}. Specifically, we have
\[
[\mathbf{H}_1^N (\bar{\tau}_\textnormal{init}, \bar{\mathbf{c}}_\textnormal{init})] \bar{\mathbf{v}}_1 = \bar{\lambda}^m \bar{\mathbf{v}}_1, \qquad
\bar{\mathbf{v}}_j
= \bar{\lambda}^{-1} (\mathbf{E} \boldsymbol{\pi}^N + \boldsymbol{\pi}^N \mathbf{S} ( \tfrac{\bar{\tau}_\textnormal{init}}{2} [\mathbf{I} - 3\mathbf{M}_{(\bar{\mathbf{c}}_\textnormal{init})_{j-1}^{*2}}] ) \boldsymbol{\pi}^N) \bar{\mathbf{v}}_{j-1}, \quad j = 2, \dots, m.
\]
%

%%%%%%%%%%%%%%%%%%%%%%%
%% UNSTABLE MANIFOLD %%
%%%%%%%%%%%%%%%%%%%%%%%

\section{Computation of the unstable manifold}\label{sec:unstable_manifold}
%!TEX root = poincare_scenario_dde.tex

To construct a homoclinic orbit for a periodic solution, we first compute the unstable manifold via the \emph{parameterization method}.
The essential references for the parameterization method are the three articles \cite{Cabre2003, Cabre2003-2, Cabre2005}. The parameterization method was extended to equilibrium and periodic orbits of DDEs in \cite{Groothedde2017, Henot2022}, and in the last reference just cited the authors obtain validated, computer assisted bounds on the discretization and truncation errors.
We refer also to the works of \cite{MR3706909,MR4446093}, where parameterization methods for computing stable/unstable manifolds attached to periodic orbits of explicitly and implicitly defined finite dimensional discrete time dynamical system are developed.  The technique developed below extends this work to infinite dimensional, implicitly defined, compact maps.

Loosely speaking, the idea is to find a parameterization characterized as a mapping lifting the trajectories from the unstable eigenspace onto the unstable manifold.
An appeal of this method is that, under a \emph{non-resonance} condition for the eigenvalues, the resulting parameterization is not constrained to be a local graph.

Let $\phi \in C([-\tau, 0], \mathbb{R}^n)$ and $t \in \mathbb{R}/m\tau\mathbb{Z} \mapsto S_t (\phi)$ be a $m\tau$-periodic orbit of the DDE \eqref{eq:dde}.
Assume that $D S_{m\tau}(\phi)$ has no defective unstable eigenvalues $\lambda_1^m, \dots, \lambda_{n_\textnormal{u}}^m \in \mathbb{C}$ (i.e. the associated eigenvectors are linearly independent and span $E_\textnormal{u}$).
Consider the diagonal matrix
\[
\Lambda \bydef \begin{pmatrix}
\lambda_1 & & 0 \\
& \ddots & \\
0 & & \lambda_{n_\textnormal{u}}
\end{pmatrix}.
\]
We look for a mapping $P : (\mathbb{R}/m\tau\mathbb{Z}) \times \mathbb{C}^{n_\textnormal{u}} \to \mathcal{C}^n$ which acts as a topological conjugacy between the semi-flow of the DDE \eqref{eq:dde} and the corresponding linearized flow about the $m\tau$-periodic orbit, namely
\begin{equation}\label{eq:conjugacy}
P(t + \theta~(\text{mod}~m\tau), \Lambda^{t/\tau} \sigma) = S_t ( P(\theta, \sigma) ),
\end{equation}
where $t \in \mathbb{R}$, $\theta \in \mathbb{R}/m\tau\mathbb{Z}$ and $\sigma \in \mathbb{C}^{n_\textnormal{u}}$. Note that if \eqref{eq:conjugacy} holds, then the solution operator $S_t$ defines a flow, i.e. is well-defined for $t < 0$, on the image of $P$. Moreover, evaluating at $\sigma = 0$ shows that $t \mapsto P(t + \theta~(\text{mod}~m\tau), 0)$ is a periodic orbit of the DDE \eqref{eq:dde}; it also follows that trajectories on the image of $P$ goes to this periodic orbit as $t \to -\infty$. In other words, the conjugacy $P$, if it exists, yields a parameterization of the unstable manifold.

Note that we slightly constrained the image of $P$ to $\mathcal{C}^n$ instead of $C([-\tau, 0],\mathbb{R}^n)$ since we wish to deal with the phase space of the DDE \eqref{eq:dde}, the DDSs \eqref{eq:dds} and \eqref{eq:unrolled_dds} altogether.

The question of existence of $P$ is especially relevant in the context of DDEs where it is common for solution operators to not be one-to-one.
In such cases, the unstable manifold may collapse into a lower dimensional manifold.
However, even then, there exists a small enough neighbourhood of the periodic orbit for which the local unstable manifold is a graph over the unstable eigenspace.
Hence, $P$ would be, at best, a local graph of the unstable manifold.
Yet, we stress that $P$ does not have to be a graph.
In fact, for analytic DDEs (as in the context of this article) trajectories on the unstable manifold are analytic, the one-to-oneness of the solution operator on this manifold is guaranteed and the unstable manifold is analytic.
See for instance \cite{Hale1977}.

Then, it is known (see \cite{Cabre2003,Cabre2003-2,Cabre2005}) that a mapping $P$ satisfying the conjugacy relation \eqref{eq:conjugacy} always exists provided that the eigenvalues $\lambda_1^m, \dots, \lambda_{n_\textnormal{u}}^m$ are non-resonant; the definition of non-resonant eigenvalues is reported below.

\begin{definition}\label{def:resonance}
A collection $\mu_1, \dots, \mu_K \in \mathbb{C}$ of eigenvalues are non-resonant whenever the equalities
\[
\mu_1^{\alpha_1} \times \ldots \times \mu_K^{\alpha_K} = \mu_j, \qquad j = 1, \dots, K,
\]
for $\alpha_1, \dots, \alpha_K \in \mathbb{N} \cup \{ 0 \}$, only hold for the trivial case $\alpha_j = 1$, $\alpha_l = 0$ for all $l = 1, \dots, K$, $l\neq j$.
\end{definition}

The following lemma details the relation between the parameterization of the local unstable manifold of a periodic orbit of the DDE \eqref{eq:dde} and the DDSs \eqref{eq:dds} and \eqref{eq:unrolled_dds}.

\begin{lemma}\label{lem:equivalence_param}
Let $\tau > 0$, $m \in \mathbb{N}$, $\phi \in C([-\tau, 0], \mathbb{R}^n)$, $\lambda_1, \dots, \lambda_{n_\textnormal{u}} \in \mathbb{C}$, $\Lambda \bydef \textnormal{diag}(\lambda_1, \dots, \lambda_{n_\textnormal{u}})$ and $U$ be an open subset of $\mathbb{C}^{n_\textnormal{u}}$. Assume that $t \in \mathbb{R}/m\tau\mathbb{Z} \mapsto S_t (\phi)$ is a $m\tau$-periodic orbit of the DDE \eqref{eq:dde} with non-resonant unstable Floquet multipliers $\lambda_1^m, \dots, \lambda_{n_\textnormal{u}}^m$. The following statements are equivalent:
\begin{enumerate}
\item $P : (\mathbb{R}/m\tau\mathbb{Z}) \times U \to \mathcal{C}^n$ parameterizes a local unstable manifold and satisfies the conjugacy relation $P(t + \theta~(\text{mod}~m\tau), \Lambda^{t/\tau} \sigma) = S_t ( P(\theta, \sigma))$.
\item $\mathcal{P}_1, \dots, \mathcal{P}_m : U \to \mathcal{C}^n$ parameterizes a local unstable manifold and satisfies the conjugacy relations $\mathcal{P}_1 (\Lambda \sigma) = \mathcal{F} (\mathcal{P}_m (\sigma))$ and $\mathcal{P}_j (\Lambda \sigma) = \mathcal{F} (\mathcal{P}_{j-1} (\sigma))$ for $j = 2, \dots, m$.
\item $\mathcal{P} : U \to (\mathcal{C}^n)^m$ parameterizes a local unstable manifold and satisfies the conjugacy relation $\mathcal{P} (\Lambda \sigma) = \mathring{\mathcal{F}} (\mathcal{P}(\sigma))$.
\end{enumerate}
\end{lemma}

\begin{proof}
Assuming Point 1 holds, for $j = 1, \dots, m$, define $\mathcal{P}_j (\sigma) \bydef P ((j-1)\tau, \sigma)$.
By construction, the images of $\mathcal{P}_1, \dots, \mathcal{P}_m$ cover the local unstable manifold of the periodic orbit.
According to the conjugacy relation $P(t + \theta~(\text{mod}~m\tau), \Lambda^{t/\tau} \sigma) = S_t ( P(\theta, \sigma) )$, it follows that
\[
\mathcal{P}_1 (\Lambda \sigma)
= P (0, \Lambda \sigma)
= S_\tau (P((m-1)\tau, \sigma))
= S_\tau (\mathcal{P}_m (\sigma))
= \mathcal{F} (\mathcal{P}_m (\sigma)).
\]
Repeating this argument for $j = 2, \dots, m$, we obtain $\mathcal{P}_j (\Lambda \sigma) = \mathcal{F} (\mathcal{P}_{j-1} (\sigma))$ as desired.

Conversely, if Point 2 holds, then for any $\theta \in \mathbb{R}/m\tau\mathbb{Z}$, define $P(\theta, \sigma) \bydef S_\theta ( \mathcal{P}_1(\Lambda^{-\theta/\tau} \sigma) )$.
From the conjugacy relations satisfied by $\mathcal{P}_j$, for $j = 1, \dots, m$, we have that
\[
P((j-1)\tau, \sigma) = S_{(j-1)\tau} ( \mathcal{P}_1(\Lambda^{-(j-1)} \sigma) ) = \mathcal{F}^{j-1} (\mathcal{P}_1(\Lambda^{-(j-1)} \sigma)) = \mathcal{P}_j(\sigma).
\]
Thus, by construction, the image of $P$ covers the local unstable manifold of the periodic orbit.
Now, given $t \in \mathbb{R}$ and $\theta \in \mathbb{R}/m\tau\mathbb{Z}$, we have
\[
P(t + \theta~(\text{mod}~m\tau), \Lambda^{t/\tau} \sigma)
= S_{t + \theta} ( \mathcal{P}_1 (\Lambda^{-(t + \theta)/\tau} \Lambda^{t/\tau} \sigma) )
= S_t( S_\theta ( \mathcal{P}_1 (\Lambda^{-\theta/\tau} \sigma) ) )
= S_t( P(\theta, \sigma) ).
\]

Lastly, Point 2 is equivalent to Point 3 from the equality
\[
\mathcal{P} (\Lambda \sigma)
= \begin{pmatrix}
\mathcal{P}_1 (\Lambda \sigma) \\ \mathcal{P}_2 (\Lambda \sigma) \\ \vdots \\ \mathcal{P}_m (\Lambda \sigma)
\end{pmatrix}
= \begin{pmatrix}
\mathcal{F} ( \mathcal{P}_m(\sigma) ) \\ \mathcal{F} ( \mathcal{P}_1(\sigma) ) \\ \vdots \\ \mathcal{F} ( \mathcal{P}_{m-1}(\sigma) )
\end{pmatrix}
= \mathring{\mathcal{F}} (\mathcal{P}(\sigma)). \qedhere
\]
\end{proof}

As we computed the $m\tau$-periodic orbit of the DDE \eqref{eq:dde} by working with the DDS \eqref{eq:unrolled_dds}, we shall retrieve its unstable manifold in the framework of the DDS \eqref{eq:unrolled_dds}.
From the equivalence between Point 1 and Point 3 of Lemma \ref{lem:equivalence_param}, the resulting unstable manifold can be expressed in the context of the DDE \eqref{eq:dde}.
Henceforth, we consider $\mathcal{P}$, as described in Point 3 of Lemma \ref{lem:equivalence_param}, satisfying the conjugacy relation
\begin{equation}\label{eq:unrolled_conjugacy}
\mathcal{P} (\Lambda \sigma) = \mathring{\mathcal{F}} (\mathcal{P}(\sigma)).
\end{equation}

The parameterization $\mathcal{P}$ can always be written as an analytic function on $\mathbb{D}^{n_\textnormal{u}}$, where $\mathbb{D} \bydef \{ z \in \mathbb{C} \, : \, |z| < 1\}$ is the unit open disk in the complex plane.
Indeed, let
\[
[\mathcal{P} (\sigma)](s) = \sum_{|\alpha| \ge 0} \frac{1}{\alpha!} [D_\sigma^\alpha \mathcal{P}(\sigma)]_{\sigma = 0} (s) \sigma^\alpha, \qquad \text{for all } s \in [-\tau, 0], \, \max_{i = 1, \dots, n_\textnormal{u}}|\sigma_i| < \gamma,
\]
where $\gamma > 0$ is the radius of convergence of the series, $| \alpha | = \alpha_1 + \ldots + \alpha_{n_\textnormal{u}}$, $\sigma^\alpha = \sigma_1^{\alpha_1} \times \ldots \sigma_{n_\textnormal{u}}^{\alpha_{n_\textnormal{u}}}$, $\alpha! = \alpha_1 ! \times \ldots \times \alpha_{n_\textnormal{u}}!$ and $D_\sigma^\alpha = D_{\sigma_1}^{\alpha_1} \ldots D_{\sigma_{n_\textnormal{u}}}^{\alpha_{n_\textnormal{u}}}$.
We shall carry on this standard multi-indices notation throughout this article.
Then, the conjugacy relation \eqref{eq:unrolled_conjugacy} yields
\[
\mathcal{P}(0) = \mathring{\mathcal{F}}( \mathcal{P}(0) ) \qquad \text{and} \qquad
\lambda_l [D_\sigma^{e_l} \mathcal{P}(\sigma)]_{\sigma = 0} = [D \mathring{\mathcal{F}}(  \mathcal{P}(0) )] [D_\sigma^{e_l} \mathcal{P}(\sigma)]_{\sigma = 0}, \quad l = 1, \dots, n_\textnormal{u},
\]
where $( e_l )_k$ is the \emph{Kronecker delta}, that is $( e_l )_k = 1$ if $k = l$ and $0$ otherwise.
Thus, the zero-th order (i.e. $|\alpha|=0$) Taylor coefficient is a fixed-point of the DDS \eqref{eq:unrolled_dds} and each order 1 (i.e. each $|\alpha|=1$) Taylor coefficient corresponds to an unstable eigenvector. Moreover, from the Fa\`{a} di Bruno formula \cite{Constantine1996}, we obtain
\begin{align*}
&\frac{1}{\alpha!} [D_\sigma^\alpha \mathring{\mathcal{T}} ( \mathcal{P}(\Lambda \sigma), \mathcal{P}(\sigma) )]_{\sigma = 0} \\
&= [D_1 \mathring{\mathcal{T}} ( \mathcal{P}(0), \mathcal{P}(0) )] \Big( \frac{\lambda^\alpha}{\alpha!} [D_\sigma^\alpha \mathcal{P}(\sigma)]_{\sigma = 0} \Big) + [D_2 \mathring{\mathcal{T}} ( \mathcal{P}(0), \mathcal{P}(0) )] \Big( \frac{1}{\alpha!} [D_\sigma^\alpha \mathcal{P}(\sigma)]_{\sigma = 0} \Big) + \mathcal{R}_\alpha ( \mathcal{P}(\Lambda \sigma), \mathcal{P}(\sigma) ),
\end{align*}
where $\mathcal{R}_\alpha ( \mathcal{P}(\Lambda \sigma), \mathcal{P}(\sigma) )$ only depends on the lower order Taylor coefficients $\frac{1}{\beta!} [D_\sigma^\beta \mathcal{P}(\sigma)]_{\sigma = 0}$ for $|\beta| < |\alpha|$. From this equality and the conjugacy relation \eqref{eq:unrolled_conjugacy}, it follows that the higher order Taylor coefficients are explicitly given by the formula
\begin{equation}\label{eq:rec_rel}
\frac{1}{\alpha!} [D_\sigma^\alpha \mathcal{P}(\sigma)]_{\sigma = 0} = \left( \lambda^\alpha I - \lambda^\alpha D_1 \mathring{\mathcal{T}}( \mathcal{P}(0), \mathcal{P}(0) ) - D_2 \mathring{\mathcal{T}}( \mathcal{P}(0), \mathcal{P}(0) ) \right)^{-1} \mathcal{R}_\alpha ( \mathcal{P}(\Lambda \sigma), \mathcal{P}(\sigma) ), \qquad |\alpha| \ge 2.
\end{equation}
Choosing the scaling of the eigenvectors to be $\gamma \left[\frac{\partial^{e_l}}{\partial \sigma^{e_l}} \mathcal{P}(\sigma)\right]_{\sigma = 0}$, for $l = 1, \dots, n_\textnormal{u}$, the recurrence relation \eqref{eq:rec_rel} generates the series
\[
\sum_{|\alpha| \ge 0} \frac{\gamma^{|\alpha|}}{\alpha!} [D_\sigma^\alpha \mathcal{P}(\sigma)]_{\sigma = 0}(s) \sigma^\alpha, \qquad \text{for all } s \in [-\tau, 0], \, \sigma \in \mathbb{D}^{n_\textnormal{u}},
\]
which is equal to $\mathcal{P}(\gamma \sigma)$. Hence, we have obtained a parameterization of the local unstable manifold, satisfying the conjugacy relation \eqref{eq:unrolled_conjugacy}, whose radius of convergence is $1$ as initially claimed.

\begin{remark}[Image of the parameterization and covering of the local unstable manifold]
\label{rem:param_image}
The parameterization of the unstable manifold may be complex-valued in the event of complex eigenvalues.
The unstable manifold is real and is covered by the image of the parameterization intersected with $C([-\tau, 0], \mathbb{R}^n)$.
The set of such values in the domain $\mathbb{D}^{n_\textnormal{u}}$ which yields a real image is entirely traceable from the nature of the eigenvalues.
The general case of $n_\textnormal{u}$ unstable Floquet multipliers may be deduced from the three following cases:
\begin{itemize}
\item there is a single unstable eigenvalue $\lambda^m > 1$, then $\mathcal{P} (\sigma) \in C([-\tau, 0], \mathbb{R}^n)$ for all $\sigma \in (-1, 1)$.

\item there are two complex conjugate unstable eigenvalues $\lambda, \lambda^* \in \mathbb{C}$ (where the star symbolizes the complex conjugacy), then $\mathcal{P} ( \sigma, \sigma^*) \in C([-\tau, 0], \mathbb{R}^n)$ for all $\sigma \in \mathbb{D}$.

\item there is a single unstable eigenvalue $\lambda^m < -1$, then, without loss of generality, $\lambda = |\lambda^m|^{\frac{1}{m}} e^{i\frac{\pi}{m}}$ and the corresponding eigenvector is $(v_1, \lambda^{-1} v_2, \dots, \lambda^{-(m-1)} v_m)$ with $v_j \in C([-\tau, 0], \mathbb{R}^n)$ for $j = 1, \dots, m$ (cf. Lemma \ref{lem:equivalence_eig}). It follows that $\mathcal{P}_j ( e^{i\frac{(j-1) \pi}{m}} \sigma ) \in C([-\tau, 0], \mathbb{R}^n)$ for all $\sigma \in (-1, 1)$ and $j = 1, \dots, m$.
In this case, the manifold is non-orientable, it is topologically equivalent to a Möbius strip: according to the conjugacy relation \eqref{eq:unrolled_conjugacy}, we have that $\mathring{\mathcal{F}}^m (\mathcal{P}( |\lambda^m|^{-1} \sigma )) = \mathcal{P}(\lambda^m |\lambda^m|^{-1} \sigma) = \mathcal{P}(- \sigma)$ for all $\sigma \in \mathbb{D}$.
\end{itemize}
\end{remark}

While possible to construct a zero-finding problem in the same vein as in Section \ref{sec:periodic_orbit}, we favour generating the parameterization of the local unstable manifold via an explicit recurrence relation.
First, for $j = 1, \dots, m$, we expand $\mathcal{P}_j$ as the Taylor-Chebyshev series
\[
[\mathcal{P}_j (\sigma)](s(t)) = \sum_{|\alpha| \ge 0} \left( \{\mathbf{p}_j\}_{0,\alpha} + 2 \sum_{\beta \ge 1} \{\mathbf{p}_j\}_{\beta, \alpha} T_\beta (t) \right) \sigma^\alpha, \qquad \text{for all } t \in [-1, 1], \, \sigma \in \mathbb{D}^{n_\textnormal{u}},
\]
where $s(t) \bydef \frac{\tau}{2}(t-1)$ scales $[-1, 1]$ to $[-\tau, 0]$.
Also, we define $\{ \mathbf{p}_j \}_\alpha \bydef \{ \{ \mathbf{p}_j \}_{\beta, \alpha} \}_{\beta \ge 0} \in (\ell^1_\nu)^n$ for $|\alpha| \ge 0$.
In other words, $\{ \mathbf{p} \}_\alpha = (\{ \mathbf{p}_1 \}_\alpha, \dots, \{ \mathbf{p}_m \}_\alpha)$ denotes the sequences of Chebyshev coefficients corresponding to analytic functions in $(C([-1, 1], \mathbb{R}^n))^m$.
The analyticity of $\mathcal{P}_1, \dots, \mathcal{P}_m$ implies that there exists $\nu > 1$ such that their sequence of Taylor-Chebyshev coefficients solving the conjugacy relation \eqref{eq:unrolled_conjugacy} belongs to
\[
\ell^1 (\ell^1_\nu) \bydef \left\{ \mathbf{a} \in (\ell^1_\nu)^{(\mathbb{N} \cup \{0\})^{n_\textnormal{u}}} \, : \, | \mathbf{a} |_{\ell^1 (\ell^1_\nu)} \bydef \sum_{|\alpha| \ge 0} | \{ \mathbf{a} \}_\alpha |_{\ell^1_\nu} = \sum_{|\alpha| \ge 0} \left( |\{\mathbf{a}\}_{0,\alpha}| + 2 \sum_{\beta \ge 1} |\{\mathbf{a}\}_{\beta, \alpha}| \nu^\beta \right) < \infty \right\}.
\]
This is a Banach algebra with the discrete convolution product
\[
\mathbf{a} \circledast \mathbf{b} \bydef \left\{ \sum_{|\beta| = 0}^{|\alpha|} \{\mathbf{a}\}_{\alpha - \beta} * \{\mathbf{b}\}_\beta \right\}_{|\alpha| \ge 0}, \qquad \text{for all } \mathbf{a}, \mathbf{b} \in \ell^1(\ell^1_\nu),
\]
which corresponds to the Cauchy product for Taylor series whose coefficients are Chebyshev series. Since $f$ in the DDE \eqref{eq:dde} is polynomial, there is a natural mapping, denoted with the same symbol, $f : (\ell^1(\ell^1_\nu))^n \times (\ell^1(\ell^1_\nu))^n \to (\ell^1(\ell^1_\nu))^n$ defined by replacing products of Taylor-Chebyshev series with the convolution product $\circledast$.
Once again, we believe that our abuse of notation will not lead to confusion; the algebraic rules defining the polynomial $f$ are unambiguously deduced from its arguments.

\begin{remark}
Let $\mathbf{a} = (\mathbf{a}_1, \dots, \mathbf{a}_m) \in ((\ell^1(\ell^1_\nu))^n)^m$ with $a_j = ((a_j)_1, \dots, (a_j)_n)$ for $j=1, \dots,m$. For convenience, we denote $\{ \mathbf{a} \}_\alpha = (\{ \mathbf{a}_1\}_\alpha, \dots, \{ \mathbf{a}_m \}_\alpha) \in ((\ell^1_\nu)^n)^m$ and $\{ \mathbf{a}_j \}_\alpha = (\{ (\mathbf{a}_j)_1 \}_\alpha, \dots, \{ (\mathbf{a}_j)_n \}_\alpha) \in (\ell^1_\nu)^n$ for $j=1, \dots, m$.
\end{remark}

We draw the reader's attention to the previously derived recurrence relation \eqref{eq:rec_rel}.
The implication of this sequence of equations is that, having computed the bundle of the unstable eigenspace (e.g. computed from Section \ref{sec:eigendecomposition}) over the periodic orbit (e.g. computed from Section \ref{sec:periodic_orbit}), the higher order terms of the parameterization are simply obtained by recursively solving linear equations.
To be precise, for $j = 1, \dots, m$, fix $\{ \mathbf{p}_j \}_0 = \mathbf{c}_j \in (\ell^1_\nu)^n$ and $\{ \mathbf{p}_j \}_{e_l} = \mathbf{v}_{j,l} \in (\ell^1_\nu)^n$, for $l=1, \dots,n_{\textnormal{u}}$, where $\mathbf{c} = (\mathbf{c}_1, \dots, \mathbf{c}_m)$ is a zero of $\mathbf{F}_\circ$, or $\mathbf{F}_{\circ,\textnormal{elem}}$, and the unstable eigenvectors $\mathbf{v}_{1,l}, \dots, \mathbf{v}_{1,l}$ are given by \eqref{eq:eigenvector} for $l = 1, \dots, n_\textnormal{u}$.
Then, the higher order Taylor-Chebyshev coefficients $\{ \mathbf{p} \}_\alpha$ of the parameterization are given by solving recursively the linear systems
\begin{equation}\label{eq:recurrence_relation}
(\lambda^\alpha \mathbf{I} - \lambda^\alpha \mathbf{K}_1 (\tau, \mathbf{c}) - \mathbf{K}_2 (\tau, \mathbf{c})) \{ \mathbf{p} \}_\alpha = \mathbf{R}_\alpha (\tau, \mathbf{L}(\lambda, \mathbf{p}), \mathbf{p} ), \qquad |\alpha| \ge 2,
\end{equation}
with
\begin{align*}
\mathbf{R}_\alpha (\tau, \mathbf{L}(\lambda, \mathbf{p}), &\mathbf{p} ) \\
&= \begin{pmatrix}
\mathbf{S}\left(\frac{\tau}{2}
\left\{
f( \mathbf{L}(\boldsymbol{\pi}_\textnormal{T}^{ |\alpha|-1}\mathbf{p}_1), \boldsymbol{\pi}_\textnormal{T}^{|\alpha|-1}\mathbf{p}_m ) - f(0, 0) - [Df(0, 0)]
\begin{pmatrix}
\mathbf{L}(\boldsymbol{\pi}_\textnormal{T}^{|\alpha|-1}\mathbf{p}_1) \\
\boldsymbol{\pi}_\textnormal{T}^{|\alpha|-1}\mathbf{p}_m
\end{pmatrix}
\right\}_\alpha
\right) \\
\mathbf{S}\left(\frac{\tau}{2}
\left\{
f( \mathbf{L}(\boldsymbol{\pi}_\textnormal{T}^{ |\alpha|-1}\mathbf{p}_2), \boldsymbol{\pi}_\textnormal{T}^{|\alpha|-1}\mathbf{p}_1 ) - f(0, 0) - [Df(0, 0)]
\begin{pmatrix}
\mathbf{L}(\boldsymbol{\pi}_\textnormal{T}^{|\alpha|-1}\mathbf{p}_2) \\
\boldsymbol{\pi}_\textnormal{T}^{|\alpha|-1}\mathbf{p}_1
\end{pmatrix}
\right\}_\alpha
\right) \\
\vdots \\
\mathbf{S}\left(\frac{\tau}{2}
\left\{
f( \mathbf{L}(\boldsymbol{\pi}_\textnormal{T}^{ |\alpha|-1}\mathbf{p}_m), \boldsymbol{\pi}_\textnormal{T}^{|\alpha|-1}\mathbf{p}_{m-1} ) - f(0, 0) - [Df(0, 0)]
\begin{pmatrix}
\mathbf{L}(\boldsymbol{\pi}_\textnormal{T}^{|\alpha|-1}\mathbf{p}_m) \\
\boldsymbol{\pi}_\textnormal{T}^{|\alpha|-1}\mathbf{p}_{m-1}
\end{pmatrix}
\right\}_\alpha
\right)
\end{pmatrix}
\end{align*}
and where
\begin{subequations}
\begin{align}
\begin{split}
\{ \boldsymbol{\pi}_\textnormal{T}^{N'} \mathbf{a} \}_{\alpha,\beta} &\bydef
\begin{cases}
\{\mathbf{a}\}_{\alpha,\beta}, & |\beta| \le N' \\
0, & |\beta| > N',
\end{cases} \qquad \text{for all } N' \in \mathbb{N} \cup \{0\}, \mathbf{a} \in \ell^1(\ell^1_\nu),
\end{split} \\
\begin{split}\label{eq:def_L}
\mathbf{L}(\lambda, \mathbf{a}) &\bydef \{ \lambda^\alpha \{ \mathbf{a} \}_\alpha \}_{|\alpha| \ge 0}, \qquad \text{for all } \mathbf{a} \in \ell^1(\ell^1_\nu),
\end{split} \\
\begin{split}
\mathbf{K}_1 (\tau, \mathbf{c}) &\bydef
\begin{pmatrix}
\mathbf{S} [\tfrac{\tau}{2} D_1 f(\mathbf{c}_1, \mathbf{c}_m)] & & & 0 \\
& \mathbf{S} [\tfrac{\tau}{2} D_1 f(\mathbf{c}_2, \mathbf{c}_1)] & & \\
& & \ddots & \\
0 & & & \mathbf{S} [\tfrac{\tau}{2} D_1 f(\mathbf{c}_m, \mathbf{c}_{m-1})]
\end{pmatrix},
\end{split} \\
\begin{split}
\mathbf{K}_2 (\tau, \mathbf{c}) &\bydef
\begin{pmatrix}
0 & \cdots & 0 & \mathbf{E} + \mathbf{S} [\tfrac{\tau}{2} D_2 f(\mathbf{c}_1, \mathbf{c}_m)] \\
\mathbf{E} + \mathbf{S} [\tfrac{\tau}{2} D_2 f(\mathbf{c}_2, \mathbf{c}_1)] & & 0 & 0 \\
& \ddots & & \vdots \\
0 & & \mathbf{E} + \mathbf{S} [\tfrac{\tau}{2} D_2 f(\mathbf{c}_m, \mathbf{c}_{m-1})] & 0
\end{pmatrix}.
\end{split}
\end{align}
\end{subequations}
Again, we remark that the decay of the higher order coefficients is controlled 
by fixing the length of the associated eigenvectors.

Let us review how such a parameterization of the local unstable manifold for the DDE \eqref{eq:dde} yields a parameterization of the local unstable manifold for the DDE \eqref{eq:dde_elem}.
Suppose the periodic orbit of the auxiliary polynomial DDE \eqref{eq:dde} represents a periodic orbit of the DDE \eqref{eq:dde_elem}.
Then, Corollary 1 in \cite{Henot2021} guarantees that the non-resonance property of the Floquet multipliers holds for both DDEs \eqref{eq:dde_elem} and \eqref{eq:dde}.
Furthermore, Point 3 of Theorem 3.1 in \cite{Henot2021} implies that the unstable manifold for both DDEs \eqref{eq:dde_elem} and \eqref{eq:dde} coincide.

\subsection{Numerical considerations}
\label{sec:unstable_manifold_num}

An approximation of the parameterization of the local unstable manifold can be obtained as follows.
Consider an approximate zero $\bar{\tau} > 0$ and $\bar{\mathbf{c}} = (\bar{\mathbf{c}}_1, \dots, \bar{\mathbf{c}}_m) \in \boldsymbol{\pi}^N ((\ell^1_\nu)^n)^m$ of $\mathbf{F}_\circ$, or $\mathbf{F}_{\circ, \textnormal{elem}}$ (cf. Section \ref{sec:periodic_orbit_num}).
Moreover, for $l = 1, \dots, n_\textnormal{u}$, consider the approximate unstable eigenvalue $\bar{\lambda}^m_l$ and associated approximate eigenvectors $\bar{\mathbf{v}}_{1,l}, \dots, \bar{\mathbf{v}}_{m,l}$ given by \eqref{eq:eigenvector_num} (cf. Section \ref{sec:eigendecomposition_num}).
Let $\bar{\lambda} = (\bar{\lambda}_1, \dots, \bar{\lambda}_{n_\textnormal{u}})$.
For $j=1, \dots, m$, set $\{\bar{\mathbf{p}}_j\}_0 = \bar{\mathbf{c}}_j$ and $\{\bar{\mathbf{p}}_j\}_{e_l} = \bar{\mathbf{v}}_{j,l}$ for $l = 1, \dots, n_\textnormal{u}$.
Then, we have that $\mathbf{p}$, generated by the recurrence relation \eqref{eq:recurrence_relation}, is approximated by $\bar{\mathbf{p}}$, generated by the recurrence relation
\begin{equation}\label{eq:recurrence_relation_num}
\boldsymbol{\pi}^N(\bar{\lambda}^\alpha \mathbf{I} - \bar{\lambda}^\alpha \mathbf{K}_1 (\bar{\tau}, \bar{\mathbf{c}}) - \mathbf{K}_2(\bar{\tau}, \bar{\mathbf{c}})) \boldsymbol{\pi}^N\{ \bar{\mathbf{p}} \}_\alpha = \boldsymbol{\pi}^N \mathbf{R}_\alpha (\bar{\tau}, \mathbf{L}(\bar{\lambda}, \bar{\mathbf{p}}), \bar{\mathbf{p}} ), \qquad |\alpha| \ge 2.
\end{equation}

We conclude this section by computing the parameterization of the local unstable manifold associated with the $m\tau$-periodic orbit obtained for the cubic Ikeda equation \eqref{eq:cubic_ikeda} (cf. Section \ref{sec:cubic_ikeda_po}).

\subsection{Example: unstable manifold for the cubic Ikeda equation}
\label{sec:cubic_ikeda_param}

For the cubic Ikeda equation \eqref{eq:cubic_ikeda}, $f$, given in \eqref{eq:f_cubic_ikeda}, is polynomial and acts on the Taylor-Chebyshev coefficients as $f (\mathbf{a}, \mathbf{b}) = \mathbf{b} - \mathbf{b}^{\circledast 3}$ where $\mathbf{b}^{\circledast k} \bydef \underbrace{\mathbf{b} \circledast \dots \circledast \mathbf{b}}_{k \text{ times}}$.

Consider $\bar{\tau}_\textnormal{init} > 0$ and $\bar{\mathbf{c}}_\textnormal{init} \in \boldsymbol{\pi}^N (\ell^1_\nu)^m$ computed in Section \ref{sec:cubic_ikeda_po}; in particular, here $m=8$ and $N = 30$.
Recall from Section \ref{sec:cubic_ikeda_eigendecomposition} that the unstable manifold is expected to be $1$-dimensional (i.e. $n_\textnormal{u}=1$) since there seems to be a single unstable Floquet multiplier $\bar{\mu} \approx -4.624622928960324$.
We also point out to the reader that since $\bar{\mu} < -1$, the unstable manifold is a topological Möbius strip (see also Remark \ref{rem:param_image}).

In Section \ref{sec:cubic_ikeda_eigendecomposition}, we chose arbitrarily one of the $m$-th root $\bar{\lambda} = |\bar{\mu}|^{\frac{1}{m}} e^{i\frac{\pi}{m}}$ and retrieved
the unstable eigenvectors $\bar{\mathbf{v}}_1, \dots, \bar{\mathbf{v}}_m$ associated with $\bar{\lambda}^m$.
For $j = 1, \dots, m$, we set $\{\bar{\mathbf{p}}_j\}_0 = (\bar{\mathbf{c}}_\textnormal{init})_j$ and $\{\bar{\mathbf{p}}_j\}_1 = \bar{\mathbf{v}}_j$.
Then, the recurrence relation \eqref{eq:recurrence_relation_num} reads
\[
\boldsymbol{\pi}^N (\bar{\lambda}^\alpha \mathbf{I} - \mathbf{K}_2(\bar{\tau}_\textnormal{init}, \bar{\mathbf{c}}_\textnormal{init})) \boldsymbol{\pi}^N \{ \bar{\mathbf{p}} \}_\alpha = \boldsymbol{\pi}^N 
\begin{pmatrix}
\mathbf{S} \left(-\frac{\bar{\tau}_\textnormal{init}}{2} \{(\boldsymbol{\pi}_\textnormal{T}^{\alpha-1}\bar{\mathbf{p}}_m)^{\circledast 3}\}_\alpha\right) \\
\mathbf{S} \left(-\frac{\bar{\tau}_\textnormal{init}}{2} \{(\boldsymbol{\pi}_\textnormal{T}^{\alpha-1}\bar{\mathbf{p}}_1)^{\circledast 3}\}_\alpha\right) \\
\vdots \\
\mathbf{S} \left(-\frac{\bar{\tau}_\textnormal{init}}{2} \{(\boldsymbol{\pi}_\textnormal{T}^{\alpha-1}\bar{\mathbf{p}}_{m-1})^{\circledast 3}\}_\alpha\right)
\end{pmatrix}, \qquad \alpha \ge 2,
\]
where
\begin{align*}
&\mathbf{K}_2(\bar{\tau}_\textnormal{init}, \bar{\mathbf{c}}_\textnormal{init}) \\
&=
\begin{pmatrix}
0 & \cdots & 0 & \mathbf{E} + \mathbf{S} (  \tfrac{\bar{\tau}_\textnormal{init}}{2}[\mathbf{I} - 3\mathbf{M}_{(\bar{\mathbf{c}}_\textnormal{init})_m^{*2}} ] ) \\
\mathbf{E} + \mathbf{S} (\tfrac{\bar{\tau}_\textnormal{init}}{2} [\mathbf{I} - 3\mathbf{M}_{(\bar{\mathbf{c}}_\textnormal{init})_1^{*2}}] ) & & 0 & 0 \\
& \ddots & & \vdots \\
0 & & \mathbf{E} + \mathbf{S} (\tfrac{\bar{\tau}_\textnormal{init}}{2} [\mathbf{I} - 3\mathbf{M}_{(\bar{\mathbf{c}}_\textnormal{init})_{m-1}^{*2}}] ) & 0
\end{pmatrix}.
\end{align*}

\begin{figure}
\centering
\begin{subfigure}[b]{0.32\textwidth}
\includegraphics[width=\textwidth]{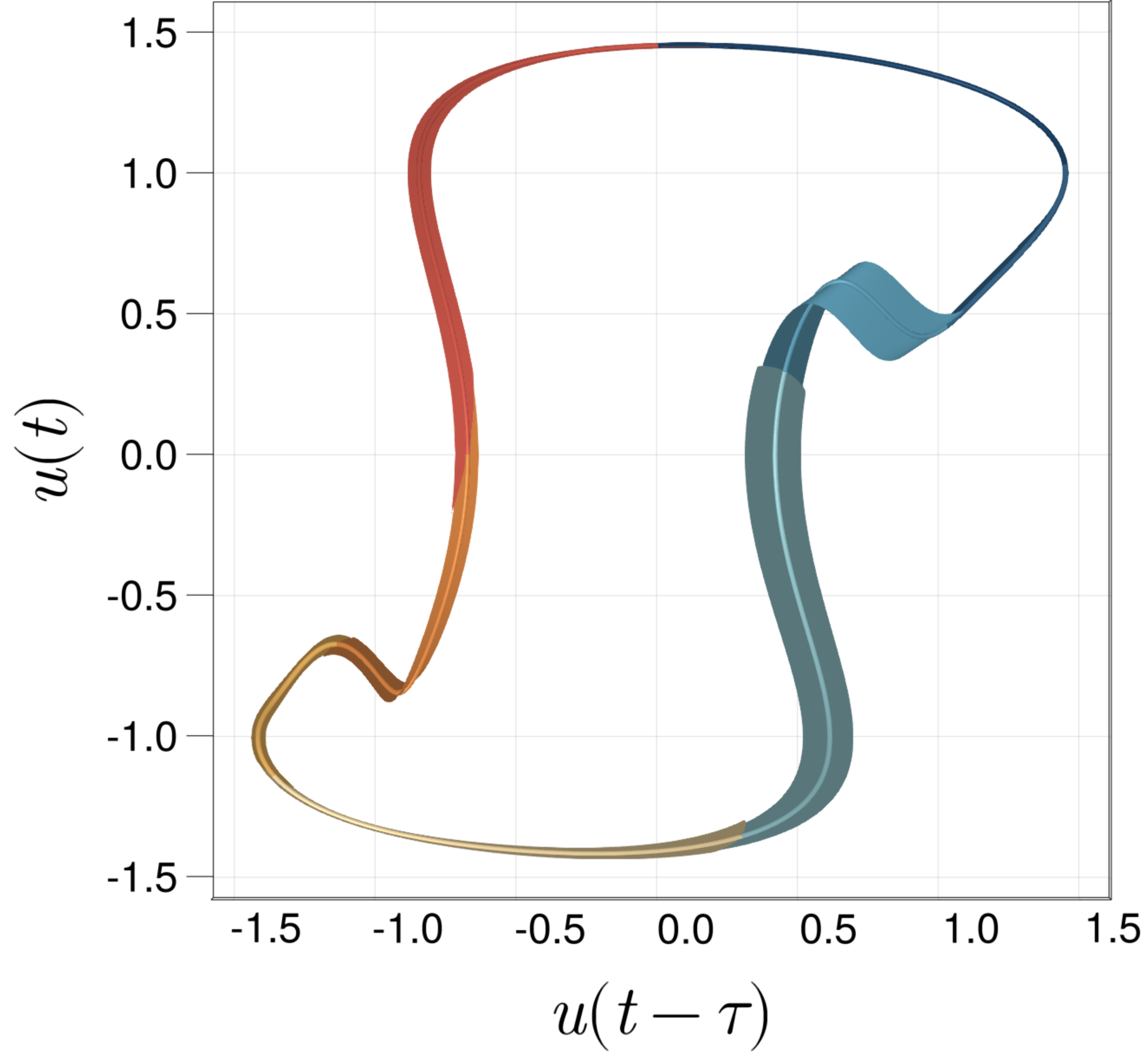}
\caption{}
\end{subfigure}
\hspace{0.03\textwidth}
\begin{subfigure}[b]{0.29\textwidth}
\includegraphics[width=\textwidth]{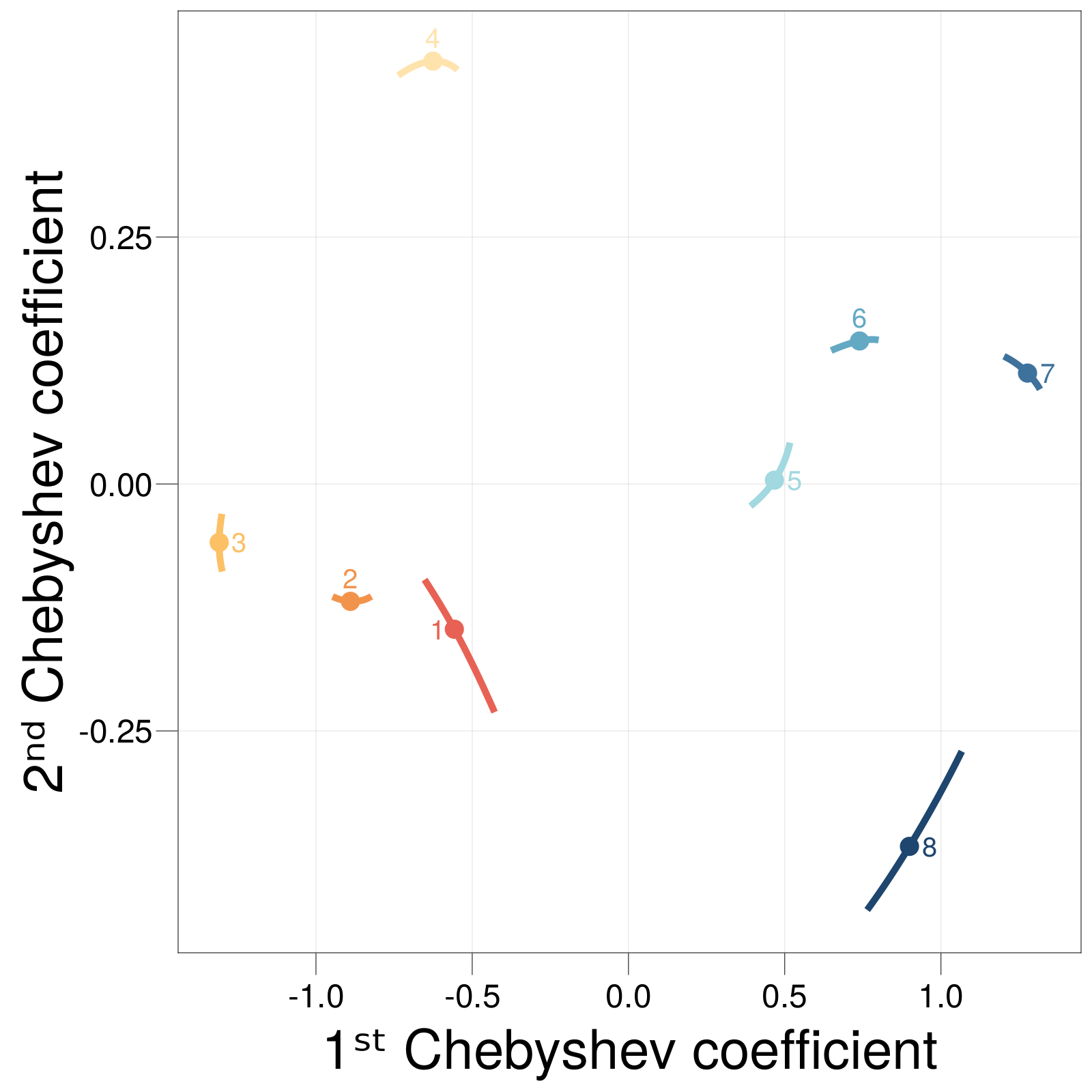}
\caption{}
\end{subfigure}
\hspace{0.03\textwidth}
\begin{subfigure}[b]{0.3\textwidth}
\includegraphics[height=\textwidth]{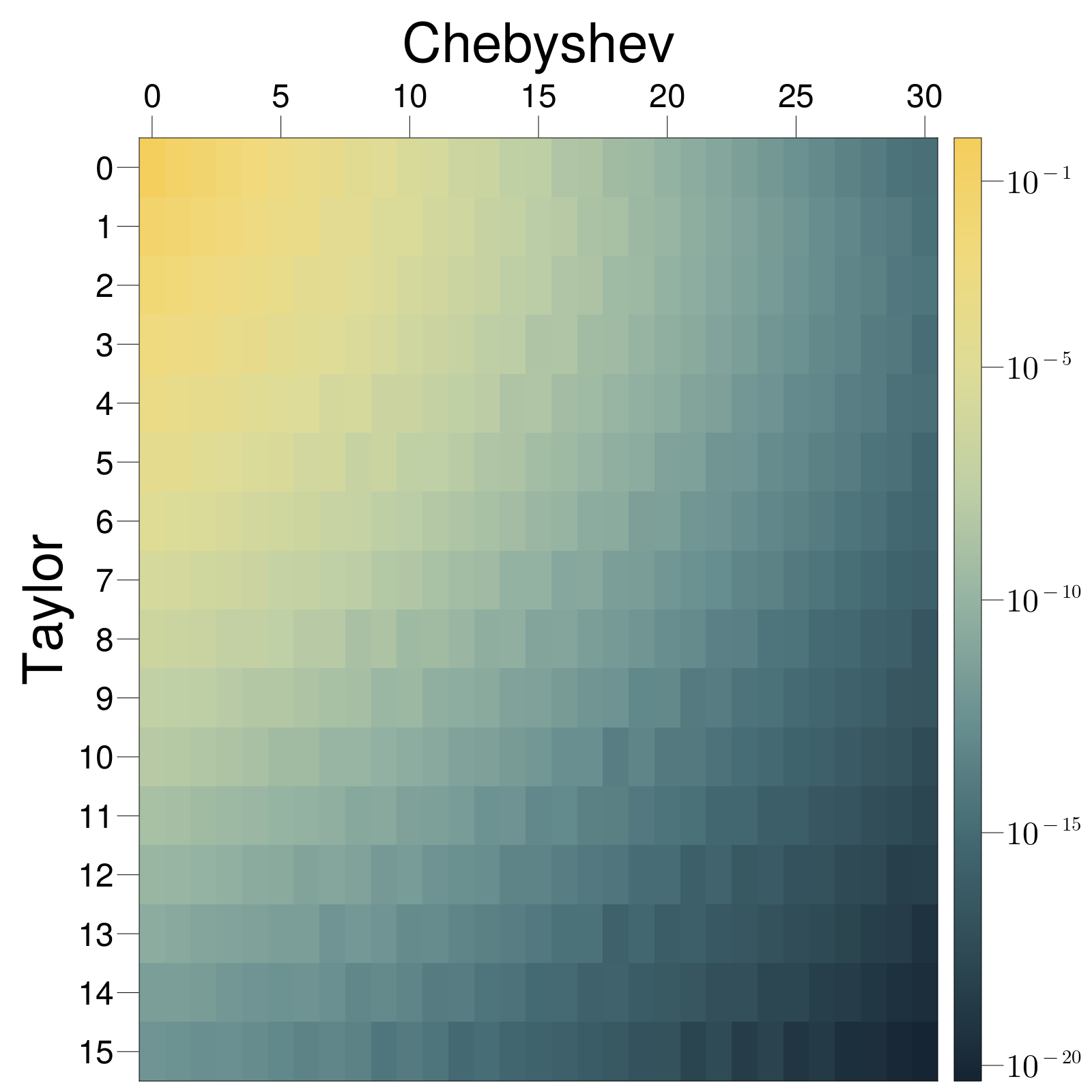}
\caption{}
\end{subfigure}
\caption{(a) Local unstable manifold of the $m\tau$-periodic orbit shown on Figure \ref{fig:cubic_ikeda_po} for the cubic Ikeda equation.
(b) Representation in \emph{Chebyshev space} of the parameterization of the local unstable manifold shown in (a). The dots correspond to the periodic orbit.
The numbers indicate the labelling of the $m$ pieces; the numbering follows the successive iterations of the time-$\tau$ map.
(c) Average $\{ m^{-1}\sum_{j=1}^m |\{ \bar{\mathbf{p}}_j \}_{\alpha,\beta}| \}_{\alpha,\beta \ge 0}$ of the sequences of Taylor-Chebyshev coefficients of the parameterization of the local unstable manifold shown in (a).}
\label{fig:cubic_ikeda_param}
\end{figure}

Each linear system is set on $\boldsymbol{\pi}^N (\ell^1_\nu)^m \simeq \mathbb{C}^{m(N+1)} = \mathbb{C}^{8\times31} = \mathbb{C}^{248}$.
We choose the Taylor truncation order to be $N' = 15$, thus the parameterization has a total of $m (N+1) (N'+1) = 8 \times 31 \times 16 = 3,968$ Taylor-Chebyshev coefficients; see Figure \ref{fig:cubic_ikeda_param}.

%%%%%%%%%%%%%%%%%%%%%%
%% CONNECTING ORBIT %%
%%%%%%%%%%%%%%%%%%%%%%

\section{Computation of the transverse homoclinic orbit}\label{sec:homoclinic_orbit}
%!TEX root = poincare_scenario_dde.tex

We are after the transverse intersection of the stable and unstable manifolds of an $m\tau$-periodic orbit of the DDE \eqref{eq:dde_elem}.
In the context of hyperbolic periodic orbits, the center manifold is generated by the translation invariance of the periodic solution. Hence, given a Poincar\'{e} section $\Sigma \subset C([-\tau, 0],\mathbb{R})$, the stable and unstable manifold intersect transversely whenever there exists a point $p \in W_\textnormal{s} \cap W_\textnormal{u} \cap \Sigma$ such that $T_p (W_\textnormal{s} \cap \Sigma) \oplus T_p (W_\textnormal{u} \cap \Sigma) = \Sigma$.

In Section \ref{sec:unstable_manifold}, we put our hands on a local unstable manifold by finding a parameterization via a conjugacy relation between the nonlinear and linear flows.
The stable manifold however cannot be swayed by such conjugacy due to its infinite-dimensional nature.
Still, we do know that there exists a differentiable graph of a local stable manifold of the periodic orbit (e.g. see \cite{Hale1977}).

As discussed in Section \ref{sec:eigendecomposition}, a hyperbolic periodic orbit of the DDE \eqref{eq:dde_elem} has $n$ centre directions when considered as a periodic orbit of the auxiliary polynomial DDE \eqref{eq:dde}.
This is clear when the DDE \eqref{eq:dde_elem} is already polynomial, since then the DDEs \eqref{eq:dde_elem} and \eqref{eq:dde} coincide and $n = 1$.
On the other hand, when the DDE \eqref{eq:dde_elem} has elementary nonlinearities, then $n = 1 + d$ with $d$ being the number of appended coordinates to obtain the auxiliary polynomial DDE \eqref{eq:dde} (recall the construction of $f$ given in \eqref{eq:f_poly}).
The $d$ additional coordinates introduce $d$ center directions which do not pertain to the original DDE \eqref{eq:dde_elem}.

Consequently, there exists $U \subset \mathcal{C}^n$, with $\textnormal{codim}\, U = n + n_\textnormal{u}$, such that the local graph of the stable manifold is the image of $Q \in C^1 ((\mathbb{R}/m\tau\mathbb{Z}) \times U, \mathcal{C}^n)$ (the nature of the stable manifold justifies writing $\mathcal{C}^n$ instead of just $C([-\tau, 0], \mathbb{R}^n)$).

The following lemma characterizes the correspondence between the transverse intersection of the stable and unstable manifolds for the DDE \eqref{eq:dde} and the DDS \eqref{eq:dds}.

\begin{lemma}\label{lem:equivalence_transversality}
Let $\tau > 0$, $m \in \mathbb{N}$ and $\phi \in C([-\tau, 0], \mathbb{R}^n)$.
The following statements are equivalent:
\begin{enumerate}
\item There exists a transverse homoclinic orbit of the DDE \eqref{eq:dde} joining the hyperbolic (with non-resonant unstable Floquet multipliers) $m\tau$-periodic orbit $t \in \mathbb{R}/m\tau\mathbb{Z} \mapsto S_t (\phi)$.
\item There exists a transverse ``heteroclinic" orbit of the DDS \eqref{eq:dds} joining the unstable manifold of the hyperbolic (with non-resonant unstable Floquet multipliers) $m$-periodic orbit $j \in \mathbb{Z}/m\mathbb{Z} \mapsto \mathcal{F}^j (S_{\theta} (\phi))$, for some $\theta \in \mathbb{R}/m\tau\mathbb{Z}$, to the stable manifold of the $m$-periodic orbit $j \in \mathbb{Z}/m\mathbb{Z} \mapsto \mathcal{F}^j (\phi)$.
\end{enumerate}
\end{lemma}

\begin{proof}
Assume that Point 1 holds. Let $n_\textnormal{u} \in \mathbb{N}$ denote the number of unstable Floquet multipliers.
Since the latter are non-resonant, let $P$ denote the parameterization of a local unstable manifold of $t \in \mathbb{R}/m\tau\mathbb{Z} \mapsto S_t (\phi)$ as described in Lemma \ref{lem:equivalence_param}. Let $U$ be an open subset of $C([-\tau, 0], \mathbb{R}^n)$ such that $Q : \mathbb{R}/m\tau\mathbb{Z} \times U \to \mathcal{C}^n$ is the local graph of the stable manifold of the periodic orbit.

Since Point 1 holds, there exists $\theta \in \mathbb{R}/m\tau\mathbb{Z}$, $\sigma \in \mathbb{R}^{n_\textnormal{u}}$, $h \in \mathcal{C}^n$ and $k \in \mathbb{N}\cup\{0\}$ such that $S_{k \tau} (P(\theta, \sigma)) = Q(0, h)$. By hyperbolicity, the transversality of the intersection amounts to $\textnormal{image} \, [D_{\sigma'} S_{k \tau} (P(\theta, \sigma'))]_{\sigma' = \sigma} \cap \textnormal{image} \, [D_{h'} Q (0, h')]_{h' = h} = \{ 0 \}$.

Now, according to the proof of Lemma \ref{lem:equivalence_param}, $\mathcal{P}_j (\sigma') \bydef P((j-1)\tau + \theta \pmod{m\tau}, \sigma')$, for all $j = 1, \dots, m$, yields a parameterization of the local unstable manifold of the $m$-periodic orbit $j \in \mathbb{Z}/m\mathbb{Z} \mapsto \mathcal{F}^j (S_\theta (\phi))$ of the DDS \eqref{eq:dds}. This periodic orbit is nothing more than the $m\tau$-periodic orbit shifted by $\theta$, that is $t \in \mathbb{R}/m\tau\mathbb{Z} \mapsto S_t (S_\theta(\phi))$.
Moreover, $\mathcal{Q}_j (h') \bydef Q((j-1)\tau, h')$, for all $h' \in U$ and $j = 1, \dots, m$, gives the local graph of the stable manifold of the $m$-periodic orbit $j \in \mathbb{Z}/m\mathbb{Z} \mapsto \mathcal{F}^j (\phi)$. Of course, this periodic orbit coincides exactly with $t \in \mathbb{R}/m\tau\mathbb{Z} \mapsto S_t (\phi)$.

Moreover, for any $\sigma' \in \mathbb{C}^{n_\textnormal{u}}$,
\[
\mathcal{F}^k (\mathcal{P}_1 (\sigma')) = S_{k \tau} (\mathcal{P}_1 (\sigma')) = S_{k \tau} (P(\theta, \sigma')).
\]
In particular, we have $\mathcal{F}^k (\mathcal{P}_1 (\sigma)) = S_{k \tau} (P(\theta, \sigma)) = Q(0, h) = \mathcal{Q}_1(h)$, thus the homoclinic orbit of the DDE \eqref{eq:dde} is equivalent to an ``heteroclinic" orbit of the DDS \eqref{eq:dds}.
The intersection is transverse since $D \mathcal{Q}_1 (h) = [D_{h'} Q(0, h')]_{h' = h}$ and, from the previous equality,
\[
[D_{\sigma'} \mathcal{F}^k (\mathcal{P}_0 (\sigma'))]_{\sigma' = \sigma}
=
[D_{\sigma'} S_{k \tau} (P(\theta, \sigma'))]_{\sigma' = \sigma}. \qedhere
\]
\end{proof}

The previous lemma highlights that a transverse homoclinic orbit of the DDE \eqref{eq:dde} may look like a transverse ``heteroclinic" orbit of the DDS \eqref{eq:dds} whenever $\theta \ne 0$. We used quotation marks to emphasize that truly the connection is homoclinic and not heteroclinic. This is merely an artefact due to the time-$\tau$ map $\mathcal{F}$.

\subsection{Connection of the invariant manifolds}
\label{sec:connection}

Suppose $c = (c_1, \dots, c_m) \in (\mathcal{C}^n)^m$ is a fixed-point of the DDS \eqref{eq:unrolled_dds}. Let $c_\theta = ( (c_\theta)_1, \dots, (c_\theta)_m)$ denote the fixed-point $c$ shifted by $\theta$, that is $(c_\theta)_j = S_\theta (c_j)$. Suppose that $\mathcal{P}$ is the parameterization of the local unstable manifold of $c_\theta$ (cf. Section \ref{sec:unstable_manifold}).
We seek $\sigma, h$ and a trajectory $y_1, \dots, y_k \in \mathcal{C}^n$ of the DDS \eqref{eq:dds} satisfying
\begin{equation}\label{eq:co}
\begin{cases}
y_1 = \mathcal{P}_m (\sigma), \\
y_j = \mathcal{F} (y_{j-1}), & j = 2, \dots, k, \\
\mathcal{Q}( h ) = \mathcal{F} (y_k),
\end{cases}
\end{equation}
where the image of $\mathcal{Q}$ is the local graph of the stable manifold of one of the $m$-periodic orbits $c_1, \dots, c_m$.
Note that we made the arbitrary choice of departing from $\mathcal{P}_m$.

We now rephrase the equations \eqref{eq:co} on appropriate sequence spaces.
Within the context of this work, it is appropriate to discretize the Banach space $\mathcal{C}_\alpha \subset \mathcal{C}$ whose elements are functions with absolutely convergent Chebyshev series, that is their sequence of coefficients belongs to $\ell^1$.
This function space, for instance, contains absolutely continuous functions (functions in $\mathcal{C}$ and differentiable almost everywhere).

On the other hand, the connecting orbit is analytic and we expand $y_1, \dots, y_k$ as Chebyshev series
\[
y_j (s (t)) = \{ \mathbf{y} \}_0 + 2 \sum_{\alpha \ge 1} \{ \mathbf{y} \}_\alpha T_\alpha (t), \quad \text{for all } t \in [-1, 1], \qquad j = 1, \dots, m,
\]
where $s(t) = \frac{\tau}{2} (t-1)$ scales $[-1, 1]$ to $[-\tau, 0]$.

Then, for any $\mathbf{Q} \in C^1((\ell^1)^n, (\ell^1)^n)$, consider the mapping $\mathbf{F}_\textnormal{co} : \mathbb{C} \times ((\ell^1(\ell^1_\nu))^n)^m \times \mathbb{C}^{n_\textnormal{u}} \times (\ell^1)^n \times ((\ell^1_\nu)^n)^k \to (\ell^1)^n \times ((\ell^1_\nu)^n)^k$ defined by
\begin{equation}\label{eq:def_F_co}
\mathbf{F}_\textnormal{co} (\tau, \mathbf{p}, \sigma, \mathbf{h}, \mathbf{y}; \mathbf{Q})
\bydef
\begin{pmatrix}
\mathbf{E} (\mathbf{y}_k) + \mathbf{S} (\frac{\tau}{2} f( \mathbf{Q}(\mathbf{h}), \mathbf{y}_k )) - \mathbf{Q}(\mathbf{h}) \\
\mathbf{E}_{\textnormal{T},\sigma} (\mathbf{p}_m) - \mathbf{y}_1 \\
\mathbf{E} (\mathbf{y}_1) + \mathbf{S} (\frac{\tau}{2} f( \mathbf{y}_2, \mathbf{y}_1 )) - \mathbf{y}_2 \\
\vdots \\
\mathbf{E} (\mathbf{y}_{k-1}) + \mathbf{S} (\frac{\tau}{2} f( \mathbf{y}_k, \mathbf{y}_{k-1} )) - \mathbf{y}_k
\end{pmatrix},
\end{equation}
where $\mathbf{E}$ is given in \eqref{eq:def_Cheb_eval}, $\mathbf{S}$ is given in \eqref{eq:def_Cheb_integral} and $\mathbf{E}_{\textnormal{T},\sigma} : (\ell^1(\ell^1_\nu))^n \to (\ell^1_\nu)^n$ represents an evaluation operator at $\sigma$ with respect to the Taylor expansion. Namely, for all $\mathbf{a} = (\mathbf{a}_1, \dots, \mathbf{a}_n) \in (\ell^1(\ell^1_\nu))^n$ and $i=1, \dots,n$,
\[
(\mathbf{E}_{\textnormal{T},\sigma} ( \mathbf{a} ))_i \bydef \sum_{|\alpha| \ge 0} \{ \mathbf{a}_i \}_\alpha \sigma^\alpha.
\]

Suppose $\tau \in \mathbb{R}$ is the delay, $\mathbf{p} \in ((\ell^1(\ell^1_\nu))^n)^m$ denotes the sequences of Taylor-Chebyshev coefficients of $\mathcal{P}$ and $\mathbf{Q}$ represents the action of $\mathcal{Q}$ on sequences of Chebyshev coefficients.
It follows that if there exist $\sigma \in \mathbb{C}^{n_\textnormal{u}}, \mathbf{h} \in (\ell^1)^n$ and $\mathbf{y} = (\mathbf{y}_1, \dots, \mathbf{y}_k) \in ((\ell^1_\nu)^n)^k$ such that $\mathbf{F}_\textnormal{co} (\tau, \mathbf{p}, \sigma, \mathbf{h}, \mathbf{y}; \mathbf{Q}) = 0$, then  $\mathbf{y}_1, \dots, \mathbf{y}_k$ are the sequences of Chebyshev coefficients of a connecting orbit joining the stable and unstable manifolds.

\subsection{Calibration of the phase of the unstable manifold}

As introduced in Lemma \ref{lem:equivalence_transversality}, the quantity $\theta$ embodies a shift of the periodic orbit.
This shift is an unavoidable effect of our strategy since we iterate under the time-$\tau$ map which has a fixed step-size $\tau$.
Now, throughout this article we made a stand to use Chebyshev series due to their nice convergence properties. Sadly, the shift operator, on the level of the Chebyshev coefficients, is not well understood.
Instead, we consider the reciprocal situation: the connection occurs for a specific value of the phase of the unstable manifold.
Thus, the phase $\delta$ of the unstable manifold will be an unknown of the zero-finding problem for the transverse homoclinic orbit; the homoclinic orbit will be given as a trajectory of the DDS \eqref{eq:dds} joining the stable and unstable manifolds, where the stable manifold will be fixed and the unstable manifold will be solved for.

We go back to Section \ref{sec:unstable_manifold} and turn to the practical question of solving the conjugacy relation \eqref{eq:unrolled_conjugacy}; namely, for all $s \in [-\tau, 0]$,
\begin{equation}\label{eq:unstable_manifold_fun}
\begin{cases}
\displaystyle [\mathcal{P}_1(\Lambda \sigma)](s) = [\mathcal{P}_m(\sigma)](0) + \int_{-\tau}^s f( [\mathcal{P}_1(\Lambda \sigma)](s'), [\mathcal{P}_m(\sigma)](s') ) \, ds', \\
\displaystyle [\mathcal{P}_j(\Lambda \sigma)](s) = [\mathcal{P}_{j-1}(\sigma)](0) + \int_{-\tau}^s f( [\mathcal{P}_j(\Lambda \sigma)](s'), [\mathcal{P}_{j-1}(\sigma)](s') ) \, ds' , & j = 2, \dots, m.
\end{cases} %\qquad \text{for all } s \in [-\tau, 0].
\end{equation}
First, for $j = 1, \dots, m$, we expand $\mathcal{P}_j$ as the Taylor-Chebyshev series
\[
[\mathcal{P}_j (\sigma)](s(t)) = \sum_{|\alpha| \ge 0} \left( \{\mathbf{p}_j\}_{0,\alpha} + 2 \sum_{\beta \ge 1} \{\mathbf{p}_j\}_{\beta, \alpha} T_\beta (t) \right) \sigma^\alpha, \qquad \text{for all } t \in [-1, 1], \, \sigma \in \mathbb{D}^{n_\textnormal{u}},
\]
where $s(t) \bydef \frac{\tau}{2}(t-1)$ scales $[-1, 1]$ to $[-\tau, 0]$.
Also, we define $\{ \mathbf{p}_j \}_\alpha \bydef \{ \{ \mathbf{p}_j \}_{\beta, \alpha} \}_{\beta \ge 0} \in (\ell^1_\nu)^n$ for $|\alpha| \ge 0$.
In other words, $\{ \mathbf{p} \}_\alpha = (\{ \mathbf{p}_1 \}_\alpha, \dots, \{ \mathbf{p}_m \}_\alpha)$ denotes the sequences of Chebyshev coefficients corresponding to analytic functions in $(C([-1, 1], \mathbb{R}^n))^m$.
Then, the system of equations \eqref{eq:unstable_manifold_fun} is equivalent to
\[
\begin{cases}
\lambda^\alpha \{ \mathbf{p}_1 \}_\alpha = \mathbf{E} (\{\mathbf{p}_m\}_\alpha) + \mathbf{S} (\tfrac{\tau}{2} \{f( \mathbf{L}(\lambda, \mathbf{p}_1), \mathbf{p}_m )\}_\alpha ), \\
\lambda^\alpha \{ \mathbf{p}_j \}_\alpha = \mathbf{E} (\{\mathbf{p}_{j-1}\}_\alpha) + \mathbf{S} (\tfrac{\tau}{2} \{f( \mathbf{L}(\lambda, \mathbf{p}_j), \mathbf{p}_{j-1} )\}_\alpha ), & j = 2, \dots, m,
\end{cases} \qquad |\alpha| \ge 0,
\]
where $\mathbf{E}$ is given in \eqref{eq:def_Cheb_eval}, $\mathbf{S}$ is given in \eqref{eq:def_Cheb_integral} and $\mathbf{L}$ is given in \eqref{eq:def_L}.

Consider the mapping $\mathbf{F}_{W_\textnormal{u}} : \mathbb{C} \times \mathbb{C}^{n_\textnormal{u}} \times (\ell^1(\ell^1_\nu))^m \times \mathbb{C} \to \mathbb{C} \times \mathbb{C}^{n_\textnormal{u}} \times (\ell^1(\ell^1_\nu))^m$ defined by
\begin{equation}\label{eq:def_F_unstable_manifold}
\mathbf{F}_{W_\textnormal{u}} (\tau, \lambda, \mathbf{p}, \delta)
\bydef
\begin{pmatrix}
\{\mathbf{E} \big(\{\mathbf{p}_m\}_0 \big)\}_0 - \delta \\
\{\mathbf{E} \big(\{\mathbf{p}_m\}_{e_1} \big)\}_0 - \gamma_1 \\
\vdots \\
\{\mathbf{E} \big(\{\mathbf{p}_m\}_{e_{n_\textnormal{u}}}\big)\}_0 - \gamma_{n_\textnormal{u}} \\
\left\{ \mathbf{E} \big( \{\mathbf{p}_m\}_\alpha \big) + \mathbf{S} \big( \frac{\tau}{2} \{f( \mathbf{L}(\lambda, \mathbf{p}_1), \mathbf{p}_m ) \}_\alpha \big) - \lambda^\alpha \{\mathbf{p}_1\}_\alpha \right\}_{|\alpha| \ge 0} \\
\left\{ \mathbf{E} \big( \{\mathbf{p}_1\}_\alpha \big) + \mathbf{S} \big( \frac{\tau}{2} \{f( \mathbf{L}(\lambda, \mathbf{p}_2), \mathbf{p}_1 ) \}_\alpha \big) - \lambda^\alpha \{\mathbf{p}_2\}_\alpha \right\}_{|\alpha| \ge 0} \\
\vdots \\
\left\{ \mathbf{E} \big( \{\mathbf{p}_{m-1}\}_\alpha \big) + \mathbf{S} \big( \frac{\tau}{2} \{f( \mathbf{L}(\lambda, \mathbf{p}_m), \mathbf{p}_{m-1} ) \}_\alpha \big) - \lambda^\alpha \{\mathbf{p}_m\}_\alpha \right\}_{|\alpha| \ge 0}
\end{pmatrix},
\end{equation}
where $\gamma_1, \dots, \gamma_{n_\textnormal{u}} \in \mathbb{R}$ are fixed.
If $\mathbf{F}_{W_\textnormal{u}} (\tau, \lambda, \mathbf{p}, \delta) = 0$, with $\lambda = (\lambda_1, \dots, \lambda_{n_\textnormal{u}}) \in \mathbb{C}^{n_\textnormal{u}}$ such that $|\lambda_l| > 1$ for $l = 1, \dots, n_\textnormal{u}$, then $\mathbf{p}_1, \dots, \mathbf{p}_m$ are the sequences of Taylor-Chebyshev coefficients of the parameterization of the local unstable manifold of an $m\tau$-periodic orbit of the DDE \eqref{eq:dde_elem}.
As a matter of fact, this is still not quite sufficient since one must guarantee that the collection of unstable eigenvalues $\lambda$ is complete; to make this argument completely rigorous, one would need to obtain the Morse index of the periodic solution.
In Section \ref{sec:eigendecomposition}, we briefly mentioned that such a strategy is possible through an homotopy argument, e.g. see \cite{Lessard2020}.
For the needs of the present article, we will rely on the count obtained by numerically computing the spectrum of the operator $\mathbf{H}_j^N$, for some $j \in \{1, \dots, m\}$, given in \eqref{eq:def_H_num}.

Therefore, solving for a zero of $\mathbf{F}_{W_\textnormal{u}}$ amounts to solving for the delay, the unstable eigenvalues, the periodic orbit together with its phase, its unstable eigenvectors and the higher order term of the parameterization.

As done in the zero-finding problem \eqref{eq:zero_finding_problem_fp}, we impose that $\{\mathbf{E} \big(\{\mathbf{p}_m\}_0 \big)\}_0 - \delta = 0$ to fix the phase of the periodic orbit which is balanced out by solving for the delay $\tau$.

Additionally, since the multiplication of an eigenvector with a scalar also yields an eigenvector, we isolate them by fixing the length of the $m$-th component which is achieved by the set of equations $\{\mathbf{E}\big(\{\mathbf{p}_m\}_{e_1}\big)\}_0 - \gamma_1 = \ldots = \{\mathbf{E}\big(\{\mathbf{p}_m\}_{e_{n_\textnormal{u}}}\big)\}_0 - \gamma_{n_\textnormal{u}} = 0$. These equations are themselves compensated by solving for the eigenvalues $\lambda = (\lambda_1, \dots, \lambda_{n_\textnormal{u}})$.
The reader should recall from Section \ref{sec:unstable_manifold} that the scaling $\gamma_1, \dots, \gamma_{n_\textnormal{u}}$ of the eigenvectors directly impacts the decay rate of $\mathbf{p}$; in practice, one often tries several values until satisfied with the decay rate.

Once again, since the periodic orbit is an argument of the mapping $\mathbf{F}_{W_\textnormal{u}}$, we follow the same arguments that led to the mapping $\mathbf{F}_{\circ, \textnormal{elem}}$ given in \eqref{eq:zero_finding_problem_fp_poly}. For all $(\tau, \eta, \lambda, \mathbf{p}, \delta) \in \mathbb{C} \times \mathbb{C}^d \times \mathbb{C}^{n_\textnormal{u}} \times ((\ell^1(\ell^1_\nu))^{1+d})^m \times \mathbb{C}^{1+d}$, with $\mathbf{p}_j = (\mathbf{p}_j^{(1)}, \mathbf{p}_j^{(2)}) \in (\ell^1(\ell^1_\nu))^{1+d}$ such that $\mathbf{p}_j^{(1)} \in \ell^1(\ell^1_\nu)$, $\mathbf{p}_j^{(2)} \in (\ell^1(\ell^1_\nu))^d$ for $j = 1, \dots, m$, consider the mapping $\mathbf{F}_{W_\textnormal{u},\textnormal{elem}} : \mathbb{C} \times \mathbb{C}^d \times \mathbb{C}^{n_\textnormal{u}} \times ((\ell^1(\ell^1_\nu))^{1+d})^m \times \mathbb{C}^{1+d} \to \mathbb{C} \times \mathbb{C}^d \times \mathbb{C}^{n_\textnormal{u}} \times ((\ell^1(\ell^1_\nu))^{1+d})^m$ defined by
\begin{align}\label{eq:def_F_unstable_manifold_poly}
\mathbf{F}_{W_\textnormal{u}, \textnormal{elem}} (\tau, \eta, &\lambda, \mathbf{p}, \delta) \nonumber\\
&\bydef
\begin{pmatrix}
\{\mathbf{E} \big(\{\mathbf{p}_m\}_0 \big)\}_0 - \delta \\
\{\mathbf{E} \big(\{\mathbf{p}_m^{(1)}\}_{e_1} \big)\}_0 - \gamma_1 \\
\vdots \\
\{\mathbf{E} \big(\{\mathbf{p}_m^{(1)}\}_{e_{n_\textnormal{u}}}\big)\}_0 - \gamma_{n_\textnormal{u}} \\
\left\{ \mathbf{E} \big( \{\mathbf{p}_m\}_\alpha \big) +
\mathbf{S} \left(\frac{\tau}{2} \{f(\mathbf{L}(\lambda, \mathbf{p}_1), \mathbf{p}_m)\}_\alpha + \begin{pmatrix}0 \\ \boldsymbol{\iota}(\eta)\end{pmatrix} \right)
- \lambda^\alpha \{\mathbf{p}_1\}_\alpha \right\}_{|\alpha|\ge 0} \\
\left\{ \mathbf{E} \big( \{\mathbf{p}_1\}_\alpha \big) +
\mathbf{S} \left(\frac{\tau}{2} \{f(\mathbf{L}(\lambda, \mathbf{p}_2), \mathbf{p}_1)\}_\alpha + \begin{pmatrix}0 \\ \boldsymbol{\iota}(\eta)\end{pmatrix} \right)
- \lambda^\alpha \{\mathbf{p}_2\}_\alpha \right\}_{|\alpha|\ge 0} \\
\vdots \\
\left\{ \mathbf{E} \big( \{\mathbf{p}_{m-1}\}_\alpha \big) +
\mathbf{S} \left(\frac{\tau}{2} \{f(\mathbf{L}(\lambda, \mathbf{p}_m), \mathbf{p}_{m-1})\}_\alpha + \begin{pmatrix}0 \\ \boldsymbol{\iota}(\eta)\end{pmatrix} \right)
- \lambda^\alpha \{\mathbf{p}_m\}_\alpha \right\}_{|\alpha|\ge 0}
\end{pmatrix},
\end{align}
where $\gamma_1, \dots, \gamma_{n_\textnormal{u}} \in \mathbb{R}$ are fixed and $\boldsymbol{\iota}$ is given in \eqref{eq:def_iota}.
If $\mathbf{F}_{W_\textnormal{u}, \textnormal{elem}} (\tau, \eta, \lambda, \mathbf{p}, \delta) = 0$ and $\eta = 0$, then $\mathbf{p}_1^{(1)}, \dots, \mathbf{p}_m^{(1)}$ are the sequences of Taylor-Chebyshev coefficients of the parameterization of the local unstable manifold of an $m\tau$-periodic orbit of the original DDE \eqref{eq:dde_elem}. To see why this statement holds, one can see from Point 3 of Theorem 3.1 in \cite{Henot2021} that the unstable manifold in the auxiliary polynomial DDE \eqref{eq:dde} and the original DDE \eqref{eq:dde_elem} coincide whenever the periodic orbit is a periodic orbit of the original DDE \eqref{eq:dde_elem}.

\subsection{Zero-finding problem for the transverse homoclinic orbit}

We can now combine $\mathbf{F}_\textnormal{co}$ \eqref{eq:def_F_co} with $\mathbf{F}_{W_\textnormal{u}}$ \eqref{eq:def_F_unstable_manifold}, or $\mathbf{F}_{W_\textnormal{u}, \textnormal{elem}}$ \eqref{eq:def_F_unstable_manifold_poly}.

For all $(\tau, \lambda, \mathbf{p}, \mathbf{d}, \mathbf{y}) \in \mathbb{C} \times \mathbb{C}^{n_\textnormal{u}} \times (\ell^1(\ell^1_\nu))^m \times \ell^1 \times (\ell^1_\nu)^k$, define $\sigma \bydef (\{\mathbf{d}\}_0, \dots, \{\mathbf{d}\}_{n_\textnormal{u}-1}) \in \mathbb{C}^{n_\textnormal{u}}$, $\delta \bydef \{ \mathbf{d} \}_{n_\textnormal{u}} \in \mathbb{C}$ and $\mathbf{h} \bydef \{ \{ \mathbf{d} \}_{n_\textnormal{u} + 1 + \alpha} \}_{\alpha \ge 0} \in \ell^1$.
For any $\mathbf{Q} \in C^1(\ell^1, \ell^1)$, consider the mapping $$\mathbf{F}_\pitchfork : \mathbb{C} \times \mathbb{C}^{n_\textnormal{u}} \times (\ell^1(\ell^1_\nu))^m \times \ell^1 \times (\ell^1_\nu)^k \to \mathbb{C} \times \mathbb{C}^{n_\textnormal{u}} \times (\ell^1(\ell^1_\nu))^m \times \ell^1 \times (\ell^1_\nu)^k$$ defined by
\begin{equation}\label{eq:zero_finding_problem_connection}
\mathbf{F}_\pitchfork (\tau, \lambda, \mathbf{p}, \mathbf{d}, \mathbf{y}; \mathbf{Q})
\bydef
\begin{pmatrix}
\mathbf{F}_{W_\textnormal{u}} (\tau, \lambda, \mathbf{p}, \delta) \\
\mathbf{F}_\textnormal{co} (\tau, \mathbf{p}, \sigma, \mathbf{h}, \mathbf{y}; \mathbf{Q})
\end{pmatrix},
\end{equation}
Similarly, for all $(\tau, \eta, \lambda, \mathbf{p}, \mathbf{d}, \mathbf{y}) \in \mathbb{C} \times \mathbb{C}^d \times \mathbb{C}^{n_\textnormal{u}} \times ((\ell^1(\ell^1_\nu))^{1+d})^m \times (\ell^1)^{1+d} \times ((\ell^1_\nu)^{1+d})^k$, define $\sigma \bydef (\{\mathbf{d}\}_0, \dots, \{\mathbf{d}\}_{n_\textnormal{u}-1}) \in \mathbb{C}^{n_\textnormal{u}}$, $\delta \bydef (\{ \mathbf{d} \}_{n_\textnormal{u}}, \dots, \{ \mathbf{d} \}_{n_\textnormal{u}+d}) \in \mathbb{C}^{1+d}$ and $\mathbf{h} \bydef \{ \{ \mathbf{d} \}_{n_\textnormal{u}+d + 1 + \alpha} \}_{\alpha \ge 0} \in (\ell^1)^{1+d}$.
For any $\mathbf{Q} \in C^1((\ell^1)^{1+d}, (\ell^1)^{1+d})$, consider the mapping
\begin{align*}
\mathbf{F}_{\pitchfork, \textnormal{elem}} : \mathbb{C} \times \mathbb{C}^d \times \mathbb{C}^{n_\textnormal{u}} \times ((\ell^1(\ell^1_\nu))^{1+d})^m \times (\ell^1&)^{1+d} \times ((\ell^1_\nu)^{1+d})^k \\
&\to \mathbb{C} \times \mathbb{C}^d \times \mathbb{C}^{n_\textnormal{u}} \times ((\ell^1(\ell^1_\nu))^{1+d})^m \times (\ell^1)^{1+d} \times ((\ell^1_\nu)^{1+d})^k
\end{align*}
defined by
\begin{equation}\label{eq:zero_finding_problem_connection_poly}
\mathbf{F}_{\pitchfork, \textnormal{elem}} (\tau, \eta, \lambda, \mathbf{p}, \mathbf{d}, \mathbf{y}; \mathbf{Q})
\bydef
\begin{pmatrix}
\mathbf{F}_{W_\textnormal{u}, \textnormal{elem}} (\tau, \eta, \lambda, \mathbf{p}, \delta) \\
\mathbf{F}_\textnormal{co} (\tau, \mathbf{p}, \sigma, \mathbf{h}, \mathbf{y}; \mathbf{Q})
\end{pmatrix}.
\end{equation}

The following theorem is the core of this article as it motivates the entire design of the method.

\begin{theorem}\label{thm:transversality}
Let $\tau > 0$, $\eta \in \mathbb{C}^{n-1}$, $\lambda = (\lambda_1, \dots, \lambda_{n_\textnormal{u}}) \in \mathbb{C}^{n_\textnormal{u}}$ satisfying $|\lambda_l| > 1$ for $l = 1, \dots, n_\textnormal{u}$, $\mathbf{p} \in ((\ell^1(\ell^1_\nu))^n)^m$, $\mathbf{d} \in (\ell^1_\nu)^n$ and $\mathbf{y} \in ((\ell^1_\nu)^n)^k$.

\begin{enumerate}
\item Suppose $(\tau, \mathbf{c})$, with $\mathbf{c} = (\mathbf{c}_1, \dots, \mathbf{c}_m)$, is a zero of the mapping $\mathbf{F}_\circ$ given in \eqref{eq:zero_finding_problem_fp} and $\mathbf{Q} \in C^1 (\ell^1, \ell^1)$ whose image represents the local graph of the stable manifold of one of the $m$-periodic orbits $\mathbf{c}_1, \dots, \mathbf{c}_m$.
Assume further that the $m\tau$-periodic orbit represented by $\mathbf{c}$ is hyperbolic, with exactly $n_\textnormal{u}$ unstable Floquet multipliers.

If $\mathbf{F}_\pitchfork(\tau, \lambda, \mathbf{p}, \mathbf{d}, \mathbf{y}; \mathbf{Q}) = 0$ such that $\{\mathbf{p}\}_0$ corresponds to a phase shift of the $m\tau$-periodic orbit represented by $\mathbf{c}$, then $\lambda_1, \dots, \lambda_{n_\textnormal{u}}$ are (non-resonant) unstable Floquet multipliers, $\mathbf{p}$ is the corresponding parameterization of the local unstable manifold and $\mathbf{y}$ is a homoclinic orbit of the DDE \eqref{eq:dde_elem}.
Moreover, if $D\mathbf{F}_\pitchfork (\tau, \lambda, \mathbf{p}, \mathbf{d}, \mathbf{y}; \mathbf{Q})$ is invertible, then the stable and unstable manifolds intersect transversely in $\mathcal{C}_\alpha$.

\item Suppose $(\tau, 0, \mathbf{c})$, with $\mathbf{c} = (\mathbf{c}_1, \dots, \mathbf{c}_m)$, is a zero of the mapping $\mathbf{F}_{\circ, \textnormal{elem}}$ given in \eqref{eq:zero_finding_problem_fp_poly} and $\mathbf{Q} \in C^1 ((\ell^1)^{1+d}, (\ell^1)^{1+d})$ whose image represents the local graph of the stable manifold of one of the $m$-periodic orbits $\mathbf{c}_1, \dots, \mathbf{c}_m$.
Assume further that the $m\tau$-periodic orbit represented by $\mathbf{c}$ is hyperbolic, with exactly $n_\textnormal{u}$ unstable Floquet multipliers.

If $\mathbf{F}_{\pitchfork,\textnormal{elem}}(\tau, \eta, \lambda, \mathbf{p}, \mathbf{d}, \mathbf{y}; \mathbf{Q}) = 0$ such that $\{\mathbf{p}\}_0$ corresponds to a phase shift of the $m\tau$-periodic orbit represented by $\mathbf{c}$, then $\eta = 0$, $\lambda_1, \dots, \lambda_{n_\textnormal{u}}$ are (non-resonant) unstable Floquet multipliers, $\mathbf{p}$ is the corresponding parameterization of the local unstable manifold and $\mathbf{y}$ is a homoclinic orbit of the DDE \eqref{eq:dde_elem}.
Moreover, if $D\mathbf{F}_{\pitchfork,\textnormal{elem}} (\tau, \eta, \lambda, \mathbf{p}, \mathbf{d}, \mathbf{y}; \mathbf{Q})$ is invertible, then the stable and unstable manifolds intersect transversely in $\mathcal{C}_\alpha^{1+d}$.
\end{enumerate}
In the above, the derivatives $D\mathbf{F}_\pitchfork (\tau, \lambda, \mathbf{p}, \mathbf{d}, \mathbf{y}; \mathbf{Q})$ and $D\mathbf{F}_{\pitchfork,\textnormal{elem}} (\tau, \eta, \lambda, \mathbf{p}, \mathbf{d}, \mathbf{y}; \mathbf{Q})$ do not differentiate with respect to $\mathbf{Q}$.
\end{theorem}

\begin{proof}
Note that the proof of Point 1 can be used almost verbatim to prove Point 2.
One important difference is the implication that $\eta = 0$ for Point 2.
Let us detail why this holds.
Since $(\tau, 0, \mathbf{c})$ is a zero of $\mathbf{F}_{\circ, \textnormal{elem}}$ where, by assumption, $\tau$ is a real strictly positive number, we have that $\mathbf{c}$ represents a $m\tau$-periodic orbit of the DDE \eqref{eq:dde_elem} (cf. Section \ref{sec:periodic_orbit}).
Also, we assume that $\{\mathbf{p}\}_0$ corresponds to a phase shift of the $m\tau$-periodic orbit represented by $\mathbf{c}$, hence it is necessary that $\eta = 0$ (cf. Lemma \ref{lem:Lemma_3_2_revisited}).

Moreover, according to Definition \ref{def:resonance}, the eigenvalues $\lambda_1, \dots, \lambda_{n_{\textnormal{u}}}$ are non-resonant if and only if $\lambda^\alpha = \lambda_1^{\alpha_1} \times \ldots \times \lambda_{n_\textnormal{u}}^{\alpha_{n_\textnormal{u}}}$ is not an eigenvalue for all $\alpha \in (\mathbb{N} \cup \{ 0 \})^{n_\textnormal{u}}$ such that $|\alpha| \ge 2$.
Since $\mathbf{F}_{W_\textnormal{u}} (\tau, \lambda, \mathbf{p}, \delta) = 0$, then the recurrence relation \eqref{eq:recurrence_relation} has a solution for all $|\alpha| \ge 2$.
In particular, $\lambda^\alpha \mathbf{I} - \lambda^\alpha \mathbf{K}_1 (\tau, \mathbf{c}) - \mathbf{K}_2 (\tau, \mathbf{c})$ is invertible and $\lambda^\alpha$ is not an eigenvalue for all $|\alpha| \ge 2$.

The only remaining statement to prove is that the invertibility of $D\mathbf{F}_\pitchfork(\tau, \lambda, \mathbf{p}, \mathbf{d}, \mathbf{y}; \mathbf{Q})$ implies that the stable and unstable manifolds intersect transversely.

Define $\mathcal{P}(\sigma) \bydef \sum_{|\alpha| \ge 0} \{ \mathbf{p} \}_\alpha \sigma^\alpha$, $\sigma \bydef (\{ \mathbf{d} \}_0, \dots, \{ \mathbf{d} \}_{n_\textnormal{u}-1})$, $\mathbf{h} \bydef \{ \{ \mathbf{d} \}_{n_\textnormal{u} + 1 + \alpha} \}_{\alpha \ge 0}$ and $h(s(t)) \bydef \{ \mathbf{h} \}_0 + 2 \sum_{\alpha \ge 1} \{ \mathbf{h} \}_\alpha T_\alpha (t)$, where $s(t) \bydef \frac{\tau}{2}(t - 1)$ scales $[-1, 1]$ to $[-\tau, 0]$.
By assumption, the periodic orbit is hyperbolic. Hence, we must show that
\[
\textnormal{image} \, [D_{\sigma'} \mathcal{F}^{k+1} (\mathcal{P}_m (\sigma'))]_{\sigma' = \sigma} \cap \textnormal{image} \, D\mathcal{Q}( h ) = \{ 0 \},
\]
where $\mathcal{Q}$ represents the action of $\mathbf{Q}$ on $\mathcal{C}_\alpha$.
We shall argue by contradiction: suppose there exist $\psi_\sigma \in \mathbb{C}^{n_\textnormal{u}}$, $\xi \in \mathcal{C}$ such that $\psi_\sigma \ne 0$, $\xi \ne 0$ and
\begin{equation}\label{eq:proof_homoclinic}
[D_{\sigma'} \mathcal{F}^{k+1} (\mathcal{P}_m (\sigma'))]_{\sigma' = \sigma} \psi_\sigma = [D \mathcal{Q}(h)] \xi.
\end{equation}
Then, it follows that there exists a sequence of Chebyshev coefficients $\psi_h \in \ell^1$ such that $\xi (s(t)) = \{\psi_h\}_0 + 2 \sum_{\alpha \ge 1} \{ \psi_h \}_\alpha T_\alpha(t)$.

By assumption, $D\mathbf{F}_\pitchfork(\tau, \lambda, \mathbf{p}, \mathbf{d}, \mathbf{y}; \mathbf{Q})$ is invertible, in particular $[D_{\sigma'} \mathbf{E}_{\textnormal{T},\sigma'} (\mathbf{p}_m)]_{\sigma' = \sigma} : \mathbb{C}^{n_\textnormal{u}} \to \ell^1_\nu$.
From \eqref{eq:proof_homoclinic}, we have $\psi_{y_1}, \psi_{y_2}, \dots, \psi_{y_k} \in \ell^1_\nu$ satisfying
\begin{equation}\label{eq:proof_homoclinic_1}
\begin{cases}
\psi_{y_1} = [D_{\sigma'} \mathbf{E}_{\textnormal{T},\sigma'} (\mathbf{p}_m)]_{\sigma' = \sigma} \psi_\sigma, \\
\psi_{y_j} - \mathbf{S} ( \frac{\tau}{2} [D_1 f(\mathbf{y}_j, \mathbf{y}_{j-1})] \psi_{y_j} ) = \mathbf{E} (\psi_{y_{j-1}}) + \mathbf{S} ( \frac{\tau}{2} [D_2 f(\mathbf{y}_j, \mathbf{y}_{j-1})] \psi_{y_{j-1}} ), & j = 2, \dots, k, \\
[D \mathbf{Q}(\mathbf{h})] \psi_h - \mathbf{S} ( \frac{\tau}{2} [D_1 f(\mathbf{Q}(\mathbf{h}), \mathbf{y}_k)] [D \mathbf{Q}(\mathbf{h})] \psi_h ) = \mathbf{E} (\psi_{y_k}) + \mathbf{S} ( \frac{\tau}{2} [D_2 f(\mathbf{Q}(\mathbf{h}), \mathbf{y}_k)] \psi_{y_k} ).
\end{cases}
\end{equation}

Define $\psi_d \in \ell^1$ as
\[
\{\psi_d\}_\alpha \bydef
\begin{cases}
(\psi_\sigma)_{\alpha+1}, & \alpha = 0, \dots, n_\textnormal{u}-1, \\
0, & \alpha = n_\textnormal{u}, \\
\{\psi_h\}_{\alpha-n_\textnormal{u} - 1}, & \alpha \ge n_\textnormal{u} + 1.
\end{cases}
\]
A direct computation shows that
\begin{align*}
[D \mathbf{F}_\pitchfork (\tau, \lambda, \mathbf{p}, \mathbf{d}, \mathbf{y}; \mathbf{Q})]
&\begin{pmatrix}
0 \\
0 \\
0 \\
\psi_d \\
\psi_{y_1} \\
\vdots \\
\psi_{y_k}
\end{pmatrix} \\
&=
\begin{pmatrix}
0 \\
0 \\
0 \\
\mathbf{E} (\psi_k) + \mathbf{S} \Big(\frac{\tau}{2} [D f( \mathbf{Q}( \mathbf{h} ), \mathbf{y}_k )]
\begin{pmatrix}[D\mathbf{Q}( \mathbf{h} )]\psi_h \\ \psi_{y_k} \end{pmatrix} \Big) - [D\mathbf{Q}( \mathbf{h} )]\psi_h \\
[D_{\sigma'} \mathbf{E}_{\textnormal{T},\sigma'} (\mathbf{p}_m)]_{\sigma' = \sigma} \psi_\sigma - \psi_{y_1} \\
\mathbf{E} (\psi_{y_1}) + \mathbf{S} \Big(\frac{\tau}{2} [D f( \mathbf{y}_2, \mathbf{y}_1 )] \begin{pmatrix}\psi_{y_2}\\\psi_{y_1}\end{pmatrix} \Big) - \psi_{y_2} \\
\vdots \\
\mathbf{E} (\psi_{y_{k-1}}) + \mathbf{S} \Big(\frac{\tau}{2} [D f( \mathbf{y}_k, \mathbf{y}_{k-1} )] \begin{pmatrix}\psi_{y_k}\\\psi_{y_{k-1}}\end{pmatrix} \Big) - \psi_{y_k}
\end{pmatrix} \\
&= 0,
\end{align*}
where the last equality is equivalent to the equations \eqref{eq:proof_homoclinic_1}.
Thus, the injectivity of $D \mathbf{F}_\pitchfork (\tau, \lambda, \mathbf{p}, \mathbf{d}, \mathbf{y})$ is violated and the proof is complete.
\end{proof}

The hypotheses of Point 1 (resp. Point 2) of Theorem \ref{thm:transversality} presuppose some knowledge about the periodic orbit.
The idea is that the receiving end of the BVP (i.e. the stable manifold) is fixed a priori.
The \emph{return periodic orbit} is known initially as a zero $(\tau, \mathbf{c})$ of $\mathbf{F}_\circ$ \eqref{eq:zero_finding_problem_fp} (resp. a zero $(\tau, 0, \mathbf{c})$ of $\mathbf{F}_{\circ, \textnormal{elem}}$ \eqref{eq:zero_finding_problem_fp_poly}).
Next, one verifies the hyperbolicity and retrieves the Morse index of $\mathbf{c}$ (e.g. by adapting the work in \cite{Lessard2020}).
Then, one needs to obtain the mapping $\mathbf{Q}$ (see e.g. \cite{Llave2016}).
Lastly, one checks that for a zero of $\mathbf{F}_\pitchfork$ (resp. $\mathbf{F}_{\pitchfork, \textnormal{elem}}$) the Chebyshev coefficients $\{\mathbf{p}\}_0$ represent nothing more than a phase shift of the $m\tau$-periodic orbit represented by $\mathbf{c}$.
For the purpose of this article, we shall follow this procedure numerically.

\subsection{Numerical considerations}
\label{sec:homoclinic_orbit_num}

In this section, we detail how to apply numerically Theorem \ref{thm:transversality}.
We take this opportunity to backtrack to the beginning to describe the whole picture.

Suppose numerical simulations on the DDE \eqref{eq:dde_elem} yield an \emph{initial periodic orbit}.
Following Section \ref{sec:periodic_orbit_num}, we obtain a numerical approximation of the delay $\bar{\tau}_\textnormal{init} > 0$ and its sequences of Chebyshev coefficients $\bar{\mathbf{c}}_\textnormal{init} \in \boldsymbol{\pi}^N ((\ell^1_\nu)^n)^m$.
Then, Section \ref{sec:eigendecomposition_num} allows us to approximate the Floquet multipliers and associated eigenvectors for $\bar{\mathbf{c}}_\textnormal{init} \in \boldsymbol{\pi}^N ((\ell^1_\nu)^n)^m$ via the operator $\mathbf{H}^N_1 (\bar{\tau}_\textnormal{init}, \bar{\mathbf{c}}_\textnormal{init})$ given in \eqref{eq:def_H_num}.
We retrieve the numerical Morse index $n_\textnormal{u}$; we also check that there are $n$ center eigenvalues, thereby suggesting that the initial periodic orbit is hyperbolic.

At this point, if $n_\textnormal{u} > 0$, we consider that we have a potential \emph{candidate} and our task is to build the connecting orbit.
The parameterization of the local unstable manifold is obtained by applying the technique developed in Section \ref{sec:unstable_manifold_num}.
We numerically grow its boundary and monitor the distance with respect to the \emph{initial periodic orbit}. If this distance is below a prescribed tolerance, then we have found a connection to a \emph{return periodic orbit} corresponding to a phased shift of the \emph{initial periodic orbit}.

Now, let us fix the \emph{receiving side} of the BVP.
Firstly, we follow again Section \ref{sec:periodic_orbit_num} to produce a numerical approximation of the delay $\bar{\tau} > 0$ and the sequences of Chebyshev coefficients $\bar{\mathbf{c}} = (\bar{\mathbf{c}}_1, \dots, \bar{\mathbf{c}}_m) \in \boldsymbol{\pi}^N ((\ell^1_\nu \cap \mathbb{R}^{\mathbb{N} \cup \{0\}} )^n)^m$ representing the \emph{return periodic orbit}.
Secondly, we approximate $\mathbf{Q}$ by taking finitely many stable eigenvectors associated with the largest stable eigenvalues; thus, we consider the first-order approximation
\[
\bar{\mathbf{Q}} ( \boldsymbol{\pi}^N \mathbf{h}) = \bar{\mathbf{c}}_{j_*} + \bar{\mathbf{V}} \boldsymbol{\pi}^N \mathbf{h},
\]
for some $j_* \in \{1, \dots, m\}$, where the operator $\bar{\mathbf{V}}$ is the matrix whose columns are approximations of the stable eigenvectors in $\boldsymbol{\pi}^N (\ell^1_\nu)^n$ of $\mathbf{H}^N_{j_*} (\bar{\tau}, \bar{\mathbf{c}})$.
The neglected linear components should be small due to the decreasing (as the Chebyshev truncation order $N$ increases) contribution of the stable eigenvectors as their associated eigenvalues accumulate to $0$.

To find an approximate zero of $\mathbf{F}_\pitchfork$ (resp. $\mathbf{F}_{\pitchfork, \textnormal{elem}}$), we imitate the procedure in Section \ref{sec:periodic_orbit_num}: we apply Newton's method on a truncated sequence space.
For $N, N' \in \mathbb{N} \cup \{0\}$, define the truncation operator $\boldsymbol{\pi}^{N,N'} : \ell^1(\ell^1_\nu) \to \ell^1(\ell^1_\nu)$ by
\[
\{ \boldsymbol{\pi}^{N,N'} \mathbf{a} \}_{\alpha,\beta} \bydef
\begin{cases}
\{\mathbf{a}\}_{\alpha,\beta}, & \alpha \le N, |\beta| \le N', \\
0, & \textnormal{otherwise},
\end{cases} \qquad \text{for all } \mathbf{a} \in \ell^1(\ell^1_\nu).
\]
Applying Newton's method on $\boldsymbol{\pi}^{N,N'} \mathbf{F}_\pitchfork (\,\cdot\, ; \bar{\mathbf{Q}})\boldsymbol{\pi}^{N,N'}$ (resp. $\boldsymbol{\pi}^{N,N'} \mathbf{F}_{\pitchfork, \textnormal{elem}} (\,\cdot\, ; \bar{\mathbf{Q}}) \boldsymbol{\pi}^{N,N'}$) and assuming it has converged, we obtain a distance to the \emph{return periodic orbit} by computing the norm of the approximation of the stable coordinates $\mathbf{h} \bydef \{ \{ \mathbf{d} \}_{n_\textnormal{u}+ n + \alpha} \}_{\alpha \ge 0} \in (\ell^1)^n$.
In double precision, our criterion is to have a distance of order $\sim 10^{-8}$ since the neglected quadratic terms should then be of order machine precision $\sim 10^{-16}$.

At last, we also end up with a new value of the delay which we denote by $\bar{\tau}_\pitchfork$.
Since the transverse homoclinic orbit must occur for a single value of the delay $\tau$, the gap $|\bar{\tau}_\pitchfork - \bar{\tau}|$ represents some additional error coming from fixing $\bar{\mathbf{Q}}$ a priori in the BVP.

\subsection{Example: transverse homoclinic orbit for the cubic Ikeda equation}
\label{sec:cubic_ikeda_homoclinic}

The zero-finding problem $\mathbf{F}_\pitchfork$ \eqref{eq:zero_finding_problem_connection} for the transverse homoclinic orbit is composed of the mappings given in $\mathbf{F}_\textnormal{co}$ \eqref{eq:def_F_co} and $\mathbf{F}_{W_\textnormal{u}}$ \eqref{eq:def_F_unstable_manifold}.
Then, for the cubic Ikeda equation \eqref{eq:cubic_ikeda}, these mappings read
\[
\mathbf{F}_\textnormal{co} (\tau, \mathbf{p}, \sigma, \mathbf{h}, \mathbf{y}; \mathbf{Q})
=
\begin{pmatrix}
\mathbf{E}_{\textnormal{T},\sigma} (\mathbf{p}_m) - \mathbf{y}_1 \\
\mathbf{E} (\mathbf{y}_1) + \mathbf{S} (\frac{\tau}{2} (\mathbf{y}_1 - \mathbf{y}_1^{*3})) - \mathbf{y}_2 \\
\vdots \\
\mathbf{E} (\mathbf{y}_{k-1}) + \mathbf{S} (\frac{\tau}{2} (\mathbf{y}_{k-1} - \mathbf{y}_{k-1}^{*3})) - \mathbf{y}_k \\
\mathbf{E} (\mathbf{y}_k) + \mathbf{S} (\frac{\tau}{2} (\mathbf{y}_k - \mathbf{y}_k^{*3})) - \mathbf{Q}(\mathbf{h})
\end{pmatrix},
\]
and, since we are looking into a periodic orbit with a $1$-dimensional unstable manifold (cf. Section \ref{sec:cubic_ikeda_eigendecomposition}),
\[
\mathbf{F}_{W_\textnormal{u}} (\tau, \lambda, \mathbf{p}, \delta)
=
\begin{pmatrix}
\{\mathbf{E} \big(\{\mathbf{p}_m\}_0 \big)\}_0 - \delta \\
\{\mathbf{E} \big(\{\mathbf{p}_m\}_1 \big)\}_0 - \gamma \\
\left\{ \mathbf{E} \big( \{\mathbf{p}_m\}_\alpha \big) + \mathbf{S} \big( \frac{\tau}{2} \{\mathbf{p}_m - \mathbf{p}_m^{\circledast 3}\}_\alpha \big) - \lambda^\alpha \{\mathbf{p}_1\}_\alpha \right\}_{\alpha \ge 0} \\
\left\{ \mathbf{E} \big( \{\mathbf{p}_1\}_\alpha \big) + \mathbf{S} \big( \frac{\tau}{2} \{ \mathbf{p}_1 - \mathbf{p}_1^{\circledast 3} \}_\alpha \big) - \lambda^\alpha \{\mathbf{p}_2\}_\alpha \right\}_{\alpha \ge 0} \\
\vdots \\
\left\{ \mathbf{E} \big( \{\mathbf{p}_{m-1}\}_\alpha \big) + \mathbf{S} \big( \frac{\tau}{2} \{ \mathbf{p}_{m-1} - \mathbf{p}_{m-1}^{\circledast 3} \}_\alpha \big) - \lambda^\alpha \{\mathbf{p}_m\}_\alpha \right\}_{\alpha \ge 0}
\end{pmatrix},
\]
where $\gamma \in \mathbb{R}$ is fixed. As explained in Section \ref{sec:unstable_manifold}, the choice of scaling guides the convergence of the Taylor series and is generally adjusted \emph{on the fly}.

We follow diligently the procedure described in Section \ref{sec:homoclinic_orbit_num}.
Our last computation in Section \ref{sec:cubic_ikeda_param} gave use the parameterization of the local unstable manifold associated with the \emph{initial periodic orbit} $\bar{\mathbf{c}}_\textnormal{init} \in \boldsymbol{\pi}^N (\ell^1_\nu)^m$ computed in Section \ref{sec:cubic_ikeda_po}; in particular, here $m = 8$, $N = 30$ and $N' = 15$.
A standard numerical scheme allows us to grow the boundary of the local unstable manifold; in our case, this is especially easy to achieve since the manifold is $1$-dimensional.
After $17$ iterations of the time-$\tau$ map, we find a connecting orbit about $\sim 10^{-5}$ close to the $7$-th piece of a \emph{return periodic orbit} whose phase is about $-0.7879127215879392$; hence, we set $j_* = 7$ and $k = 16$ for $\mathbf{F}_\textnormal{co}$.
The \emph{return periodic orbit}, denoted $\bar{\mathbf{c}} \in \boldsymbol{\pi}^N (\ell^1_\nu \cap \mathbb{R}^{\mathbb{N} \cup \{0\}})^m$, is further refined by running Newton's method on the zero-finding problem $\mathbf{F}_\circ$ \eqref{eq:zero_finding_problem_fp}, in a similar vein as we did in Section \ref{sec:cubic_ikeda_po} but for a different phase $\delta$; the resulting approximate value of the delay is $\bar{\tau} \approx 1.5649592985680902$.
Then, we use $\mathbf{H}^N_{j^*} (\bar{\tau}, \bar{\mathbf{c}})$ to get an approximation of the stable eigenspace.

All of this gives us a good guess to use Newton's method on $\mathbf{F}_\pitchfork$ where the $\delta$ of the periodic orbit, chosen initially to be $\delta = 0$ in Section \ref{sec:cubic_ikeda_po}, will be tuned by the Newton iterations in order to reach the approximation of the local stable eigenspace of the \emph{return periodic orbit}.

\begin{figure}
\centering
\begin{subfigure}[b]{0.32\textwidth}
\includegraphics[width=\textwidth]{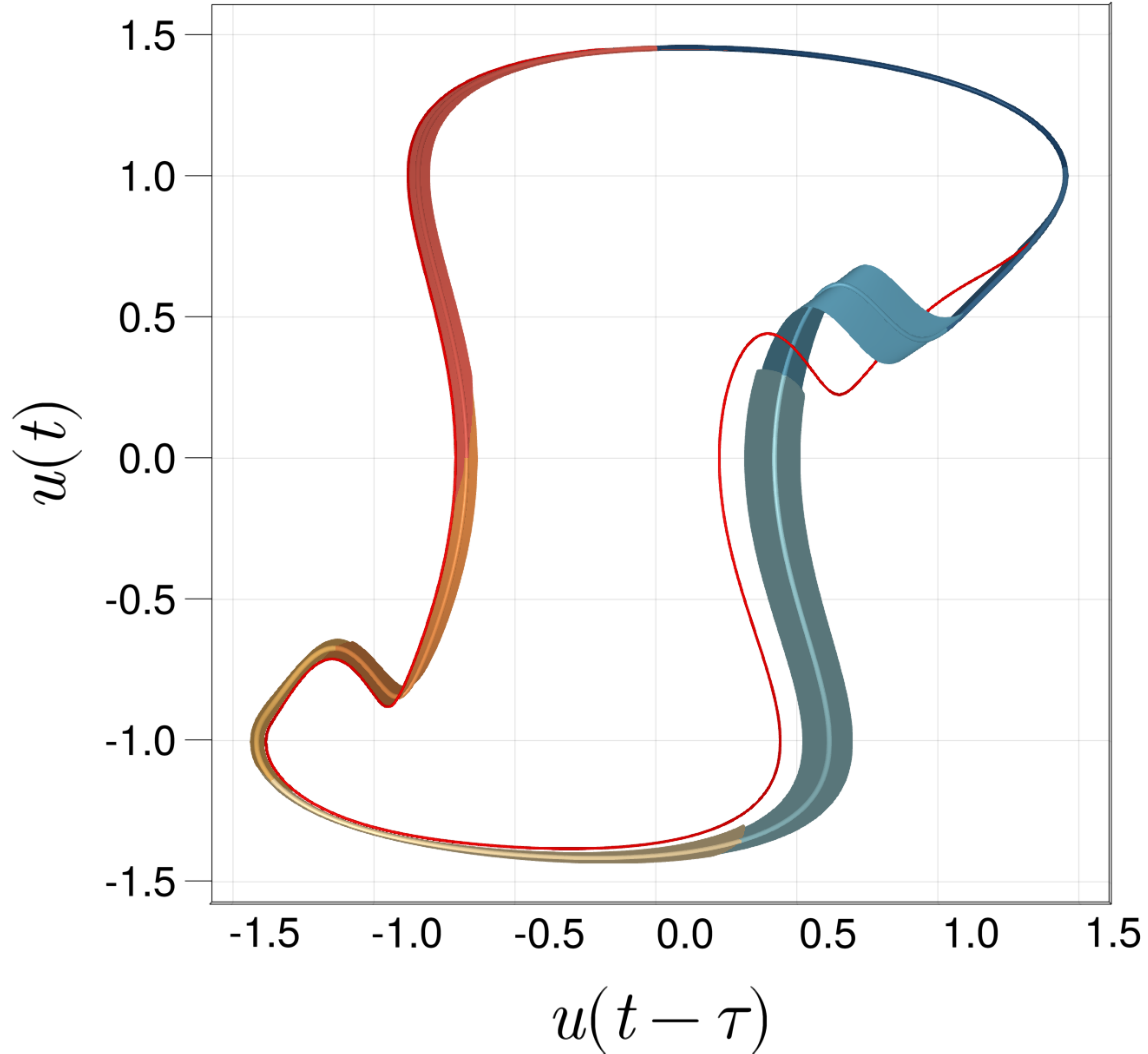}
\caption{}
\end{subfigure}
\hspace{0.03\textwidth}
\begin{subfigure}[b]{0.3\textwidth}
\includegraphics[width=\textwidth]{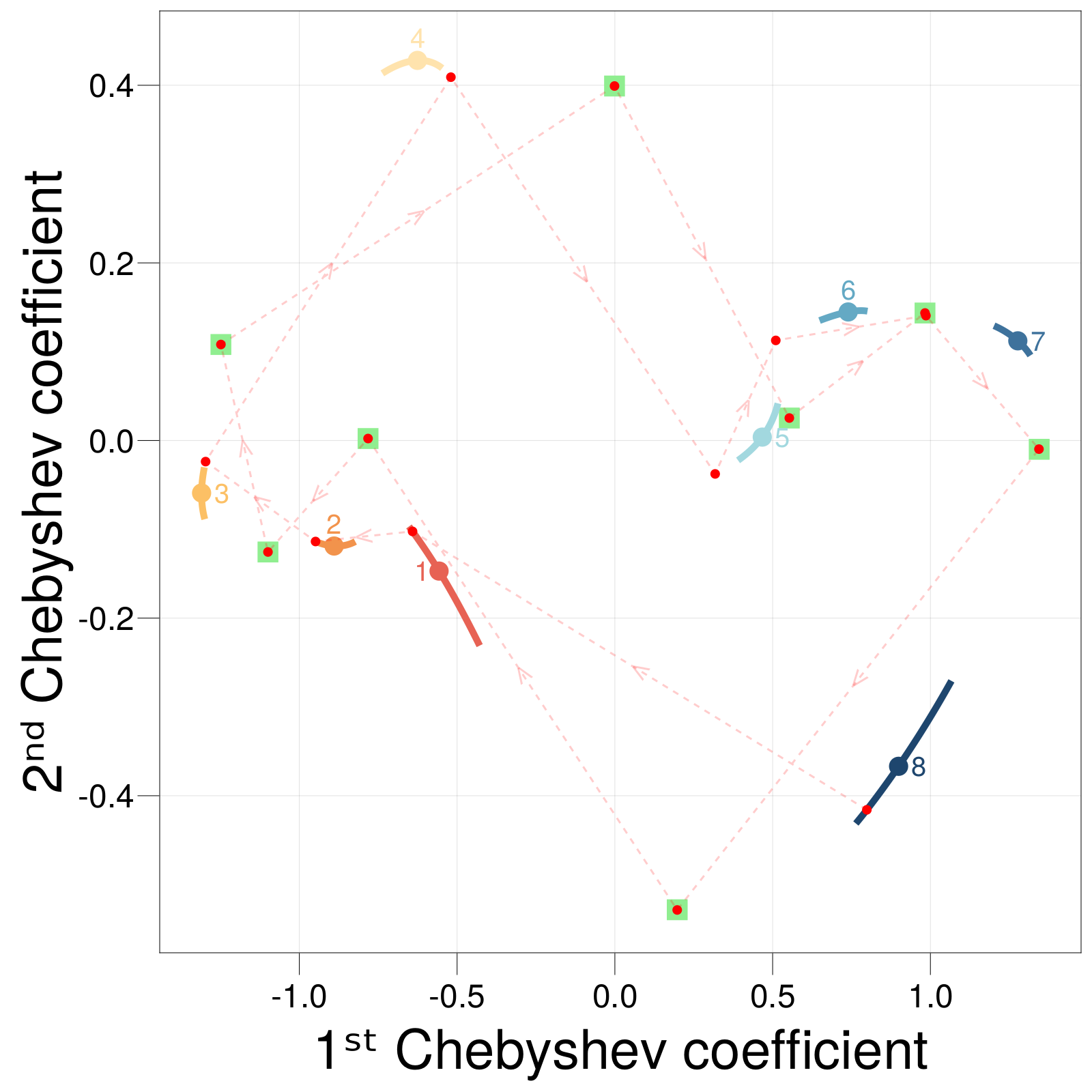}
\caption{}
\end{subfigure}
\hspace{0.03\textwidth}
\begin{subfigure}[b]{0.29\textwidth}
\includegraphics[width=\textwidth]{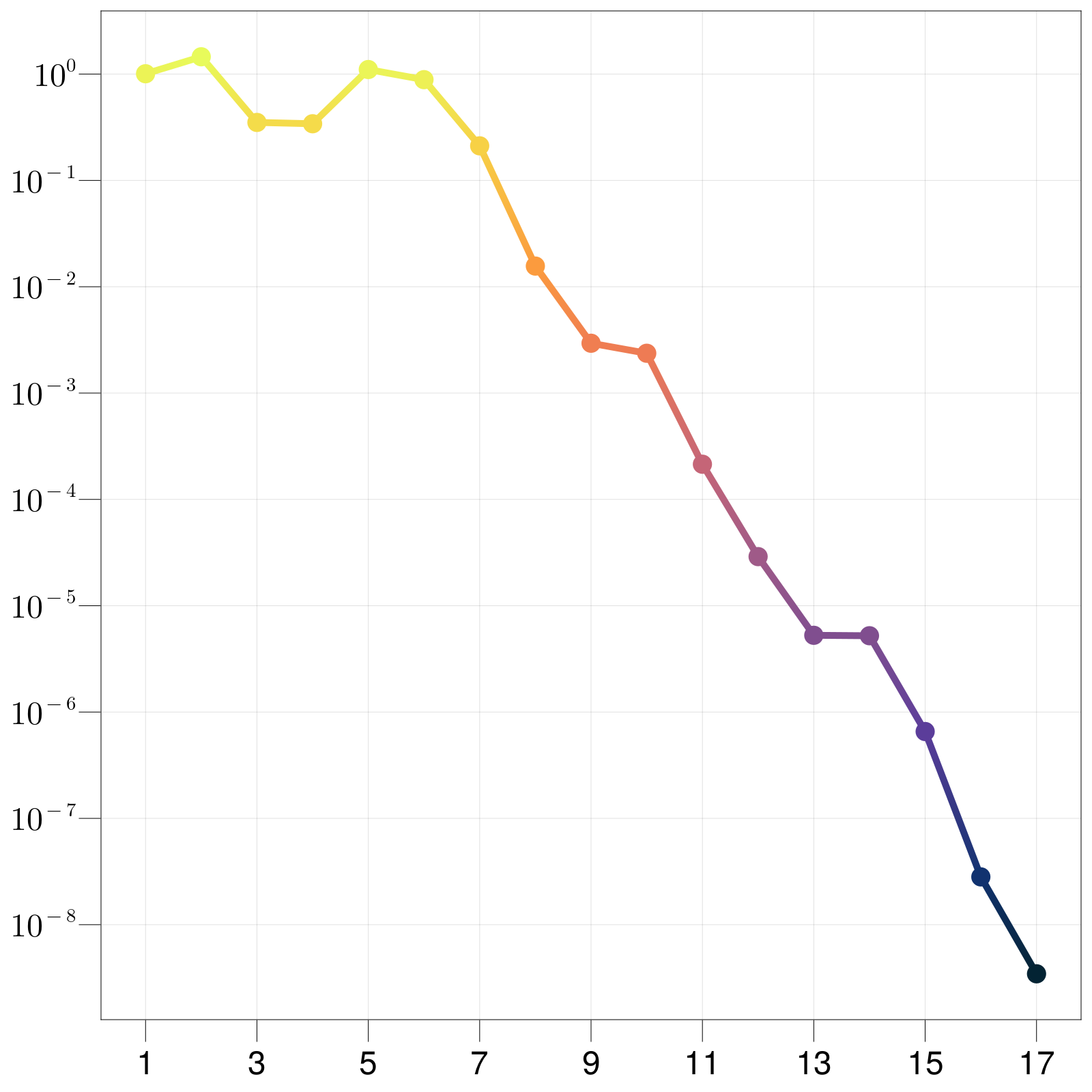}
\caption{}
\end{subfigure}
\caption{(a) Local unstable manifold of the approximate $m\tau$-periodic orbit shown in Figure \ref{fig:cubic_ikeda_param} and the transverse homoclinic orbit (red line) for the cubic Ikeda equation.
(b) Representation in \emph{Chebyshev space} of the parameterization of the local unstable manifold and the transverse homoclinic orbit (red dots) shown in (a); the connecting orbit starts at the red dot on the edge of the unstable manifold labelled 8.
The green squares represent a phase shift of the $m\tau$-periodic orbit shown in Figure \ref{fig:cubic_ikeda_po} whose local stable graph is eventually reached by the connecting orbit.
The red dashed line is meant for the reader to track the successive iterates of the time-$\tau$ map constituting the connecting orbit.
(c) Evolution of the distance of the connecting orbit as a function of the successive iterates.}
\label{fig:cubic_ikeda_homoclinic}
\end{figure}

The Newton iterations for the zero-finding problem $\mathbf{F}_\pitchfork$ are set on $\mathbb{C} \times \mathbb{C}^{n_\textnormal{u}} \times \boldsymbol{\pi}^{N,N'} (\ell^1(\ell^1_\nu))^m \times \boldsymbol{\pi}^N (\ell^1) \times \boldsymbol{\pi}^N (\ell^1_\nu)^k \simeq \mathbb{C} \times \mathbb{C}^{n_\textnormal{u}} \times \mathbb{C}^{m (N+1) (N'+1)} \times \mathbb{C}^{N+1} \times \mathbb{C}^{k (N+1)} \simeq \mathbb{C}^{1+1+8\times31\times16+31+16\times31} = \mathbb{C}^{4,497}$.
Figure \ref{fig:cubic_ikeda_homoclinic} shows the transverse homoclinic orbit.
The distance to the \emph{return periodic orbit} after convergence of Newton's method is of order $\sim 10^{-9}$ which is below our threshold $\sim 10^{-8}$.
In this computation, the two values of the delay are identical: $\bar{\tau}_\pitchfork = \bar{\tau}$.

%%%%%%%%%%%%%%%%%%
%% APPLICATIONS %%
%%%%%%%%%%%%%%%%%%

\section{Poincar\'{e} scenario for the Mackey-Glass equation}\label{sec:mackey_glass}
%!TEX root = poincare_scenario_dde.tex

In this section, we detail the computation of a transverse homoclinic orbit for the Mackey-Glass equation \eqref{eq:mg} as described in Section \ref{sec:homoclinic_orbit_num}; the code can be found at \cite{Henot2023}.
The steps are similar to the ones for cubic Ikeda equation (cf. sections \ref{sec:cubic_ikeda_po}, \ref{sec:cubic_ikeda_eigendecomposition}, \ref{sec:cubic_ikeda_param}, \ref{sec:cubic_ikeda_homoclinic}), with the notable difference that the Mackey-Glass equations has non-polynomial elementary nonlinearities, so the DDE \eqref{eq:dde} is an auxiliary polynomial DDE where $f$ is given in \eqref{eq:f_mg}; in particular, $d = 3$.

First, we fix the physiological parameters $a,b,\rho$ to $a = 2b = 1$ and $\rho = 9.65$. We also fix a value of the delay $\tau$ where chaos is numerically observed: $\tau \approx 1.82$.

\begin{figure}
\centering
\begin{subfigure}[b]{0.41\textwidth}
\includegraphics[width=\textwidth]{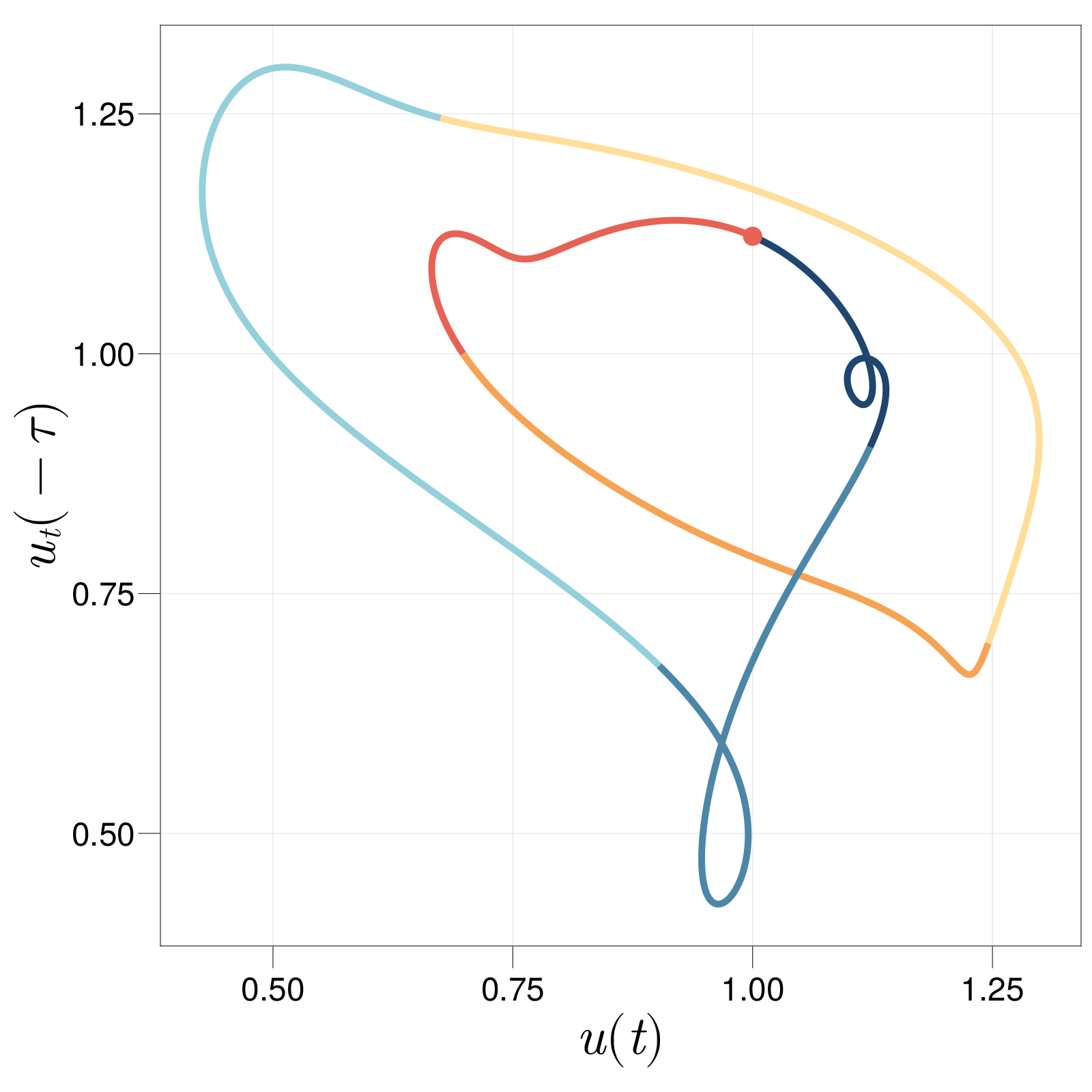}
\caption{}
\end{subfigure}
\hspace{0.06\textwidth}
\begin{subfigure}[b]{0.41\textwidth}
\includegraphics[width=\textwidth]{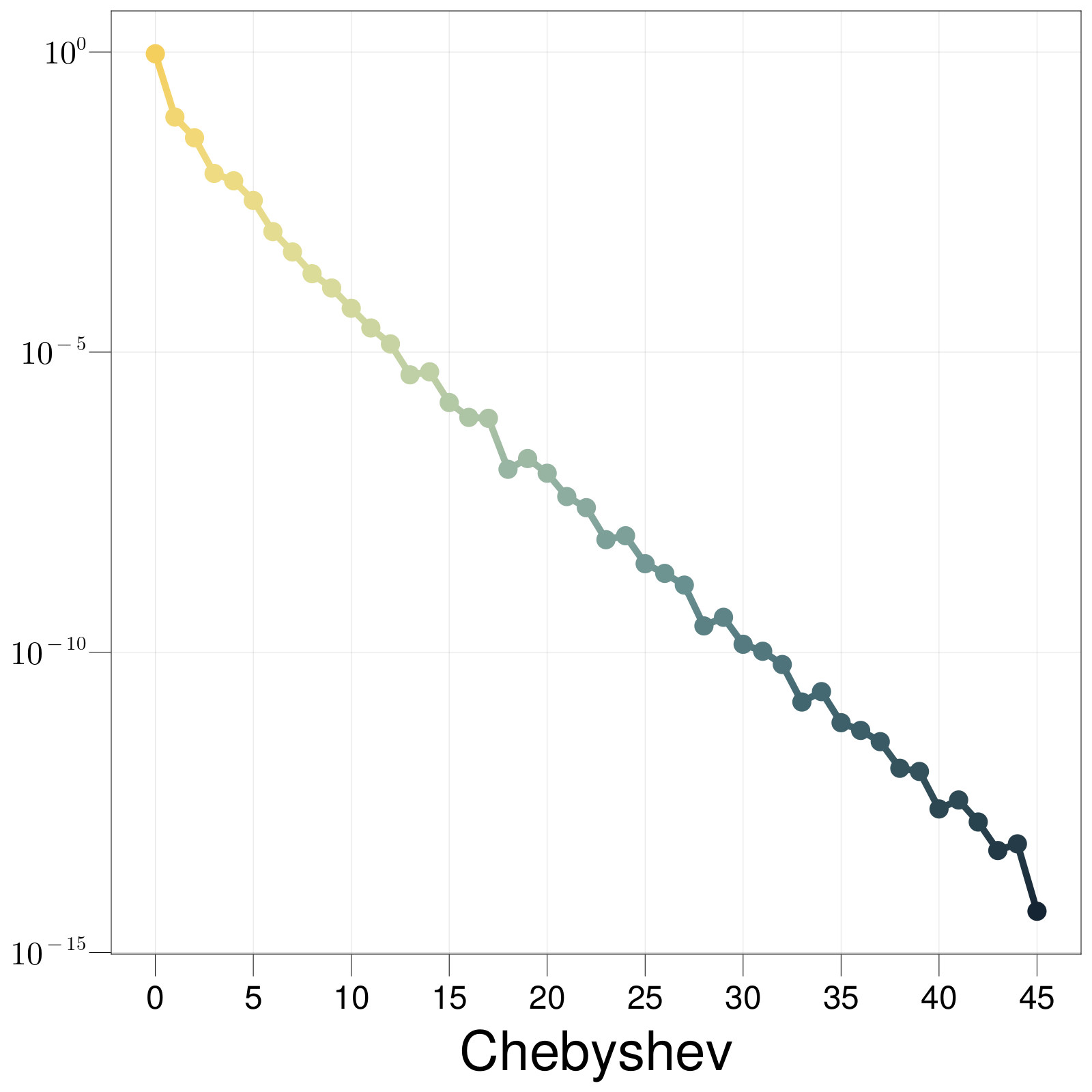}
\caption{}
\end{subfigure}
\caption{(a) $m\tau$-periodic orbit, with $m=6$, for the Mackey-Glass equation. The dot corresponds to the phase $\delta = 1$ of the periodic orbit.
(b) Average $\{ m^{-1}\sum_{j=1}^m |(\{ \bar{\mathbf{c}}_j \}_\alpha)_1| \}_{\alpha \ge 0}$ of the sequences of Chebyshev coefficients of the $m\tau$-periodic orbit shown in (a).}
\label{fig:mackey_glass_po}
\end{figure}

We identify a time series of a $m\tau$-periodic orbit with $m = 6$ and a phase $\delta = 1$.
We choose the truncation order $N = 45$ for the Chebyshev series and search for an approximate zero of $\mathbf{F}_{\circ, \textnormal{elem}}$ given in \eqref{eq:zero_finding_problem_fp_poly}.
Therefore, the Newton iterations for $\boldsymbol{\pi}^N \mathbf{F}_{\circ, \textnormal{elem}} \boldsymbol{\pi}^N$ are set on $\mathbb{R} \times \mathbb{R}^d \times \boldsymbol{\pi}^N ((\ell^1_\nu \cap \mathbb{R}^{\mathbb{N}\cup\{0\}} )^{1+d})^m \simeq \mathbb{R} \times \mathbb{R}^d \times \mathbb{R}^{m(1+d)(N+1)} \simeq \mathbb{R}^{1 + 3 + 6\times 4 \times 46} = \mathbb{R}^{1,108}$.
Performing Newton's iterations yields $\bar{\tau}_\textnormal{init} \approx 1.827334864516779$ and the sequences of Chebyshev coefficients $\bar{\mathbf{c}}_\textnormal{init} = ((\bar{\mathbf{c}}_\textnormal{init})_1, \dots, (\bar{\mathbf{c}}_\textnormal{init})_m) \in \boldsymbol{\pi}^N ((\ell^1_\nu)^{1+d})^m$.
Figure \ref{fig:mackey_glass_po} shows the approximate $m\tau$-periodic orbit and the average of the sequences of Chebyshev coefficients.

\begin{figure}
\centering
\includegraphics[height=5cm]{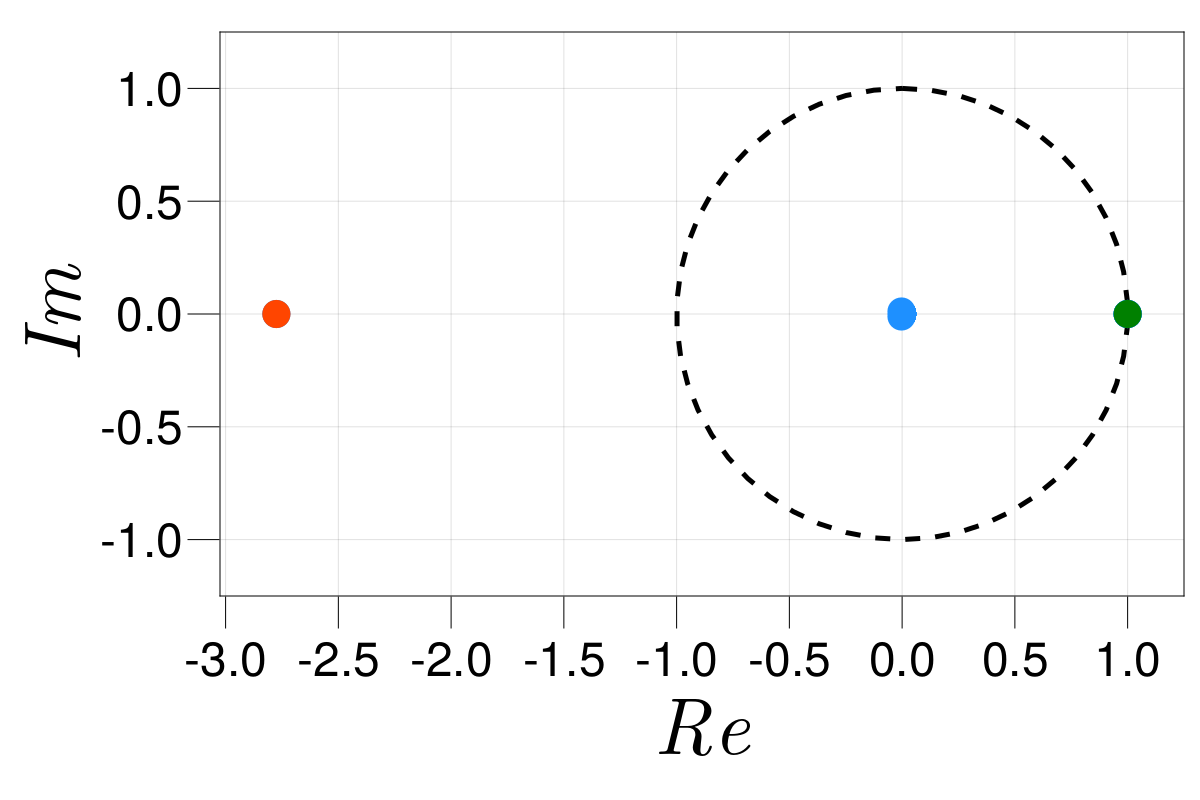}
\vspace{-.2cm}
\caption{Collection of $(1+d)(N+1) = 184$ Floquet multipliers associated with the $m\tau$-periodic orbit shown on Figure \ref{fig:mackey_glass_po} for the Mackey-Glass equation.
The black dashed circle is the unit circle.
There is $1$ unstable eigenvalue (red dot), $1+d = 4$ centre eigenvalue (green dot) and $179$ stable eigenvalues (blue dots).
Due to the proximity of the stable eigenvalues, only a single blue dot appears on the figure.}
\label{fig:mackey_glass_spectrum}
\end{figure}

We numerically retrieve the spectrum of $\mathbf{H}^N_1 (\bar{\tau}_\textnormal{init}, \bar{\mathbf{c}}_\textnormal{init})$ given in \eqref{eq:def_H_num}.
The numerical spectrum consists of $(1+d)(N+1) = 184$ eigenvalues; Figure \ref{fig:mackey_glass_spectrum} suggests that the periodic orbit has a single unstable Floquet multiplier $\bar{\mu} \approx -2.7747991365286633$.

\begin{figure}
\centering
\begin{subfigure}[b]{0.31\textwidth}
\includegraphics[width=\textwidth]{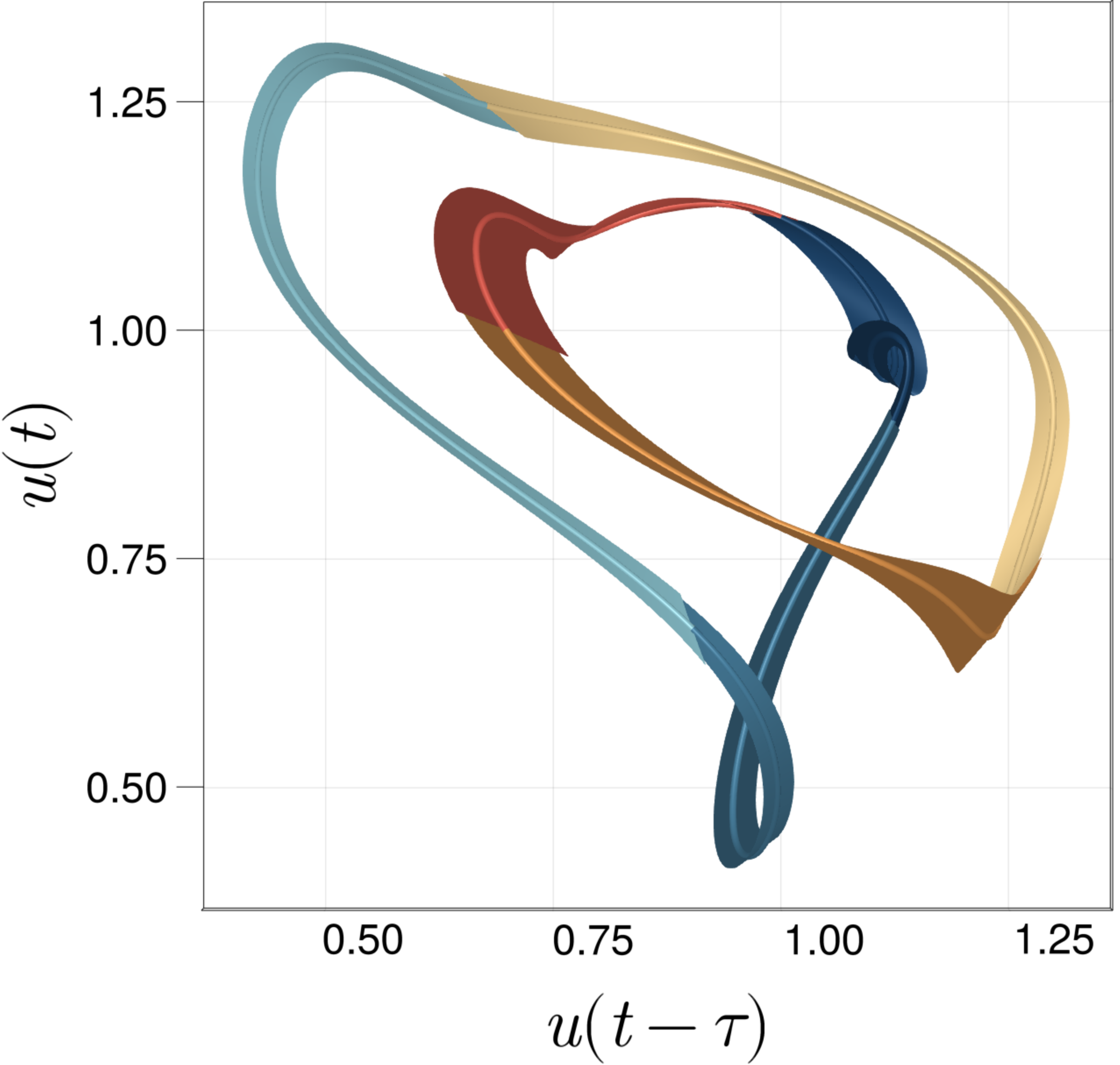}
\caption{}
\end{subfigure}
\hspace{0.03\textwidth}
\begin{subfigure}[b]{0.29\textwidth}
\includegraphics[width=\textwidth]{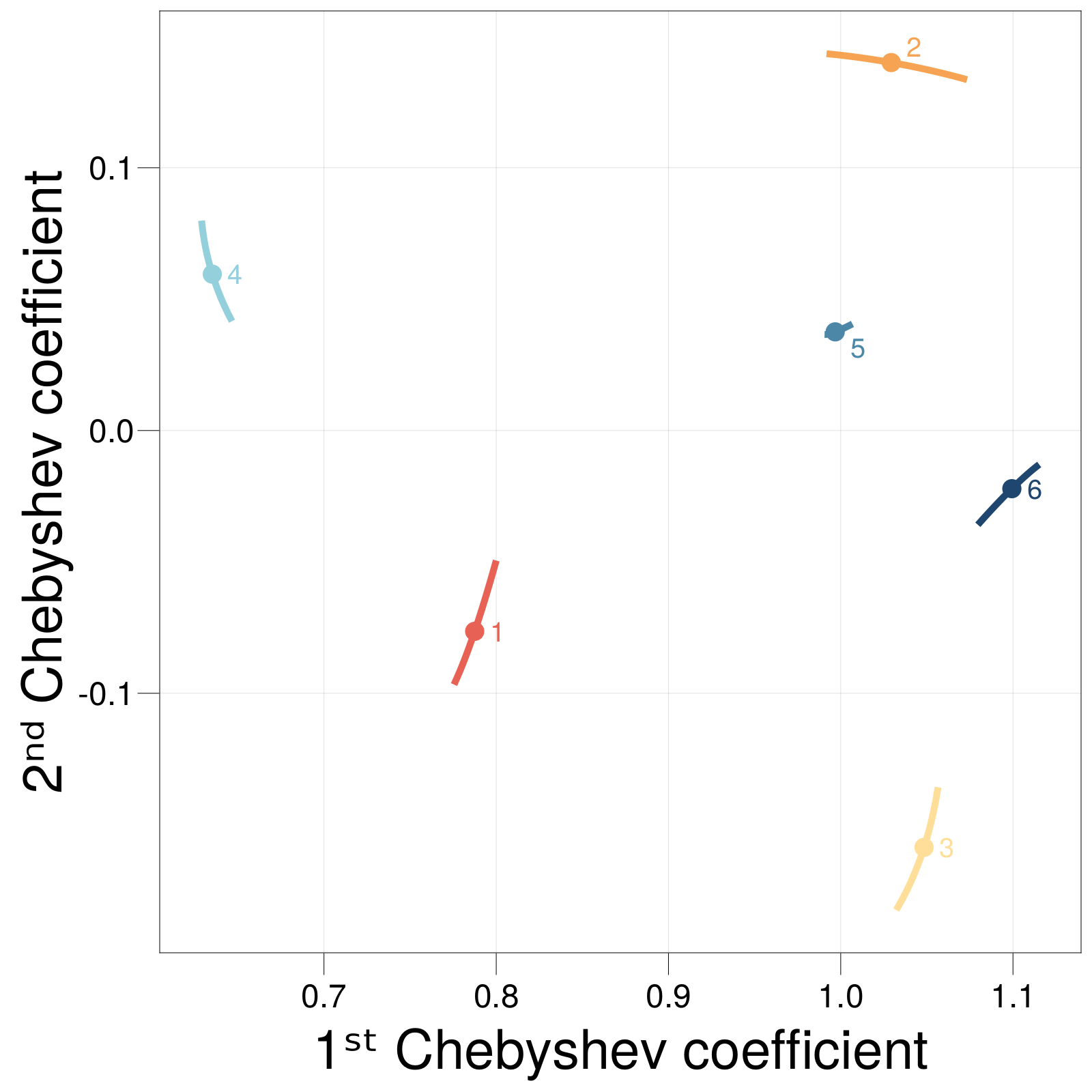}
\caption{}
\end{subfigure}
\hspace{0.03\textwidth}
\begin{subfigure}[b]{0.3\textwidth}
\includegraphics[width=\textwidth]{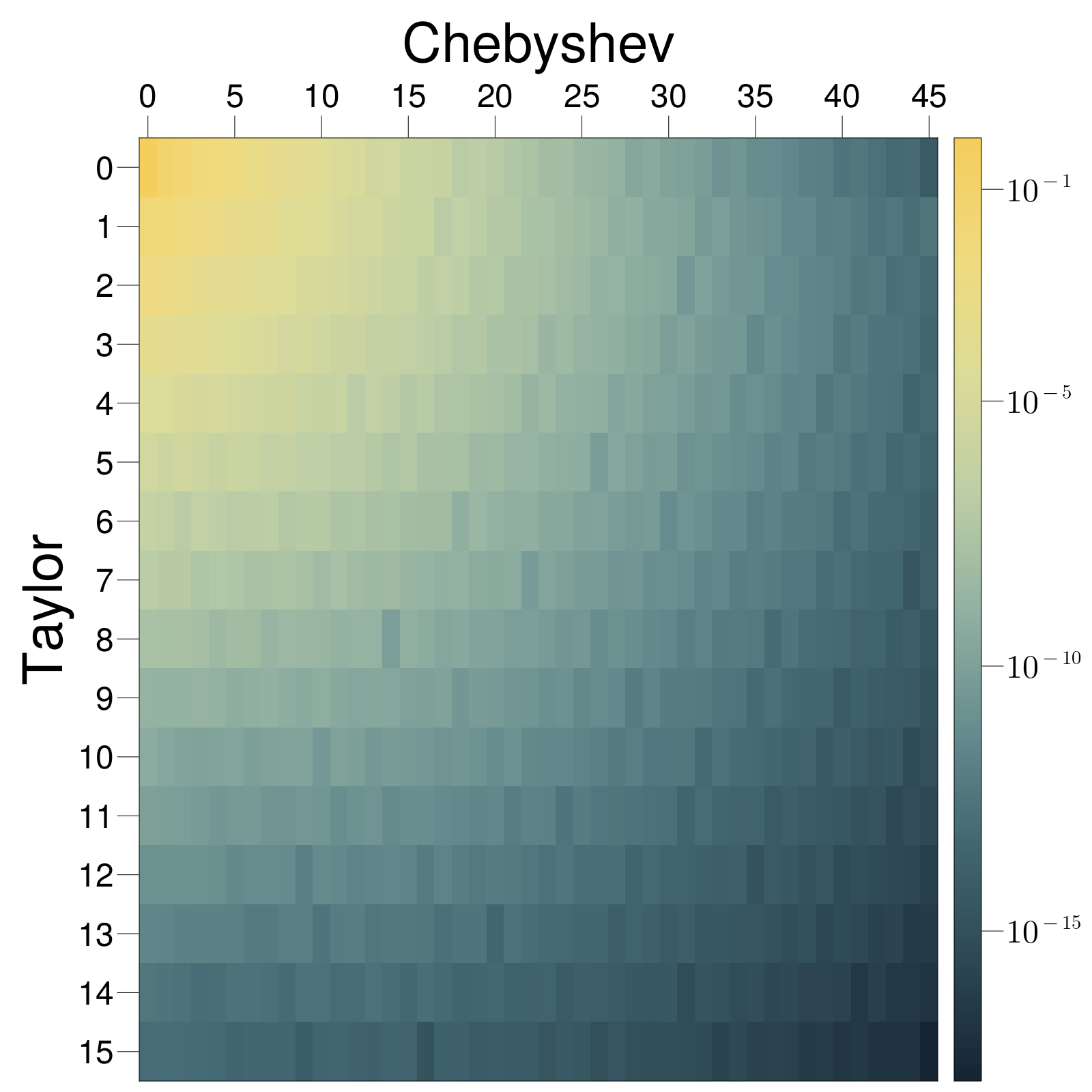}
\caption{}
\end{subfigure}
\caption{(a) Local unstable manifold of the $m\tau$-periodic orbit shown on Figure \ref{fig:mackey_glass_po} for the Mackey-Glass equation.
(b) Representation in \emph{Chebyshev space} of the parameterization of the local unstable manifold shown in (a). The dots correspond to the periodic orbit.
The numbers indicate the labelling of the $m$ pieces; the numbering follows the successive iterations of the time-$\tau$ map.
(c) Average $\{ m^{-1}\sum_{j=1}^m |(\{ \bar{\mathbf{p}}_j \}_{\alpha,\beta})_1| \}_{\alpha,\beta \ge 0}$ of the sequences of Taylor-Chebyshev coefficients of the parameterization of the local unstable manifold shown in (a).}
\label{fig:mackey_glass_param}
\end{figure}

Thus, the unstable manifold is expected to be $1$-dimensional (i.e. $n_\textnormal{u}=1$).
Since $\bar{\mu} < -1$, the unstable manifold is a topological Möbius strip (see also Remark \ref{rem:param_image}).
An approximation of the parameterization of the local unstable manifold is obtained via the recurrence relation \eqref{eq:recurrence_relation_num}.
Each linear system is set on $\boldsymbol{\pi}^N ((\ell^1_\nu)^{1+d})^m \simeq \mathbb{C}^{m(1+d)(N+1)} = \mathbb{C}^{6\times 4\times46} = \mathbb{C}^{1,104}$.
We choose the Taylor truncation order to be $N' = 15$, thus the parameterization has a total of $m (1+d) (N+1) (N'+1) = 6 \times 4 \times 46 \times 16 = 17,664$ Taylor-Chebyshev coefficients; see Figure \ref{fig:mackey_glass_param}.

\begin{figure}
\centering
\begin{subfigure}[b]{0.31\textwidth}
\includegraphics[width=\textwidth]{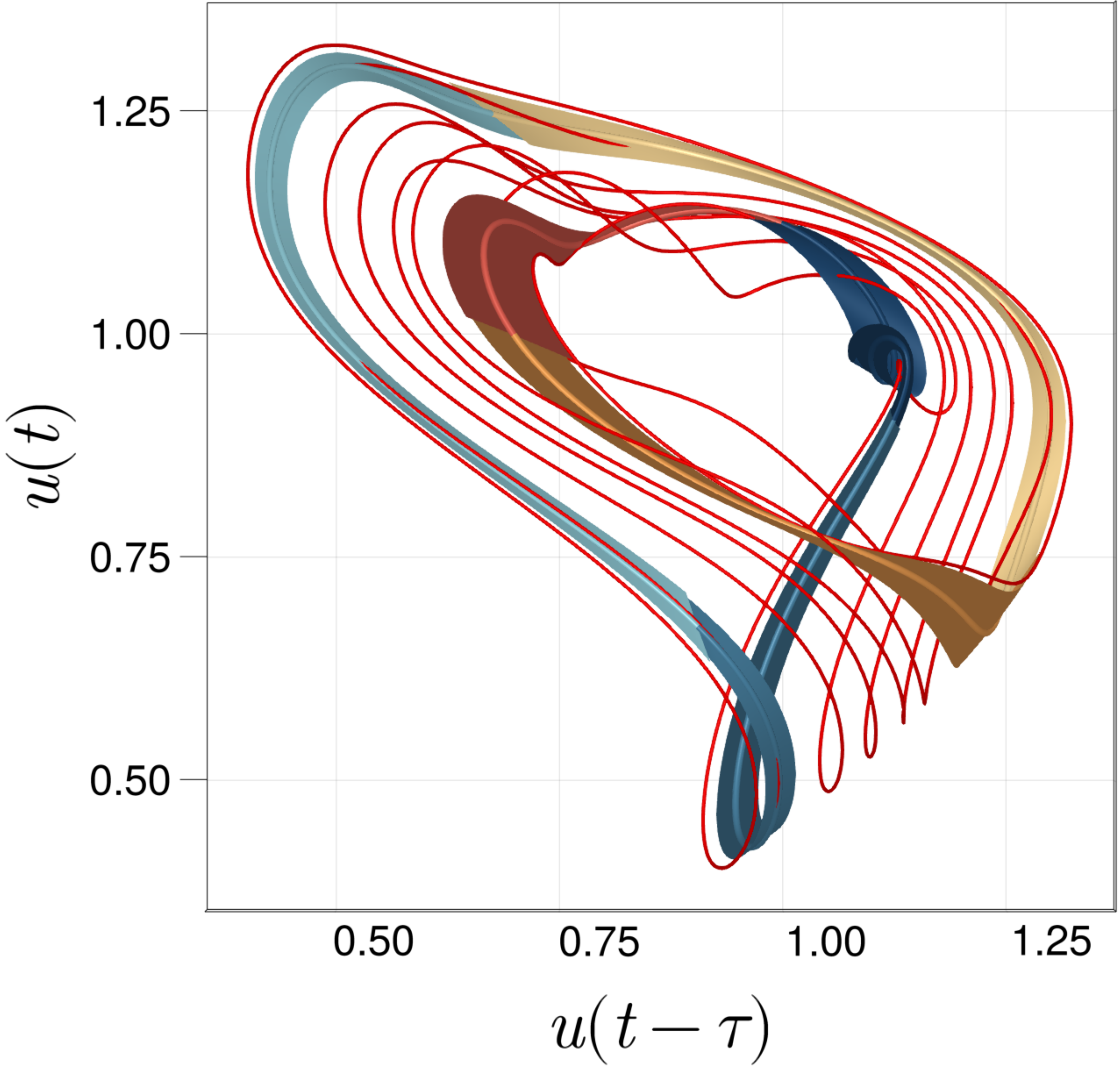}
\caption{}
\end{subfigure}
\hspace{0.03\textwidth}
\begin{subfigure}[b]{0.3\textwidth}
\includegraphics[width=\textwidth]{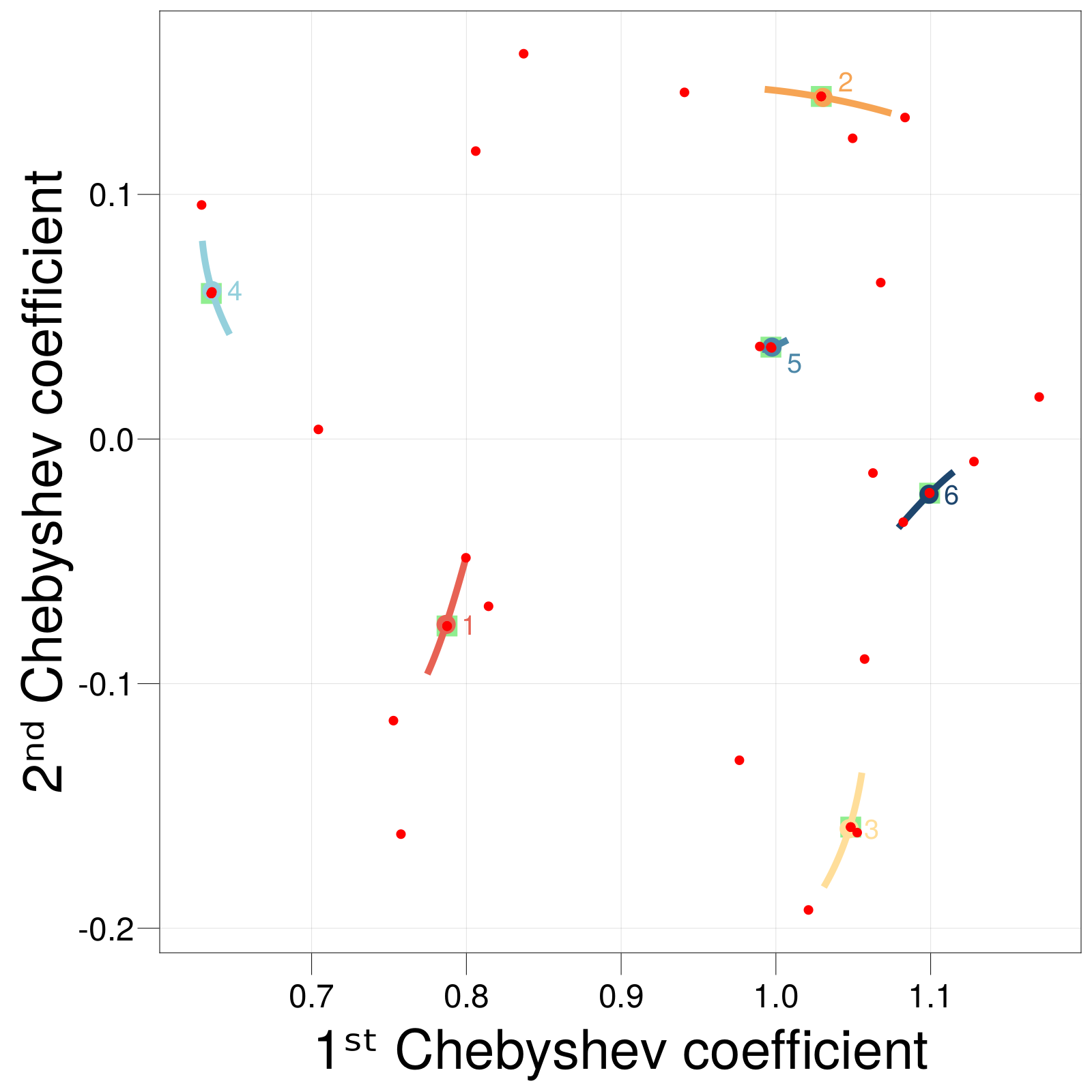}
\caption{}
\end{subfigure}
\hspace{0.03\textwidth}
\begin{subfigure}[b]{0.3\textwidth}
\includegraphics[width=\textwidth]{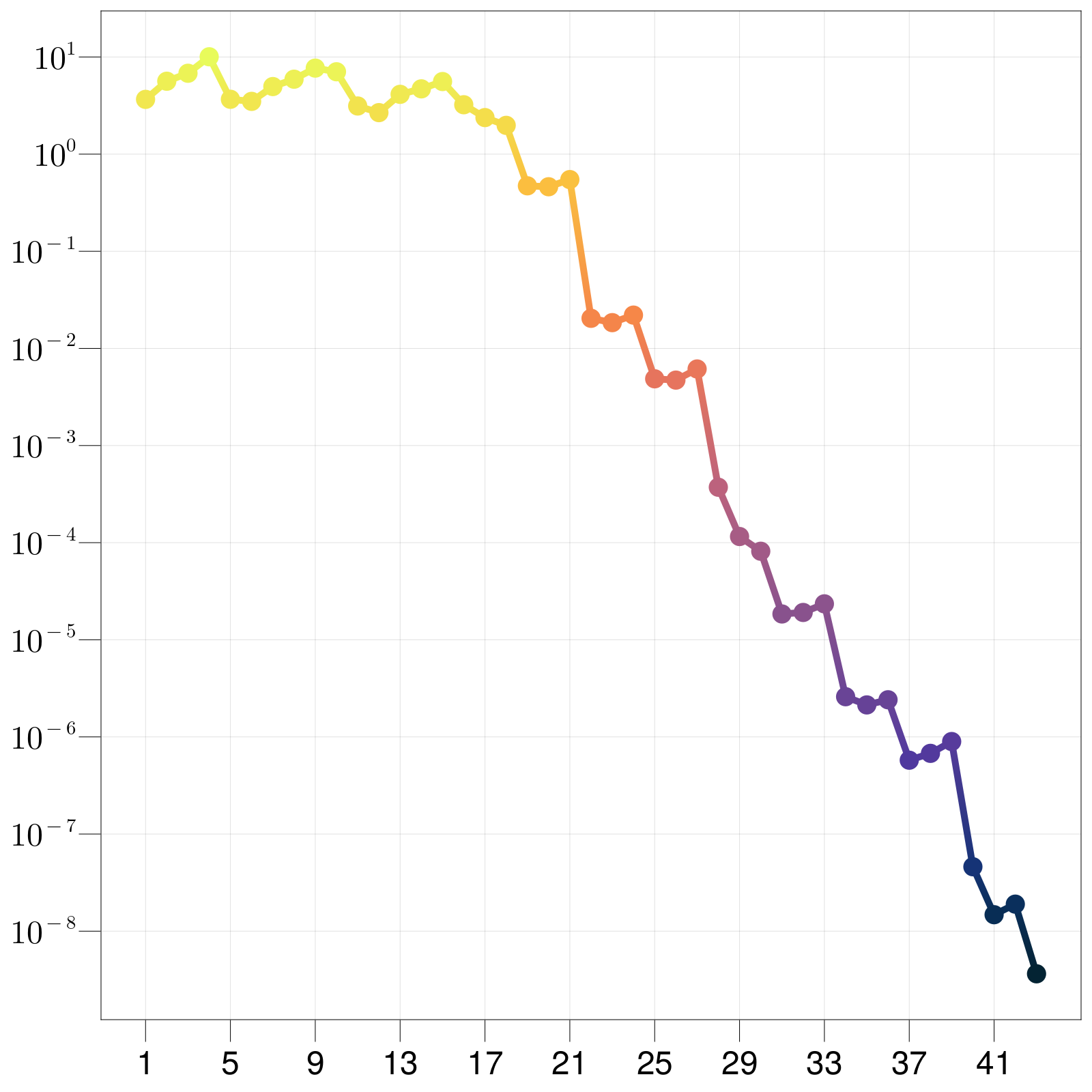}
\caption{}
\end{subfigure}
\caption{(a) Local unstable manifold of the $m\tau$-periodic orbit shown on Figure \ref{fig:mackey_glass_param} and the transverse homoclinic orbit (red line) for the Mackey-Glass equation.
(b) Representation in \emph{Chebyshev space} of the parameterization of the local unstable manifold and the transverse homoclinic orbit (red dots) shown in (a); the connecting orbit starts at the red dot on the edge of the unstable manifold labelled 6.
The green squares represent a phase shift of the $m\tau$-periodic orbit shown on Figure \ref{fig:mackey_glass_po} whose local stable graph is eventually reached by the connecting orbit.
(c) Evolution of the distance of the connecting orbit as a function of the successive iterates.}
\label{fig:mackey_glass_homoclinic}
\end{figure}

By growing the boundary of the local unstable manifold, we find, after $43$ iterations of the time-$\tau$ map, a connecting orbit about $\sim 10^{-5}$ close to the $1$-st piece of a \emph{return periodic orbit} whose phase is, roughly, $1.001792666495276$; hence, we set $j_* = 1$ and $k = 42$ for $\mathbf{F}_\textnormal{co}$ given in \eqref{eq:def_F_co}.
In fact, since in this case the phase of the \emph{return periodic orbit} is close to the initial phase $\delta = 1$, it is not necessary to solve again the zero-finding problem $\mathbf{F}_{\circ, \textnormal{elem}}$ for the \emph{return periodic orbit}. We simply set $\bar{\tau} = \bar{\tau}_\textnormal{init}$ and $\bar{\mathbf{c}} = \bar{\mathbf{c}}_\textnormal{init}$.
Then, we use $\mathbf{H}^N_{j^*} (\bar{\tau}, \bar{\mathbf{c}})$ to get an approximation of the stable eigenspace.

The Newton iterations for the zero-finding problem $\mathbf{F}_{\pitchfork, \textnormal{elem}}$, given in \eqref{eq:zero_finding_problem_connection_poly}, are set on $\mathbb{C} \times \mathbb{C}^d \times \mathbb{C}^{n_\textnormal{u}} \times \boldsymbol{\pi}^{N,N'}((\ell^1(\ell^1_\nu))^{1+d})^m \times \boldsymbol{\pi}^N (\ell^1)^{1+d} \times \boldsymbol{\pi}^N ((\ell^1_\nu)^{1+d})^k \simeq \mathbb{C} \times \mathbb{C}^d \times \mathbb{C}^{n_\textnormal{u}} \times \mathbb{C}^{m (1+d) (N+1) (N'+1)} \times \mathbb{C}^{(1+d) (N+1)} \times \mathbb{C}^{k (1+d) (N+1)} \simeq \mathbb{C}^{25,581}$.
Figure \ref{fig:mackey_glass_homoclinic} shows the transverse homoclinic orbit.
The distance to the \emph{return periodic orbit} after convergence of Newton's method is of order $\sim 10^{-9}$ which is below our threshold $\sim 10^{-8}$.
The gap between the two values of the delay is $|\bar{\tau}_\pitchfork - \bar{\tau}| \approx 1.588373876870719 \times 10^{-11}$.

Note that we ran all the computations in double precision. While Newton's method converged and gave us two approximate zeros for $\mathbf{F}_{\circ, \textnormal{elem}}$ and $\mathbf{F}_{\pitchfork, \textnormal{elem}}$ of order $\sim 10^{-15}$, there is a contribution due to the unfolding parameters $\eta$.
Forcing these to be zero gives us an approximate zero of order $\sim 10^{-13}$.
One could resort to multi-precision to improve this value which would also reduce the gap $|\bar{\tau}_\pitchfork - \bar{\tau}|$.

%%%%%%%%%%%%%%%%
%% REFERENCES %%
%%%%%%%%%%%%%%%%

\bibliographystyle{abbrv}
\bibliography{references}

\begin{thebibliography}{10}

\bibitem{Bezanson2017}
J.~Bezanson, A.~Edelman, S.~Karpinski, and V.~B. Shah.
\newblock Julia: a fresh approach to numerical computing.
\newblock {\em SIAM Review}, 59:65--98, 1 2017.

\bibitem{Bucker2006}
H.~M. B{\"u}cker and G.~F. Corliss.
\newblock A bibliography of automatic differentiation.
\newblock In {\em Automatic differentiation: applications, theory, and implementations}, volume~50 of {\em Lect. Notes Comput. Sci. Eng.}, pages 321--322. Springer Berlin, 2006.

\bibitem{Cabre2003}
X.~Cabr{\'e}, E.~Fontich, and R.~de~la Llave.
\newblock The parameterization method for invariant manifolds {I}: manifolds associated to non-resonant subspaces.
\newblock {\em Indiana University Mathematics Journal}, 52:283--328, 2003.

\bibitem{Cabre2003-2}
X.~Cabr{\'e}, E.~Fontich, and R.~de~la Llave.
\newblock The parameterization method for invariant manifolds {II}: regularity with respect to parameters.
\newblock {\em Indiana University Mathematics Journal}, 52:329--360, 2003.

\bibitem{Cabre2005}
X.~Cabr{\'e}, E.~Fontich, and R.~de~la Llave.
\newblock The parameterization method for invariant manifolds {III}: overview and applications.
\newblock {\em Journal of Differential Equations}, 218:444--515, 11 2005.

\bibitem{MR2531684}
I.~Charpentier, A.~Lejeune, and M.~Potier-Ferry.
\newblock The {Diamant} approach for an efficient automatic differentiation of the asymptotic numerical method.
\newblock In {\em Advances in automatic differentiation}, volume~64 of {\em Lect. Notes Comput. Sci. Eng.}, pages 139--149. Springer, Berlin, 2008.

\bibitem{Constantine1996}
G.~Constantine and T.~Savits.
\newblock A multivariate {Faa} di {Bruno} formula with applications.
\newblock {\em Transactions of the American Mathematical Society}, 348:503--520, 1996.

\bibitem{Danisch2021}
S.~Danisch and J.~Krumbiegel.
\newblock Makie.jl: flexible high-performance data visualization for {Julia}.
\newblock {\em Journal of Open Source Software}, 6:3349, 9 2021.

\bibitem{Llave2016}
R.~de~la Llave and J.~D. {Mireles James}.
\newblock Connecting orbits for compact infinite dimensional maps: computer assisted proofs of existence.
\newblock {\em SIAM Journal on Applied Dynamical Systems}, 15:1268--1323, 1 2016.

\bibitem{Diekmann1995}
O.~Diekmann, S.~M.~V. Lunel, S.~A. van Gils, and H.-O. Walther.
\newblock {\em Delay Equations}, volume 110.
\newblock Springer New York, 1995.

\bibitem{MR1345150}
O.~Diekmann, S.~A. van Gils, S.~M.~V. Lunel, and H.-O. Walther.
\newblock {\em Delay equations}, volume 110 of {\em Applied Mathematical Sciences}.
\newblock Springer-Verlag, New York, 1995.
\newblock Functional, complex, and nonlinear analysis.

\bibitem{Farmer1982}
J.~D. Farmer.
\newblock Chaotic attractors of an infinite-dimensional dynamical system.
\newblock {\em Physica D: Nonlinear Phenomena}, 4:366--393, 3 1982.

\bibitem{MR3706909}
J.~L. Gonzalez and J.~D. {Mireles James}.
\newblock High-order parameterization of stable/unstable manifolds for long periodic orbits of maps.
\newblock {\em SIAM J. Appl. Dyn. Syst.}, 16(3):1748--1795, 2017.

\bibitem{Groothedde2017}
C.~M. Groothedde and J.~D. {Mireles James}.
\newblock Parameterization method for unstable manifolds of delay differential equations.
\newblock {\em Journal of Computational Dynamics}, 4:2--2, 9 2017.

\bibitem{MR3973675}
L.~Guillot, B.~Cochelin, and C.~Vergez.
\newblock A generic and efficient {Taylor} series-based continuation method using a quadratic recast of smooth nonlinear systems.
\newblock {\em International Journal for Numerical Methods in Engineering}, 119(4):261--280, 2019.

\bibitem{Hale1977}
J.~K. Hale.
\newblock {\em Theory of Functional Differential Equations}, volume~3.
\newblock Springer New York, 1977.

\bibitem{Hale1986}
J.~K. Hale and X.-B. Lin.
\newblock Symbolic dynamics and nonlinear semiflows.
\newblock {\em Annali di Matematica Pura ed Applicata}, 144:229--259, 12 1986.

\bibitem{MR1243878}
J.~K. Hale and S.~M.~V. Lunel.
\newblock {\em Introduction to functional-differential equations}, volume~99 of {\em Applied Mathematical Sciences}.
\newblock Springer-Verlag, New York, 1993.

\bibitem{Hale1988}
J.~K. Hale and N.~Sternberg.
\newblock Onset of chaos in differential delay equations.
\newblock {\em Journal of Computational Physics}, 77:221--239, 7 1988.

\bibitem{Henot2021}
O.~H{\'e}not.
\newblock On polynomial forms of nonlinear functional differential equations.
\newblock {\em Journal of Computational Dynamics}, 8:307, 2021.

\bibitem{Henot2021-2}
O.~H{\'e}not.
\newblock {RadiiPolynomial.jl}, 2021.
\newblock Software, https://github.com/OlivierHnt/RadiiPolynomial.jl.

\bibitem{Henot2023}
O.~H{\'e}not.
\newblock {DDEPoincareScenario.jl}, 2023.
\newblock Implementation of the Poincar{\'e} scenario, https://github.com/OlivierHnt/DDEPoincareScenario.jl.

\bibitem{Henot2022}
O.~H{\'e}not, J.-P. Lessard, and J.~D. Mireles-James.
\newblock Parameterization of unstable manifolds for {DDEs}: formal series solutions and validated error bounds.
\newblock {\em Journal of Dynamics and Differential Equations}, 34:1285--1324, 6 2022.

\bibitem{Ikeda1987}
K.~Ikeda and K.~Matsumoto.
\newblock High-dimensional chaotic behavior in systems with time-delayed feedback.
\newblock {\em Physica D: Nonlinear Phenomena}, 29:223--235, 11 1987.

\bibitem{Jorba2005}
{\`A}.~Jorba and M.~Zou.
\newblock A software package for the numerical integration of {ODEs} by means of high-order {Taylor} methods.
\newblock {\em Experimental Mathematics}, 14:99--117, 1 2005.

\bibitem{Junges2012}
L.~Junges and J.~A. Gallas.
\newblock Intricate routes to chaos in the {Mackey-Glass} delayed feedback system.
\newblock {\em Physics Letters A}, 376:2109--2116, 6 2012.

\bibitem{Knuth1981}
D.~E. Knuth.
\newblock {\em The Art of Computer Programming, Volume 2: Seminumerical Algorithms}.
\newblock Addison-Wesley Publishing Co., Reading, Mass., second edition, 1981.

\bibitem{MR1615995}
B.~Lani-Wayda.
\newblock Erratic solutions of simple delay equations.
\newblock {\em Trans. Amer. Math. Soc.}, 351(3):901--945, 1999.

\bibitem{MR1827805}
B.~Lani-Wayda.
\newblock Wandering solutions of delay equations with sine-like feedback.
\newblock {\em Memoirs of the American Mathematical Society}, 151(718):x+121, 2001.

\bibitem{MR1329849}
B.~Lani-Wayda and H.-O. Walther.
\newblock Chaotic motion generated by delayed negative feedback. {I}. {A} transversality criterion.
\newblock {\em Differential Integral Equations}, 8(6):1407--1452, 1995.

\bibitem{MR1397673}
B.~Lani-Wayda and H.-O. Walther.
\newblock Chaotic motion generated by delayed negative feedback. {II}. {Construction} of nonlinearities.
\newblock {\em Math. Nachr.}, 180:141--211, 1996.

\bibitem{Lepri1994}
S.~Lepri, G.~Giacomelli, A.~Politi, and F.~T. Arecchi.
\newblock High-dimensional chaos in delayed dynamical systems.
\newblock {\em Physica D: Nonlinear Phenomena}, 70:235--249, 1 1994.

\bibitem{Lessard2020}
J.-P. Lessard and J.~D. {Mireles James}.
\newblock A functional analytic approach to validated numerics for eigenvalues of delay equations.
\newblock {\em Journal of Computational Dynamics}, 7:123--158, 2020.

\bibitem{Lessard2021}
J.-P. Lessard and J.~D. {Mireles James}.
\newblock A rigorous implicit {$C^1$} {Chebyshev} integrator for delay equations.
\newblock {\em Journal of Dynamics and Differential Equations}, 33:1959--1988, 12 2021.

\bibitem{Lessard2016}
J.-P. Lessard, J.~D. {Mireles James}, and J.~Ransford.
\newblock Automatic differentiation for {Fourier} series and the radii polynomial approach.
\newblock {\em Physica D: Nonlinear Phenomena}, 334:174--186, 11 2016.

\bibitem{Mackey1977}
M.~C. Mackey and L.~Glass.
\newblock Oscillation and chaos in physiological control systems.
\newblock {\em Science}, 197:287--289, 7 1977.

\bibitem{Mackey1987}
M.~C. Mackey and J.~G. Milton.
\newblock Dynamical diseases.
\newblock {\em Annals of the New York Academy of Sciences}, 504:16--32, 7 1987.

\bibitem{Mensour1998}
B.~Mensour and A.~Longtin.
\newblock Power spectra and dynamical invariants for delay-differential and difference equations.
\newblock {\em Physica D: Nonlinear Phenomena}, 113:1--25, 2 1998.

\bibitem{Nussbaum1973}
R.~D. Nussbaum.
\newblock Periodic solutions of analytic functional differential equations are analytic.
\newblock {\em Michigan Mathematical Journal}, 20, 11 1973.

\bibitem{Poincare}
H.~Poincar{\'e}.
\newblock Sur le probl{\`e}me des trois corps et les {\'e}quations de dynamique.
\newblock {\em Acta Mathematica}, 1:1--270, 1890.

\bibitem{PujoMenjouet2016}
L.~Pujo-Menjouet.
\newblock Blood cell dynamics: half of a century of modelling.
\newblock {\em Mathematical Modelling of Natural Phenomena}, 11:92--115, 2 2016.

\bibitem{MR228014}
S.~Smale.
\newblock Differentiable dynamical systems.
\newblock {\em Bulletin of the American Mathematical Society}, 73:747--817, 1967.

\bibitem{Souza2019}
D.~C.~D. Souza and A.~R. Humphries.
\newblock Dynamics of a mathematical hematopoietic stem-cell population model.
\newblock {\em SIAM Journal on Applied Dynamical Systems}, 18:808--852, 1 2019.

\bibitem{Sprott2007}
J.~Sprott.
\newblock A simple chaotic delay differential equation.
\newblock {\em Physics Letters A}, 366:397--402, 7 2007.

\bibitem{MR4446093}
A.~N. Timsina and J.~D. {Mireles James}.
\newblock Parameterized stable/unstable manifolds for periodic solutions of implicitly defined dynamical systems.
\newblock {\em Chaos Solitons Fractals}, 161:Paper No. 112345, 20, 2022.

\bibitem{Berg2022}
J.~B. van~den Berg, C.~Groothedde, and J.-P. Lessard.
\newblock A general method for computer-assisted proofs of periodic solutions in delay differential problems.
\newblock {\em Journal of Dynamics and Differential Equations}, 34:853--896, 6 2022.

\bibitem{MR623379}
H.-O. Walther.
\newblock Homoclinic solution and chaos in {$\dot x(t)=f(x(t-1))$}.
\newblock {\em Nonlinear Analysis: Theory, Methods \& Applications}, 5(7):775--788, 1981.

\bibitem{hans-otto-MG}
H.-O. Walther.
\newblock The impact on mathematics of the paper ``{Oscillation} and chaos in physiological control systems" by {Mackey} and {Glass} in {Science}, 1977.
\newblock arXiv:2001.09010, 2009.

\end{thebibliography}

\end{document}